\documentclass[a4paper]{article}

\usepackage{amsthm}
\usepackage{amsmath}
\usepackage{amssymb}
\usepackage{amsfonts}
\usepackage{enumerate}
\usepackage{authblk}
\usepackage{setspace}
\usepackage{fullpage}
\usepackage{hyperref}

\newcommand\CC{\mathbf{C}}
\newcommand\RR{\mathbf{R}}
\newcommand\QQ{\mathbf{Q}}
\newcommand\ZZ{\mathbf{Z}}
\newcommand\NN{\mathbf{N}}
\newcommand\id{\mathrm{id}}
\newcommand\AAA{\mathbf{A}}
\newcommand\sign{\mathrm{sign}}
\newcommand\whit{\mathcal{W}}
\newcommand\fin{\mathrm{fin}}
\newcommand\J{\mathbf{J}}

\newcommand\ww{\mathbf{w}}
\newcommand\rr{\mathbf{r}}
\newcommand\diag{\mathrm{diag}}
\newcommand\res{\mathrm{res}}
\newcommand\vol{\mathrm{vol}}
\newcommand\LL{\mathcal{L}}
\newcommand{\ppmod}{\!\!\!\!\!\pmod}

\renewcommand{\frak}{\mathfrak}
\renewcommand{\hat}{\widehat}

\newcommand\const{\mathrm{const.}}

\newcommand\cusp{\mathrm{cusp}}
\newcommand\cont{\mathrm{cont}}
\renewcommand\sp{\mathrm{sp}}
\newcommand\lfact[2]{#1\backslash #2}
\newcommand\rfact[2]{#1 / #2}
\newcommand\lrfact[3]{#1\backslash #2 / #3}

\newcommand\GLone[1]{\mathrm{GL}_{#1}}
\newcommand\GLtwo[2]{\mathrm{GL}_{#1}(#2)}

\newcommand\SLone[1]{\mathrm{SL}_{#1}}
\newcommand\SLtwo[2]{\mathrm{SL}_{#1}(#2)}
\newcommand\norm[1]{\mathcal{N}(#1)}
\newcommand\trace[1]{\mathrm{Tr}(#1)}

\newtheorem{prop}{Proposition}[section]
\newtheorem{lemm}[prop]{Lemma}
\newtheorem{theo}{Theorem}
\newtheorem{coro}[prop]{Corollary}
\newtheorem{partoftheo}{Part}

\begin{document}


\title{Shifted convolution sums and Burgess type subconvexity over number fields}
\author{P\'eter Maga\footnote{e-mail: maga.peter@renyi.mta.hu} \footnote{affil$_1$: MTA Alfr\'ed R\'enyi Institute of Mathematics, 13-15 Re\'altanoda u. 1053 Budapest, Hungary} \footnote{affil$_2$: Central European University, Dept of Mathematics and its Applications, 9 Nador u. 1051 Budapest, Hungary} \footnote{The research for this paper was a part of the author's PhD thesis at the Central European University, Budapest, Hungary.} \footnote{The author was supported by OTKA grant no. NK104183.}}
\date{}
\maketitle



\begin{abstract}
Let $F$ be a number field and $\pi$ an irreducible cuspidal representation of $\lfact{\GLtwo{2}{F}}{\GLtwo{2}{\AAA}}$ with unitary central character. Then the bound $$L(1/2,\pi\otimes\chi)\ll_{F,\pi,\chi_{\infty},\varepsilon} \norm{\frak{q}}^{3/8+\theta/4+\varepsilon}$$ holds for any Hecke character $\chi$ of conductor $\frak{q}$, where $\theta$ is any constant towards the Ramanujan-Petersson conjecture ($\theta=7/64$ is admissible). The proof is based on a spectral decomposition of shifted convolution sums.
\end{abstract}

\section{Introduction}\label{chap:introduction}

The subconvexity problem is concerned with the magnitude of an $L$-function on the critical line $\Re s=1/2$. Given a family of automorphic forms, we look for a bound of the form
\begin{equation*}
L(1/2,\pi)\ll_{\varepsilon} \mathrm{cond}(\pi)^{\delta+\varepsilon},
\end{equation*}
which holds for each member of the family. Here, $\mathrm{cond}(\pi)$ denotes the analytic conductor of $\pi$, see \cite[Section 2]{IwaniecSarnak}. With $\delta=1/4$, this is known as the convexity bound as it follows from the Phragm\'en-Lindel\"of convexity principle combined with the functional equation and some bound on the half-plane $\Re s>1$ (for the latter see \cite{Molteni}). However, for most applications, one needs a stronger estimate. Any improvement in the exponent (i.e. any $\delta<1/4$) is called a subconvex bound. The generalized Lindel\"of hypothesis predicts that even $\delta=0$ is admissible, and this would follow from the generalized Riemann hypothesis. On the other hand, several unconditional results are known. For example, for $\GLone{1}$ $L$-functions over $\QQ$ (i.e. Dirichlet $L$-functions), the famous Burgess bound \cite{Burgess} is the above with $\delta=3/16$. For automorphic $\GLone{2}$ $L$-functions over number fields the subconvexity problem was solved by Michel and Venkatesh \cite{MichelVenkatesh} with an unspecified $\delta$.

In a recent research \cite{BlomerHarcos}, Blomer and Harcos proved a Burgess type subconvex bound for twisted automorphic $\GLone{2}$ $L$-functions over totally real number fields. The method is based on the generalization of their earlier work \cite{BlomerHarcosDuke}, a spectral decomposition of shifted convolution sums.

In this paper, we work out the spectral decomposition over general number fields (which will be stated explicitly in Theorem \ref{spectraldecompositionofshiftedconvolutionsum}), and, as an application, we extend the subconvex bound of \cite{BlomerHarcos}. Assume $\pi$ is an irreducible cuspidal representation of $\lfact{\GLtwo{2}{F}}{\GLtwo{2}{\AAA}}$ with unitary central character, and let $\chi$ be a Hecke character of conductor $\frak{q}$. Moreover, let $\theta$ be a constant towards the Ramanujan-Petersson conjecture (by \cite{BlomerBrumley}, $\theta=7/64$ is admissible).
\begin{theo}\label{subconvextheorem} For any $\varepsilon>0$, we have the Burgess like subconvex bound
\begin{equation*}
L(1/2,\pi\otimes\chi)\ll_{F,\pi,\chi_{\infty},\varepsilon} \norm{\frak{q}}^{3/8+\theta/4+\varepsilon}.
\end{equation*}
\end{theo}
\noindent We remark that the same bound was simultaneously proved by Wu \cite{Wu}, using a method built on \cite{MichelVenkatesh}.

The family of twisted $L$-functions considered in Theorem \ref{subconvextheorem} was the first instance of the automorphic subconvexity problem that was studied systematically (see for example the works \cite{Iwaniechalfintegral}, \cite{Duke}, \cite{DukeFriedlanderIwaniecbounds1}, \cite{Bykovskii}, \cite{ConreyIwaniec}, \cite{CogdellPiatetskiShapiroSarnak}, \cite{BlomerHarcosMichel}, \cite{Venkateshsparse}). Via Waldspurger type formulae, critical values of twisted $L$-functions are connected to Fourier coefficients of modular forms of half-integral weight. The whole area is highly motivated by Hilbert's eleventh problem: which integers are integrally represented by a given quadratic form over a number field, and subconvex bounds for twisted $L$-functions give rise to the asymptotic of representation numbers of quadratic forms (see \cite{DukeSchulzePillot}). They also appear in the solution of other equidistribution problems (consult \cite{Cohen}, \cite{Zhang}, \cite{Venkateshsparse}). On the other hand, such subconvexity estimates are often used as ingredients in higher-rank subconvexity problems.

Shifted convolution sums have their own rich history and therefore several other applications as well. For example, in the investigation of the fourth moment of the Riemann zeta function on the critical line, off-diagonal terms led to the problem of the asymptotic behavior of the additive divisor sum (see \cite{Inghamriemannzeta}, \cite{Ingham}, \cite{Estermann}, \cite{Heath-Brown}, \cite{Motohashi}, \cite{Motohashibook}). Analogously, the second moment of an $L$-function corresponding to a $\GLone{2}$ cusp form leads to shifted convolution sums made of Hecke eigenvalues (see \cite{GoodJNT}, \cite{GoodMathAnn}). As an other application, we mention that Blomer and Harcos \cite[Theorem 3]{BlomerHarcos} gave the exact spectral decomposition (i.e. a spectral decomposition without error term) of a certain Dirichlet series with coefficients coming from a shifted convolution sum of Fourier coefficients of cuspidal representations: this problem originates from Selberg \cite{Selberg}. We hope that our generalization, Theorem \ref{spectraldecompositionofshiftedconvolutionsum} in Section \ref{chap:shiftedconvolutionsums} will lead to the more general analogs.

We introduce the notations we shall use later. We advise the reader to consult \cite{Weil} for the arising notions.

\subsection{The number field}

Throughout this paper, $X\ll_A Y$ means that $|X|\leq cY$ for some constant $c>0$ depending only on $A$. Also, $X\asymp_A Y$ means that $X\ll_A Y\ll_A X$.

Let $F$ be a number field, a finite algebraic extension of $\QQ$. Assuming $F$ has $r$ real and $s$ complex places, we will throughout denote the corresponding archimedean completions by $F_1,\ldots,F_{r+s}$, where $F_1,\ldots,F_r$ are all isomorphic to $\RR$ and $F_{r+1},\ldots,F_{r+s}$ are all isomorphic to $\CC$ as topological fields. Let $F_{\infty}$ stand for the direct sum of these fields (as rings), $F_{\infty}^{\times}$ for its multiplicative group, $F_{\infty,+}^{\times}$ for the totally positive elements (which are positive at each real place), and
$F_{\infty,+}^{\diag}$ for $\{(a_1,\ldots,a_{r+s})\in F_{\infty,+}^{\times}: a_1=\ldots=a_{r+s}\}$.

Denote by $\frak{o}$ the ring of integers of $F$. The ideals and fractional ideals will be denoted by gothic characters $\frak{a},\frak{b},\frak{c},\ldots$, the prime ideals by $\frak{p}$ and we keep $\frak{d}$ for the different and $D_F$ for the discriminant of $F$. Each prime ideal $\frak{p}$ determines a non-archimedean place and a corresponding completion $F_{\frak{p}}$. At such a place, we denote by $\frak{o}_{\frak{p}}$ the maximal compact subring.

Write $\AAA$ for the adele ring of $F$. Given an adele $a$, $a_j$ denotes its projection to $F_j$ for $1\leq j\leq r+s$, and $a_{\frak{p}}$ the same to $F_{\frak{p}}$ for a prime ideal $\frak{p}$. We will also use the subscripts $j,\frak{p}$ for the projections of other adelic objects to the place corresponding to $j,\frak{p}$, respectively. The subscripts $\infty$ and $\fin$ stand for the projections to $F_{\infty}$ and $\prod_{\frak{p}}F_{\frak{p}}$.

The absolute norm (module) of adeles will be denoted by $|\cdot|$, while $|\cdot|_j$ and $|\cdot|_{\frak{p}}$ will stand for the norm (module) at single places. Sometimes we will need $|\cdot|_{\infty}$, which is the product of the archimedean norms. (At this point, we call the reader's attention to the notational ambiguity that for a real or complex number $y$, we keep the conventional $|y|$ for its ordinary absolute value. We hope this will not lead to confusion. Note that at real places, $|y|_j=|y|$, while at complex places, $|y|_j=|y|^2$.) For a fractional ideal $\frak{a}$, $\norm{\frak{a}}$ will denote its absolute norm, defined as $\norm{\frak{a}}=|a|^{-1}$, where $a$ is any finite representing idele for $\frak{a}$. When $a$ is a finite idele, we may also write $\norm{a}$ for $|a|^{-1}$.

We define an additive character $\psi$ on $\AAA$: it is required to be trivial on $F$ (embedded diagonally); on $F_{\infty}$:
\begin{equation*}
\psi_{\infty}(x)=\exp(2\pi i \trace{x})=\exp(2\pi i(x_1+\ldots+x_r+x_{r+1}+\overline{x_{r+1}}+\ldots+x_{r+s}+\overline{x_{r+s}})); \end{equation*}
while on $F_{\frak{p}}$: it is trivial on $\frak{d}_{\frak{p}}^{-1}$ but not on $\frak{d}_{\frak{p}}^{-1}\frak{p}^{-1}$.

\subsection{Matrix groups}

Given a ring $R$, we define the following subgroups of $\GLtwo{2}{R}$:
\begin{equation*}
\begin{split}
&Z(R)=\left\{\begin{pmatrix} a & 0 \cr 0 & a\end{pmatrix}: a\in R^{\times}\right\},\\ &
B(R)=\left\{\begin{pmatrix} a & b \cr 0 & d\end{pmatrix}: a,d\in R^{\times}, b\in R\right\},\\ & 
N(R)=\left\{\begin{pmatrix} 1 & b \cr 0 & 1\end{pmatrix}: b\in R\right\}.
\end{split}
\end{equation*}

Assume $0\neq\frak{n}_{\frak{p}},\frak{c}_{\frak{p}}\subseteq \frak{o}_{\frak{p}}$. Then let
\begin{equation*}
K_{\frak{p}}(\frak{n}_{\frak{p}},\frak{c}_{\frak{p}})=\left\{\begin{pmatrix} a & b \cr c & d\end{pmatrix}: a,d\in \frak{o}_{\frak{p}}, b\in(\frak{n}_\frak{p}\frak{d}_\frak{p})^{-1}, c\in \frak{n}_\frak{p}\frak{d}_\frak{p}\frak{c}_\frak{p}, ad-bc\in\frak{o}_\frak{p}^{\times}\right\},
\end{equation*}
moreover in the special case $\frak{n}_{\frak{p}}=\frak{o}_{\frak{p}}$, we simply write $K_{\frak{p}}(\frak{c}_{\frak{p}})$ instead of $K_{\frak{p}}(\frak{o}_{\frak{p}},\frak{c}_{\frak{p}})$. For ideals $0\neq\frak{n},\frak{c}\subseteq\frak{o}$, let
\begin{equation*}
K(\frak{n},\frak{c})= \prod_{\frak{p}}K_{\frak{p}}(\frak{n}_{\frak{p}},\frak{c}_{\frak{p}}),\qquad K(\frak{c})=\prod_{\frak{p}}K_{\frak{p}}(\frak{c}_{\frak{p}}),
\end{equation*}
and taking the archimedean places into account, let
\begin{equation*}
K=K_{\infty}\times K(\frak{o})\subseteq\GLtwo{2}{\AAA},
\end{equation*}
where
\begin{equation*}
K_{\infty}=\prod_{j=1}^r\mathrm{SO}_2(\RR)\times \prod_{j=r+1}^{r+s}\mathrm{SU}_2(\CC).
\end{equation*}
Finally, for $0\neq\frak{n},\frak{c}\subseteq\frak{o}$, let
\begin{equation*}
\Gamma(\frak{n},\frak{c})=\left\{g_{\infty}\in\GLtwo{2}{F_{\infty}}: \exists g_{\fin}\in\prod_{\frak{p}}K_{\frak{p}}(\frak{n}_{\frak{p}},\frak{c}_{\frak{p}}) \mathrm{\ such\ that\ } g_{\infty}g_{\fin}\in\GLtwo{2}{F}\right\}.
\end{equation*}
We note that the choice of the subgroups $K$ is not canonical (they can be conjugated arbitrarily), our normalization follows \cite{BlomerHarcos}.

\paragraph{Archimedean matrix coefficients}

On $K_{\infty}$, we define the matrix coefficients (see \cite[p.8]{BumpLie}). Again, it is more convenient to give them on the factors. At a real place, on $\mathrm{SO}_2(\RR)$, for a given integer $q$, set
\begin{equation*}
\Phi_q\left(\begin{pmatrix} \cos\theta & \sin\theta \cr -\sin\theta & \cos\theta \end{pmatrix}\right)=\exp(i q\theta).
\end{equation*}
At a complex place, on $\mathrm{SU}_2(\CC)$, we introduce the parametrization
\begin{equation*}
\mathrm{SU}_2(\CC)=\left\{k[\alpha,\beta]=\begin{pmatrix} \alpha & \beta \cr -\overline{\beta} & \overline{\alpha}\end{pmatrix}: \alpha,\beta\in\CC,|\alpha|^2+|\beta|^2=1\right\}.
\end{equation*}
Assume now that the integers or half-integers $p,q,l$ satisfy $|p|,|q|\leq l$ and $p\equiv q\equiv l \pmod 1$. Then the matrix coefficient $\Phi_{p,q}^l$ is defined via
\begin{equation*}
\sum_{|p|\leq l} \Phi_{p,q}^l(k[\alpha,\beta])x^{l-p}=(\alpha x-\overline{\beta})^{l-q} (\beta x+\overline{\alpha})^{l+q},
\end{equation*}
where this equation is understood in the polynomial ring $\CC[x]$. See \cite[(3.18)]{BruggemanMotohashi} and \cite[(2.28)]{Lokvenec}. Note that
\begin{equation*}
||\Phi_{p,q}^l||_{\mathrm{SU}_2(\CC)}= \left(\int_{\mathrm{SU}_2(\CC)}|\Phi_{p,q}^l(k)|^2dk\right)^{1/2}= \frac{1}{\sqrt{2l+1}} {{2l}\choose{l-p}}^{1/2} {{2l}\choose{l-q}}^{-1/2}
\end{equation*}
by \cite[(2.35)]{Lokvenec}, where the Haar measure on $\mathrm{SU}_2(\CC)$ is the probability measure.

\subsection{Measures}

On $F_{\frak{p}}$, we normalize the Haar measure such that $\frak{o}_{\frak{p}}$ has measure $1$. On $F_{\infty}$, we use the Haar measure $|D_F|^{-1/2}dx_1\cdots dx_r|dx_{r+1}\wedge d\overline{x_{r+1}}|\cdots|dx_{r+s}\wedge d\overline{x_{r+s}}|$. On $\AAA$, we use the Haar measure $dx$, the product of these measures, this induces a Haar probability measure on $\lfact{F}{\AAA}$ (see \cite[Chapter V, Proposition 7]{Weil}).

On $\RR^{\times}$, we use the Haar measure $d_{\RR}^{\times}y=dy/|y|$, this gives rise to a Haar measure on $\CC^{\times}$ as $d_{\CC}^{\times}y=d_{\RR}^{\times}|y|d\theta/2\pi$, where $\exp(i\theta)=y/|y|$. On $F_{\infty}^{\times}$, we use the product $d_{\infty}^{\times}y$ of these measures. On $F_{\frak{p}}^{\times}$, we normalize the Haar measure such that $\frak{o}_{\frak{p}}^{\times}$ has measure $1$. The product $d^{\times}y$ of these measures is a Haar measure on $\AAA^{\times}$, inducing some Haar measure on $\lfact{F^{\times}}{\AAA^{\times}}$.

On $K$ and its factors, we use the Haar probability measures. On $\lfact{Z(F_{\infty})}{\GLtwo{2}{F_{\infty}}}$, we use the Haar measure which satisfies
\begin{equation*}
\int_{\lfact{Z(F_{\infty})}{\GLtwo{2}{F_{\infty}}}}f(g)dg= \int_{(\RR^{\times})^r\times(\RR_+^{\times})^s} \int_{F_{\infty}} \int_{K_{\infty}}f\left(\begin{pmatrix}y & x \cr 0 & 1 \end{pmatrix}k\right)dkdx\frac{d_{\infty}^{\times}y}{\prod_{j=1}^{r+s}|y_j|}.
\end{equation*}
Recalling $|y|_{\infty}=\prod_{j=1}^r|y_j|\prod_{j=r+1}^{r+s}|y_j|^2$, it follows that on $F_{\infty}^{\times}$, $d_{\infty}^{\times}y=\const dy/|y|_{\infty}$.

On $\GLtwo{2}{F_{\frak{p}}}$ we normalize the Haar measure such that $K(\frak{o}_{\frak{p}})$ has measure $1$. On the factor space $\lfact{Z(F_{\infty})}{\GLtwo{2}{\AAA}}$, we use the product of these measures, which, on the factor $\lfact{Z(\AAA)}{\GLtwo{2}{\AAA}}$, restricts as
\begin{equation*}
\int_{\lfact{Z(\AAA)}{\GLtwo{2}{\AAA}}}f(g)dg= \int_{\AAA^{\times}} \int_{\AAA} \int_{K} f\left(\begin{pmatrix}y & x \cr 0 & 1 \end{pmatrix}k\right)dkdx\frac{d^{\times}y}{|y|}.
\end{equation*}
Compare this with \cite[p.6]{BlomerHarcos} and \cite[(3.10)]{GelbartJacquet}.

\section{Background on automorphic theory}\label{chap:automorphybackground}

We review some basic facts about the automorphic theory of $\GLone{2}$ that we shall use later. In the setup, we follow the work of Blomer and Harcos \cite[Sections 2.2-2.7]{BlomerHarcos}, even when it is not emphasized. Since our aim is to extend the main results of \cite{BlomerHarcos} from totally real number fields to all number fields, we will always pay special attention to the complex places.

For a Hecke character $\omega$, we denote by $L^2(\lfact{\GLtwo{2}{F}}{\GLtwo{2}{\AAA}},\omega)$ the Hilbert space of functions $\phi:\GLtwo{2}{\AAA}\rightarrow\CC$ satisfying
\begin{equation*}
\begin{split}
&|\phi|^2 = \langle \phi,\phi \rangle<\infty,\ \mathrm{where}\ \langle \phi_1,\phi_2 \rangle = \int_{\lfact{Z(\AAA)\GLtwo{2}{F}}{\GLtwo{2}{\AAA}}} \phi_1(g)\overline{\phi_2(g)}dg;\\ & \forall z\in\AAA^{\times},\gamma\in\GLtwo{2}{F},g\in\GLtwo{2}{\AAA}: \phi\left(\begin{pmatrix}z & 0 \cr 0 & z\end{pmatrix}\gamma g\right)=\omega(z)\phi(g).
\end{split}
\end{equation*}
On $L^2(\lfact{\GLtwo{2}{F}}{\GLtwo{2}{\AAA}},\omega)$, the group $\GLtwo{2}{\AAA}$ acts via right translations. From now on, without loss of generality, we assume that $\omega$ is trivial on $F_{\infty,+}^{\diag}$ (see \cite[p.6]{BlomerHarcos}).

\subsection{Spectral decomposition and Eisenstein series}

In this section, following \cite[Section 2.2]{BlomerHarcos} closely, we give a short exposition of the spectral decomposition of the Hilbert space $L^2(\lfact{\GLtwo{2}{F}}{\GLtwo{2}{\AAA}},\omega)$. For a detailed discussion, consult \cite[Sections 2-5]{GelbartJacquet}.

First, $\phi\in L^2(\lfact{\GLtwo{2}{F}}{\GLtwo{2}{\AAA}},\omega)$ is called cuspidal if for almost every $g\in \GLtwo{2}{\AAA}$,
\begin{equation*}
\int_{\lfact{F}{\AAA}} \phi\left(\begin{pmatrix}1 & x \cr 0 & 1\end{pmatrix}g\right)dx=0.
\end{equation*}
The closed subspace generated by cuspidal functions is an invariant subspace $L_{\cusp}$ decomposing into a countable sum of irreducible representations $V_{\pi}$, each $\pi$ occuring with finite multiplicity (see \cite[Section 2]{GelbartJacquet}). This multiplicity is in fact one, as it follows from Shalika's multiplicity-one theorem, see \cite[Proposition 11.1.1]{JacquetLanglands} for the case $\GLone{2}$. Therefore, denoting the set of cuspidal representations by $\mathcal{C}_{\omega}$, we may write
\begin{equation*}
L_{\cusp}=\bigoplus_{\pi\in\mathcal{C}_{\omega}} V_{\pi},
\end{equation*}
where the irreducible representations on the right-hand side are distinct.

To any Hecke character $\chi$ with $\chi^2=\omega$, we can associate a one-dimensional representation $V_{\chi}$ generated by $g\mapsto \chi(\det g)$, these sum up to
\begin{equation*}
L_{\sp}=\bigoplus_{\chi^2=\omega} V_{\chi}.
\end{equation*}
For details, see \cite[Sections 3-4]{GelbartJacquet}.

Now
\begin{equation*}
L^2(\lfact{\GLtwo{2}{F}}{\GLtwo{2}{\AAA}},\omega)= L_{\cusp}\oplus L_{\sp}\oplus L_{\cont},
\end{equation*}
where $L_{\cont}$ can be described in terms of Eisenstein series.

Take Hecke quasicharacters $\chi_1,\chi_2:\lfact{F^{\times}}{\AAA^{\times}}\rightarrow\CC^{\times}$ satisfying $\chi_1\chi_2=\omega$. Denote by $H(\chi_1,\chi_2)$ the space of functions $\varphi:\GLtwo{2}{\AAA}\rightarrow\CC$ satisfying
\begin{equation*}
\int_K |\varphi(k)|^2dk<\infty
\end{equation*}
and
\begin{equation}\label{definitionofsectionfunctions}
\varphi\left(\begin{pmatrix}a & x \cr 0 & b\end{pmatrix}g\right)= \chi_1(a)\chi_2(b)\left|\frac{a}{b}\right|^{1/2}\varphi(g),\qquad x\in \AAA, a,b\in \AAA^{\times}.
\end{equation}
In particular, $H(\chi_1,\chi_2)$ can be identified with the set of functions $\varphi\in L^2(K)$ satisfying
\begin{equation*}
\varphi\left(\begin{pmatrix}a & x \cr 0 & b\end{pmatrix}g\right)= \chi_1(a)\chi_2(b)\varphi(g),\qquad \begin{pmatrix}a & x \cr 0 & b\end{pmatrix}\in K.
\end{equation*}
There is a unique $s\in\CC$ such that $\chi_1(a)=|a|_{\infty}^s$ and $\chi_2(a)=|a|_{\infty}^{-s}$ for $a\in F_{\infty,+}^{\diag}$ and introduce
\begin{equation*}
H(s)=\bigoplus_{\substack{\chi_1\chi_2=\omega \\ \chi_1\chi_2^{-1}=|\cdot|_{\infty}^{2s}\ \mathrm{on}\ F_{\infty,+}^{\diag}}}H(\chi_1,\chi_2).
\end{equation*}
Now regard the space $H=\int_{s\in\CC}H(s)ds$ as a holomorphic fibre bundle over base $\CC$. Given a section $\varphi\in H$, $\varphi(s)\in H(s)$ and $\varphi(s,g)\in\CC$. The bundle $H$ is trivial, since any $\varphi(0)\in H(0)$ extends to a section $\varphi\in H$ satisfying $\varphi(s,g)=\varphi(0,g)H(g)^s$, where $H(g)$ is the height function defined at \cite[p.219]{GelbartJacquet}. (One may think of this as a deformation of the function $\varphi$.)

Define
\begin{equation*}
L'_{\cont}=\int_{0}^{\infty}H(iy)dy,
\end{equation*}
and equip it with the inner product
\begin{equation*}
\begin{split}
\langle \phi_1,\phi_2 \rangle&= \frac{2}{\pi}\int_0^{\infty} \langle \phi_1(iy),\phi_2(iy) \rangle dy\\ &= \frac{2}{\pi} \int_0^{\infty} \int_{\lfact{F^{\times}}{\AAA^{1}}} \int_K \phi_1\left(iy,\begin{pmatrix} a & 0 \cr 0 & 1 \end{pmatrix}k\right) \overline{\phi_2\left(iy,\begin{pmatrix} a & 0 \cr 0 & 1 \end{pmatrix}k\right)} dkdady,
\end{split}
\end{equation*}
where $\AAA^{1}$ stands for the group of ideles of norm $1$ (see \cite[(3.15)]{GelbartJacquet}).
Then there is an intertwining operator $S:L_{\cont}\rightarrow L_{\cont}'$ given by \cite[(4.23)]{GelbartJacquet} on a dense subspace. Now combining this with the theory of Eisenstein series \cite[Section 5]{GelbartJacquet}, we obtain the spectral decomposition of $L_{\cont}$.

For $\varphi\in H$, and for $\Re s>1/2$, define the Eisenstein series
\begin{equation}\label{definitionofeisensteinseries}
E(\varphi(s),g)= \sum_{\gamma\in\lfact{B(F)}{\GLtwo{2}{F}}} \varphi(s,\gamma g)
\end{equation}
on $\GLtwo{2}{\AAA}$. This is a holomorphic function which continues meromorphically to $s\in\CC$, with no poles on the line $\Re s=0$. Now for $y\in\RR^{\times}$, consider the complex vector space
\begin{equation*}
V(iy)=\{E(\varphi(iy)):\varphi(iy)\in H(iy)\}
\end{equation*}
with the inner product
\begin{equation*}
\langle E(\varphi_1(iy)), E(\varphi_2(iy))\rangle= \langle \varphi_1(iy), \varphi_2(iy)\rangle.
\end{equation*}
As above,
\begin{equation*}
V(iy)=\bigoplus_{\substack{\chi_1\chi_2=\omega \\ \chi_1\chi_2^{-1}=|\cdot|_{\infty}^{2iy}\ \mathrm{on}\ F_{\infty,+}^{\diag}}}V_{\chi_1,\chi_2},
\end{equation*}
with
\begin{equation*}
V_{\chi_1,\chi_2}=\{E(\varphi(iy)):\varphi(iy)\in H(\chi_1,\chi_2)\}.
\end{equation*}
Here, $V(iy)=V(-iy)$ by \cite[(4.3), (4.24), (5.15)]{GelbartJacquet}. Therefore, we have a $\GLtwo{2}{\AAA}$-invariant decomposition
\begin{equation*}
L_{\cont}=\int_0^{\infty} V(iy)dy= \int_0^{\infty} \bigoplus_{\substack{\chi_1\chi_2=\omega \\ \chi_1\chi_2^{-1}=|\cdot|_{\infty}^{2iy}\ \mathrm{on}\ F_{\infty,+}^{\diag}}} V_{\chi_1,\chi_2}dy.
\end{equation*}

In fact, \cite[(4.24), (5.15-18)]{GelbartJacquet} implies that for $\phi\in L_{\cont}$, taking $S\phi=\varphi\in L_{\cont}'$,
\begin{equation*}
\phi(g)=\frac{1}{\pi}\int_0^{\infty} E(\varphi(iy),g)dy,
\end{equation*}
and also Plancherel holds, that is,
\begin{equation*}
\begin{split}
\langle \phi_1,\phi_2 \rangle&= \frac{1}{\pi} \int_0^{\infty} \langle E(\varphi(iy),g),\phi_2 \rangle dy\\ &= \frac{2}{\pi} \int_0^{\infty} \langle \varphi_1(iy),\varphi_2(iy) \rangle dy= \frac{2}{\pi} \int_0^{\infty} \langle E(\varphi_1(iy)),E(\varphi_2(iy)) \rangle dy.
\end{split}
\end{equation*}

To summarize,
\begin{equation}\label{spectraldecomposition}
L^2(\lfact{\GLtwo{2}{F}}{\GLtwo{2}{\AAA}},\omega)= \bigoplus_{\pi\in\mathcal{C}_{\omega}} V_{\pi} \oplus \bigoplus_{\chi^2=\omega} V_{\chi} \oplus \int_0^{\infty}\bigoplus_{\substack{\chi_1\chi_2=\omega \\ \chi_1\chi_2^{-1}=|\cdot|_{\infty}^{2iy}\ \mathrm{on}\ F_{\infty,+}^{\diag}}} V_{\chi_1,\chi_2}dy,
\end{equation}
a function on the left-hand side decomposes into a convergent sum and integral of functions from the spaces appearing on the right-hand side, and also Plancherel holds.

For the Eisenstein spectrum, we introduce the notation $\int_{\mathcal{E}_{\omega}}V_{\varpi}d\varpi$, where $\mathcal{E}_{\omega}$ is the set of unordered pairs of Hecke characters $\{\chi_1,\chi_2\}$ which are nontrivial on $F_{\infty,+}^{\diag}$ and satisfy $\chi_1\chi_2=\omega$.

\subsection{Derivations and weights}

We review the action of the Lie algebra $\mathfrak{sl}_2(F_{\infty})$ on the space $L^2(\lfact{\GLtwo{2}{F}}{\GLtwo{2}{\AAA}},\omega)$, following \cite[Sections 2.3 and 2.10]{BlomerHarcos} at real places, \cite[Section 3]{BruggemanMotohashi} and \cite[Chapter 2]{Lokvenec} at complex places.

First we give a real basis such that each basis element is $0$ for all but one place $F_j$. At this exceptional place, we use the following elements. For a real place ($j\leq r$), let
\begin{equation}\label{differentialoperatorbasisreal}
\mathbf{H}_j=\begin{pmatrix}1 & 0 \\ 0 & -1\end{pmatrix},\qquad
\mathbf{R}_j=\begin{pmatrix}0 & 1 \\ 0 & 0 \end{pmatrix},\qquad
\mathbf{L}_j=\begin{pmatrix}0 & 0 \\ 1 & 0 \end{pmatrix},
\end{equation}
while for a complex place ($j>r$), let
\begin{equation}\label{differentialoperatorbasiscomplex}
\begin{split}
&\mathbf{H}_{1,j}=\frac{1}{2}\begin{pmatrix}1 & 0 \\ 0 & -1\end{pmatrix},\qquad
\mathbf{V}_{1,j}=\frac{1}{2}\begin{pmatrix}0 & 1 \\ 1 & 0 \end{pmatrix},\qquad
\mathbf{W}_{1,j}=\frac{1}{2}\begin{pmatrix}0 & 1 \\ -1 & 0 \end{pmatrix},\\ &
\mathbf{H}_{2,j}=\frac{1}{2}\begin{pmatrix}i & 0 \\ 0 & -i\end{pmatrix},\qquad
\mathbf{V}_{2,j}=\frac{1}{2}\begin{pmatrix}0 & i \\ -i & 0 \end{pmatrix},\qquad
\mathbf{W}_{2,j}=\frac{1}{2}\begin{pmatrix}0 & i \\ i & 0 \end{pmatrix}.
\end{split}
\end{equation}
An element $X\in\frak{sl}_2(F_{\infty})$ acts as a right-differentiation on a function $\phi:\GLtwo{2}{\AAA}\rightarrow\CC$ via
\begin{equation*}
(X\phi)(g)=\left.\frac{d}{dt}\phi(g\exp(tX))\right|_{t=0}.
\end{equation*}
Let $\frak{g}=\frak{sl}_2(F_{\infty})\otimes_{\RR}{\CC}$ be the complexified Lie algebra and set $U(\frak{g})$ for its universal enveloping algebra, consisting of higher-order right-differentiations with complex coefficients.

The above-defined first-order differentiations give rise to local Casimir elements
\begin{equation}\label{definitionofcasimirelements}
\begin{split}
&\Omega_j=-\frac{1}{4}\left(\mathbf{H}_j^2-2\mathbf{H}_j+ 4\mathbf{R}_j\mathbf{L}_j\right),\\ &
\Omega_{\pm,j}=\frac{1}{8} \left((\mathbf{H}_{1,j}\mp\mathbf{H}_{2,j})^2+ (\mathbf{V}_{1,j}\mp\mathbf{W}_{2,j})^2- (\mathbf{W}_{1,j}\mp\mathbf{V}_{2,j})^2\right)
\end{split}
\end{equation}
at real and complex places, respectively.

On an irreducible unitary representation $(\pi,V_{\pi})$, these local Casimir elements act as scalars, that is, for $\phi\in V_{\pi}^{\infty}$, $\Omega_j\phi=\lambda_j\phi$, $\Omega_{+,j}\phi=\lambda_{+,j}\phi$, $\Omega_{-,j}\phi=\lambda_{-,j}\phi$ with
\begin{equation}\label{definitionofcasimirelementactions}
\lambda_j=\frac{1}{4}-\nu_j^2, \qquad \lambda_{\pm,j}=\frac{1}{8}\left((\nu_j\mp p_j)^2-1\right),
\end{equation}
where the parameters can be described as follows. At each place, the representation can be either even or odd (according to the action of the element which is $-\id$ at the corresponding place and $\id$ at all other places). At real places, there are three families of representations: principal series $\nu_j\in i\RR$, complementary series $\nu_j\in [-\theta,\theta]$ (only in the even case), and discrete series $\nu_j\in 1/2+\ZZ$ in the even case and $\nu_j\in \ZZ$ in the odd case. At complex places, there are two families of representations: principal series $\nu_j\in i\RR,p_j\in\ZZ$ in the even case and $\nu_j\in i\RR,p_j\in 1/2+\ZZ$ in the odd case, and complementary series $\nu_j\in [-2\theta,2\theta],p_j=0$ (only in the even case). Here, $\theta$ is a constant towards the Ramanujan-Petersson conjecture, according to the current state of art (see \cite{BlomerBrumley}), $\theta=7/64$ is admissible.

For some $\mathcal{D}\in U(\frak{g})$ and a smooth vector $\phi\in L^2(\lfact{\GLtwo{2}{F}}{\GLtwo{2}{\AAA}},\omega)$, recalling the spectral decomposition (\ref{spectraldecomposition}),
\begin{equation*}
\phi=\sum_{\pi\in\mathcal{C}_{\omega}}\phi_{\pi}+ \sum_{\chi^2=\omega}\phi_{\chi}+\int_{\mathcal{E}_{\omega}}\phi_{\varpi}d\varpi,
\end{equation*}
we have
\begin{equation}\label{spectraldecompositionofderivatives}
||\mathcal{D}\phi||^2=\sum_{\pi\in\mathcal{C}_{\omega}}||\mathcal{D}\phi_{\pi}||^2+ \sum_{\chi^2=\omega}||\mathcal{D}\phi_{\chi}||^2 +\int_{\mathcal{E}_{\omega}}||\mathcal{D}\phi_{\varpi}||^2d\varpi,
\end{equation}
see \cite[Sections 1.2-4]{CogdellPiatetskiShapiro} with references to \cite{DixmierMalliavin}. Compare (\ref{spectraldecompositionofderivatives}) also with \cite[(33)]{BlomerHarcosDuke} and \cite[(84)]{BlomerHarcos}.

We focus on the local subgroups $\mathrm{SO}_2(\RR)$ (for $j\leq r$) and $\mathrm{SU}_2(\CC)$ (for $j>r$), they are compact, connected and (modulo the center) maximal subgroups of $\GLtwo{2}{\RR}$, $\GLtwo{2}{\CC}$, respectively, with these properties. The corresponding Lie algebras are $\frak{so}_2(\RR)$ and $\frak{su}_2(\CC)$, and define
\begin{equation}\label{definitionofcasimirelementscompact}
\Omega_{\frak{k},j}=\mathbf{R}_j-\mathbf{L}_j, \qquad \Omega_{\frak{k},j}= -\frac{1}{2}(\mathbf{H}_{2,j}^2+\mathbf{W}_{1,j}^2+\mathbf{W}_{2,j}^2),
\end{equation}
at real and complex places, respectively. At a complex place, $\Omega_{\frak{k},j}$ is the Casimir element (see \cite[Definition 9 on p.72]{Sugiura}).

We now define the weight set $W(\pi)$. For $j\leq r$, let $q_j$ be any integer of the same parity as the representation at the corresponding place, with the only restriction $|q_j|\geq 2|\nu_j|+1$ in the discrete series. For $j>r$, let $(l_j,q_j)$ be any pair of numbers satisfying $|q_j|\leq l_j\geq |p_j|$ and $p_j\equiv q_j\equiv l_j\pmod 1$. Now set
\begin{equation}\label{definitionofweight}
\ww=\left(q_1,\ldots,q_r,(l_{r+1},q_{r+1}),\ldots,(l_{r+s},q_{r+s})\right)
\end{equation}
and denote by $W(\pi)$ the set of $\ww$'s satisfying the above condition.

For a given $\ww\in W(\pi)$, we say that $\phi:\GLtwo{2}{\AAA}\rightarrow\CC$ is of weight $\ww$, if for $j\leq r$,
\begin{equation}\label{definitionofcompactcasimirelementactionsreal}
\Omega_{\frak{k},j}\phi=iq_j\phi
\end{equation}
and for $j>r$,
\begin{equation}\label{definitionofcompactcasimirelementactionscomplex}
\mathbf{H}_{2,j}\phi=-iq_j\phi,\qquad \Omega_{\frak{k},j}\phi=\frac{1}{2}(l_j^2+l_j)\phi,
\end{equation}
for the action of $\Omega_{\frak{k},j}$ at complex places, see \cite[Chapter II, Proposition 5.15]{Sugiura}.

Note that $W(\pi)$, through $(q_1,\ldots,q_r,l_{r+1},\ldots,l_{r+s})$, lists all irreducible representations of $K_{\infty}$ occuring in $\pi$, while $(q_{r+1},\ldots,q_{r+s})$ is to single out a one-dimensional space from each such representation.

Similarly, introduce the notation
\begin{equation}\label{definitionofspectralparameter}
\rr= \left(\nu_1\ldots,\nu_r,(\nu_{r+1},p_{r+1}),\ldots,(\nu_{r+s},p_{r+s})\right),
\end{equation}
and also its norm
\begin{equation}\label{definitionofnormofspectralparameter}
\norm{\rr}=\prod_{j=1}^{r}(1+|\nu_j|)\prod_{j=r+1}^{r+s}(1+|\nu_j|+|p_j|)^2,
\end{equation}
compare this with \cite[Section 3.1.8]{MichelVenkatesh}.

\subsection{Cuspidal spectrum}

\paragraph{Analytic conductor, newforms and oldforms}

Let $V_{\pi}$ be a cuspidal representation occuring in $L^2(\lfact{\GLtwo{2}{F}}{\GLtwo{2}{\AAA}},\omega)$. By the tensor product theorem (see \cite[Section 3.4]{Bump} or \cite{Flath}),
\begin{equation}\label{tensorproductdecompositionforcuspidalrepresentations}
V_{\pi}=\bigotimes_{v} V_{\pi_v}
\end{equation}
as a restricted tensor product with respect to the family $\{K_{\frak{p}}(\frak{o}_{\frak{p}})\}$ (by \cite[Theorem 3.3.4]{Bump}, irreducible cuspidal representations are admissible).

For an ideal $\frak{c}\subseteq\frak{c}_{\omega}$ (with $\frak{c}_{\omega}$ standing for the conductor of $\omega$), let
\begin{equation*}
V_{\pi}(\frak{c})=\left\{ \phi\in V_{\pi}: \phi\left(g\begin{pmatrix} a & b \cr c & d\end{pmatrix}\right)=\omega_{\frak{c}}(d)\phi(g),\ \mathrm{if}\ g\in\GLtwo{2}{\AAA},\begin{pmatrix} a & b \cr c & d\end{pmatrix}\in K(\frak{c}) \right\},
\end{equation*}
where $\omega_{\frak{c}}(x)=\prod_{\frak{p}|\frak{c}}\omega_{\frak{p}}(x)$.
Then $\frak{c}'\subseteq\frak{c}$ implies $V_{\pi}(\frak{c}')\supseteq V_{\pi}(\frak{c})$ (see \cite[p.9]{BlomerHarcos}).

By \cite[Corollary 2(a) of Theorem 2]{Miyake}, there is a nonzero ideal $\frak{c}_{\pi}$ such that $V_{\pi}(\frak{c})$ is nontrivial if and only if $\frak{c}\subseteq\frak{c}_{\pi}$. Now the analytic conductor of the representation is defined as
\begin{equation}\label{definitionofanalyticconductor}
C(\pi)=\norm{\frak{c}_{\pi}}\norm{\rr}.
\end{equation}
Introducing also
\begin{equation*}
V_{\pi,\ww}(\frak{c})=\{\phi\in V_{\pi}(\frak{c}): \phi\ \mathrm{is}\ \mathrm{of}\ \mathrm{weight}\ \ww\}
\end{equation*}
for $\ww\in W(\pi)$,
\cite[Corollary 2(b) of Theorem 2]{Miyake} states that for any $\ww\in W(\pi)$, $V_{\pi,\ww}(\frak{c}_{\pi})$ is one-dimensional, that is, restricting $V_{\pi}(\frak{c}_{\pi})$ to $K_{\infty}$, each irreducible representation of $K_{\infty}$ listed in $W(\pi)$ appears with multiplicity one. A nontrivial element of $V_{\pi,\ww}(\frak{c}_{\pi})$ is called a newform of weight $\ww$.

Now consider an ideal $\frak{c}\subseteq\frak{c}_{\pi}$, and take any ideal $\frak{t}$ such that $\frak{tc}_{\pi}\supseteq\frak{c}$. Fixing some finite idele $t\in\AAA_{\fin}^{\times}$ representing $\frak{t}$, we obtain an isometric embedding
\begin{equation}\label{raisinglevel}
R_{\frak{t}}:V_{\pi}(\frak{c}_{\pi})\hookrightarrow V_{\pi}(\frak{c}), \qquad (R_{\frak{t}}\phi)(g)=\phi\left(g\begin{pmatrix}t^{-1} & 0 \cr 0 & 1 \end{pmatrix}\right).
\end{equation}
Then combining \cite[Corollary 2(c) of Theorem 2]{Miyake} with \cite[Corollary on p.306]{Casselman} and (\ref{tensorproductdecompositionforcuspidalrepresentations}), we see the decompositions
\begin{equation*}
V_{\pi}(\frak{c})=\bigoplus_{\frak{t}|\frak{cc}_{\pi}^{-1}} R_{\frak{t}}V_{\pi}(\frak{c}_{\pi}), \qquad V_{\pi,\ww}(\frak{c})=\bigoplus_{\frak{t}|\frak{cc}_{\pi}^{-1}} R_{\frak{t}}V_{\pi,\ww}(\frak{c}_{\pi}),
\end{equation*}
which are not orthogonal in general. However, in Section \ref{sec:lfunctions_rankinselbergconvolution} we will prove that for ideals $\frak{t}_1,\frak{t}_2$, $\langle R_{\frak{t}_1}\phi_1, R_{\frak{t}_2}\phi_2 \rangle=\langle \phi_1, \phi_2 \rangle C(\frak{t}_1,\frak{t}_2,\pi)$, with the constant factor $C(\frak{t}_1,\frak{t}_2,\pi)$ depending only on $\frak{t}_1,\frak{t}_2,\pi$, but not on $\ww$. This allows us to use the Gram-Schmidt method, obtaining complex numbers $\alpha_{\frak{t},\frak{s}}$ (with $\alpha_{\frak{o},\frak{o}}=1$) for any pair of ideals $\frak{s}|\frak{t}|\frak{cc}_{\pi}^{-1}$ such that the isometries
\begin{equation}\label{orthogonaloldforms}
R^{\frak{t}}=\sum_{\frak{s}|\frak{t}} \alpha_{\frak{t},\frak{s}}R_{\frak{s}}:V_{\pi}(\frak{c}_{\pi})\hookrightarrow V_{\pi}(\frak{c}),\qquad \frak{t}|\frak{cc}_{\pi}^{-1},
\end{equation}
give rise to the orthogonal decompositions
\begin{equation}\label{newformoldformdecomposition1}
V_{\pi}(\frak{c})= \bigoplus_{\frak{t}|\frak{cc}_{\pi}^{-1}}R^{\frak{t}}V_{\pi}(\frak{c}_{\pi}),  \qquad V_{\pi,\ww}(\frak{c})=\bigoplus_{\frak{t}|\frak{cc}_{\pi}^{-1}} R^{\frak{t}}V_{\pi,\ww}(\frak{c}_{\pi}).
\end{equation}

\paragraph{Whittaker functions and the Fourier-Whittaker expansion}

For a given $\rr,\ww$ (recall (\ref{definitionofweight}) and (\ref{definitionofspectralparameter})), we define the Whittaker function as the product of Whittaker functions at archimedean places. The important property of these functions is that they are the exponentially decaying eigenfunctions of the Casimir operators $\Omega,\Omega_{\pm}$, therefore, they emerge in the Fourier expansion of automorphic forms (see \cite[Section 3.5]{Bump}).

At real places,
\begin{equation}\label{definitionofrealwhittakerfunction}
\whit_{q,\nu}(y)=\frac{i^{\sign(y)\frac{q}{2}} W_{\sign(y)\frac{q}{2},\nu}(4\pi|y|)}{ (\Gamma(\frac{1}{2}-\nu+\sign(y)\frac{q}{2}) \Gamma(\frac{1}{2}+\nu+\sign(y)\frac{q}{2}))^{1/2}},
\end{equation}
$W$ denoting the classical Whittaker function (see \cite[Chapter XVI]{WhittakerWatson}). This is taken from \cite[(23)]{BlomerHarcos}.

At complex places, we first define the Whittaker function on the positive real axis via
\begin{equation}\label{definitionofcomplexwhittakerfunction}
\begin{split}
\whit_{(l,q),(\nu,p)}(|y|)&= \frac{\sqrt{8(2l+1)}}{(2\pi)^{\Re\nu}}{{2l}\choose{l-q}}^{\frac{1}{2}} {{2l}\choose{l-p}}^{-\frac{1}{2}} \sqrt{\left|\frac{\Gamma(l+1+\nu)}{\Gamma(l+1-\nu)}\right|}\\ &\cdot(-1)^{l-p} (2\pi)^{\nu} i^{-p-q} w^{l}_{q}(\nu,p;|y|),
\end{split}
\end{equation}
where
\begin{equation}\label{definitionofcomplexwhittakerfunctioncomplement1}
w^{l}_{q}(\nu,p;|y|)= \sum_{k=0}^{l-\frac{1}{2}(|q+p|+|q-p|)}(-1)^k \xi^{l}_{p}({q},k)\frac{(2\pi |y|)^{l+1-k}}{\Gamma(l+1+\nu-k)}K_{\nu+l-|q+p|-k}(4\pi |y|),
\end{equation}
$K$ denoting the $K$-Bessel function, and
\begin{equation}\label{definitionofcomplexwhittakerfunctioncomplement2}
\xi^{l}_{p}(q,k)=\frac{k!(2l-k)!}{(l-p)!(l+p)!} {{l-\frac{1}{2}(|q+p|+|q-p|)}\choose{k}} {{l-\frac{1}{2}(|q+p|-|q-p|)}\choose{k}}.
\end{equation}
Then extend this to $y\in\CC^{\times}$ to satisfy
\begin{equation*}
\whit_{(l,q),(\nu,p)}(ye^{i\theta})=e^{-iq\theta}\whit_{(l,q),(\nu,p)}(y),\qquad y\in\CC^{\times},\theta\in\RR.
\end{equation*}
This definition is borrowed from \cite[Section 5]{BruggemanMotohashi} and \cite[Section 4.1]{Lokvenec}, apart from the first line, which is a normalization to gain the right $L^2$-norm.

In both cases, the occuring numbers $\nu,p,q,l$ are those given by the representation and weight data, encoded in the action of the elements $\Omega,\Omega_{\pm},\Omega_{\frak{k}},\mathbf{H}_2$ (recall (\ref{differentialoperatorbasiscomplex}), (\ref{definitionofcasimirelements}), (\ref{definitionofcasimirelementactions}), (\ref{definitionofcasimirelementscompact}), (\ref{definitionofweight}), (\ref{definitionofspectralparameter})).

Finally, define the archimedean Whittaker function as
\begin{equation*}
\whit_{\ww,\rr}(y)=\prod_{j\leq r}\whit_{q_j,\nu_j}(y_j) \prod_{j>r}\whit_{(l_j,q_j),(\nu_j,p_j)}(y_j).
\end{equation*}
With the given normalization, for a fixed $\rr$,
\begin{equation}\label{whittakerfunctionsareorthonormal}
\int_{F_{\infty}^{\times}}\whit_{\ww,\rr}(y)\overline{\whit_{\ww',\rr}}(y) d_{\infty}^{\times}y=\delta_{\ww,\ww'}.
\end{equation}
This can be seen as the product of the analogous results at single places. For real places, see \cite[(25)]{BlomerHarcos} and \cite[Section 4]{BruggemanMotohashinew}. As for complex places, see \cite[Lemma 2]{MPkuznetsov} and \cite[Lemma 8.1 and Lemma 8.2]{MPphd}.

Now we extend \cite[Section 2.5]{BlomerHarcos} to our more general situation. For any $\pi\in\mathcal{C}_{\omega}$, $\frak{c}\subseteq\frak{c}_{\pi}$, any function $\phi\in V_{\pi,\ww}(\frak{c})$ can be expanded into Fourier series as follows. There exists a character $\varepsilon_{\pi}:\{\pm 1\}^r\rightarrow\{\pm 1\}$ depending only on $\pi$ such that
\begin{equation}\label{fourierwhittakerexpansion}
\phi\left( \begin{pmatrix} y & x \cr 0 & 1 \end{pmatrix} \right)= \sum_{t\in F^{\times}}\rho_{\phi}(ty_{\fin}) \varepsilon_{\pi}(\sign(ty_{\infty}))\whit_{\ww,\rr}(ty_{\infty})\psi(tx).
\end{equation}
Note that $\varepsilon_{\pi}$ is not well-defined, if we are in the discrete series and that the coefficient $\varrho(ty_{\fin})$ depends only on the fractional ideal generated by $ty_{\fin}$. Moreover, it is zero, if this fractional ideal is nonintegral. For the proof of this, see \cite[Section 4]{DukeFriedlanderIwaniec}, \cite[Sections 1-3]{KhuriMakdisi} or \cite[Proposition 2.1]{MPphd}.

Now assume that $\frak{c}=\frak{c}_{\pi}$, i.e. $\phi$ is a newform of weight $\ww$. In this case, the coefficients $\varrho_{\pi}(\frak{m})$ are proportional to the Hecke eigenvalues $\lambda_{\pi}(\frak{m})$:
\begin{equation*}
\varrho_{\phi}(\frak{m})= \frac{\lambda_{\pi}(\frak{m})}{\sqrt{\norm{\frak{m}}}}\varrho_{\phi}(\frak{o}).
\end{equation*}
We record
\begin{equation}\label{towardsramanujanpetersson}
\lambda_{\pi}(\frak{m})\ll_{\varepsilon} \norm{\frak{m}}^{\theta+\varepsilon}
\end{equation}
with $\theta=7/64$ \cite{BlomerBrumley}, while according to the Ramanujan-Petersson conjecture, $\theta=0$ is admissible. Also note the multiplicativity relation
\begin{equation}\label{heckeeigenvaluesmultiplicativity}
\lambda_{\pi}(\frak{m})\lambda_{\pi}(\frak{n})= \sum_{\frak{a}|\gcd(\frak{m},\frak{n})} \lambda_{\pi}(\frak{mna}^{-2}).
\end{equation}

Setting
\begin{equation}\label{kirillovvectorsforpureweightforms}
W_{\phi}(y) =\varrho_{\phi}(\frak{o})\varepsilon_{\pi}(\sign(y))\whit_{\ww,\rr}(y),\qquad y\in F_{\infty}^{\times},
\end{equation}
we obtain
\begin{equation}\label{fourierwhittakerexpansion2newforms}
\phi\left( \begin{pmatrix} y & x \cr 0 & 1 \end{pmatrix} \right)= \sum_{t\in F^{\times}} \frac{\lambda_{\pi}(ty_{\fin})}{\sqrt{\norm{ty_{\fin}}}} W_{\phi}(ty_{\infty})\psi(tx).
\end{equation}

\paragraph{The archimedean Kirillov model}

Now fixing $y_{\fin}=(1,1,\ldots)$, we can single out the term corresponding to $t=1$:
\begin{equation}\label{alternativedefinitionforkirillovvector1}
W_{\phi}(y)=\int_{\lfact{F}{\AAA}}\phi\left( \begin{pmatrix} y & x \cr 0 & 1 \end{pmatrix} \right)\psi(-x)dx.
\end{equation}
In the case of arbitrary (i.e. non-necessarily pure weight) smooth functions in $V_{\pi}(\frak{c}_\pi)$, this latter formula can be considered as the definition of the mapping $\phi\mapsto W_{\phi}$, the image is a subspace in $L^2(F_{\infty}^{\times},d_{\infty}^{\times}y)$.

\begin{prop}\label{archimedeankirillovmodelproposition} The image in fact is a dense subspace in $L^2(F_{\infty}^{\times},d_{\infty}^{\times}y)$. Moreover, there is a positive constant $C_{\pi}$ depending only on $\pi$ such that
\begin{equation}\label{archimedeankirillovmodelformula}
\langle \phi_1,\phi_2 \rangle = C_{\pi}\langle W_{\phi_1},W_{\phi_2} \rangle, 
\end{equation}
where the scalar product on the left-hand side is the scalar product in $L^2(\lfact{\GLtwo{2}{F}}{\GLtwo{2}{\AAA}},\omega)$, while on the right-hand side, it is the scalar product in $L^2(F_{\infty}^{\times},d_{\infty}^{\times}y)$.
The map $\phi\mapsto W_{\phi}$ is therefore surjective from $V_{\pi}(\frak{c}_\pi)$ to $L^2(F_{\infty}^{\times},d_{\infty}^{\times}y)$.
\end{prop}
\begin{proof} On the space $L^2(F_{\infty}^{\times},d_{\infty}^{\times}y)$, the Borel subgroup $B(F_{\infty})$ acts through the Kirillov model action
\begin{equation*}
\left(\begin{pmatrix}y' & x' \cr 0 & 1\end{pmatrix}W_{\phi}\right)(y)= \psi_{\infty}(x'y)W_{\phi}(y'y).
\end{equation*}
This action is irreducible on a single $L^2(\RR^{\times},d_{\RR}^{\times}y)$ or $L^2(\CC^{\times},d_{\CC}^{\times}y)$ (combine \cite[Propositions 2.6 and 2.7]{Knapp} with \cite[p.197]{Kirillov}), so is on their tensor product
\begin{equation*}
L^2(F_{\infty}^{\times},d_{\infty}^{\times}y)= \overline{\bigotimes_{j=1}^{r+s} L^2(F_j^{\times},d_{F_j}^{\times}y_j)}.
\end{equation*}
Then taking some $\phi\in V_{\pi}^{\infty}(\frak{c}_{\pi})$ such that $W_{\phi}$ is not identically zero, a closed, invariant subspace containing $W_{\phi}$ must equal $L^2(F_{\infty}^{\times},d_{\infty}^{\times}y)$, because of irreducibility (the existence of such a $\phi$ follows from the Fourier-Whittaker expansion, which includes harmonics with nonzero coefficients).

In Section \ref{sec:lfunctions_rankinselbergconvolution}, we will prove that if $\phi_1,\phi_2\in V_{\pi,\ww}(\frak{c}_{\pi})$, then
\begin{equation}\label{essentialisometryofkirillovmodel}
\langle \phi_1,\phi_2 \rangle = C_{\pi}\langle W_{\phi_1},W_{\phi_2} \rangle.
\end{equation}
If moreover $\phi_1,\phi_2$ are of different weights $\ww_1\neq\ww_2$, then both sides are $0$, since for pure weight forms, the associated Kirillov vectors are proportional to $\whit_{\rr,\ww_{1,2}}$ (recall (\ref{kirillovvectorsforpureweightforms})), which are orthogonal by (\ref{whittakerfunctionsareorthonormal}). Then the orthogonal decomposition
\begin{equation*}
V_{\pi}(\frak{c}_{\pi})=\bigoplus_{\ww\in W(\pi)}V_{\pi,\ww}(\frak{c}_{\pi})
\end{equation*}
completes the proof.
\end{proof}

Now turn to the general case $\frak{c}\subseteq\frak{c}_{\pi}$. Using the isometries $R^{\frak{t}}$, (\ref{fourierwhittakerexpansion2newforms}) gives rise to, for every $\phi\in R^{\frak{t}}V_{\pi}(\frak{c}_{\pi})$,
\begin{equation}\label{fourierwhittakerexpansion2oldforms}
\phi\left( \begin{pmatrix} y & x \cr 0 & 1 \end{pmatrix} \right)= \sum_{t\in F^{\times}} \frac{\lambda_{\pi}^{\frak{t}}(ty_{\fin})}{\sqrt{\norm{ty_{\fin}}}} W_{\phi}(ty_{\infty})\psi(tx),
\end{equation}
with
\begin{equation}\label{orthogonalizedfouriercofficients}
W_{\phi}=W_{(R^{\frak{t}})^{-1}\phi},\qquad \lambda_{\pi}^{\frak{t}}(\frak{m})=\sum_{\frak{s}|\mathrm{gcd}(\frak{t},\frak{m})} \alpha_{\frak{t},\frak{s}}\norm{\frak{s}}^{1/2}\lambda_{\pi}(\frak{ms}^{-1}).
\end{equation}

\subsection{Eisenstein spectrum}

In this section, we develop the theory of Eisenstein series. Since Eisenstein series will show up only in the spectral decomposition of an automorphic function of trivial central character, we assume temporarily that $\omega=1$. From now on, let $\chi\in\mathcal{E}_{1}$ be a Hecke character which is nontrivial on $F_{\infty,+}^{\diag}$.

\paragraph{Analytic conductor, newforms and oldforms}

Similarly to the cuspidal case, for any ideal $\frak{c}\subseteq\frak{o}$, define
\begin{equation*}
V_{\chi,\chi^{-1}}(\frak{c})=\left\{ \phi\in V_{\chi,\chi^{-1}}: \phi\left(g\begin{pmatrix} a & b \cr c & d\end{pmatrix}\right)=\phi(g),\ \mathrm{if}\ g\in\GLtwo{2}{\AAA},\begin{pmatrix} a & b \cr c & d\end{pmatrix}\in K(\frak{c})\right\}. 
\end{equation*}
Using that $V_{\chi,\chi^{-1}}$ and $H(\chi,\chi^{-1})$ are isomorphic as $\GLtwo{2}{\AAA}$-representations, we have
\begin{equation*}
V_{\chi,\chi^{-1}}(\frak{c})=\{E(\varphi(iy),\cdot)\in V_{\chi,\chi^{-1}}: \varphi\in H(\chi,\chi^{-1},\frak{c})\}
\end{equation*}
with
\begin{equation*}
H(\chi,\chi^{-1},\frak{c})=\left\{ \varphi\in H(\chi,\chi^{-1}): \varphi\left(g\begin{pmatrix} a & b \cr c & d\end{pmatrix}\right)=\varphi(g),\ \mathrm{if}\ \begin{pmatrix} a & b \cr c & d\end{pmatrix}\in K(\frak{c})\right\}. 
\end{equation*}

Analogously to (\ref{tensorproductdecompositionforcuspidalrepresentations}), we have
\begin{equation*}
H(\chi,\chi^{-1})=\bigotimes_{v} H_v(\chi,\chi^{-1}),
\end{equation*}
a restricted tensor product with respect to the family $\{K_{\frak{p}}(\frak{o}_{\frak{p}})\}$ again, the admissibility of $H(\chi,\chi^{-1})$ is straight-forward.

Assume $\chi$ has conductor $\frak{c}_{\chi}$. The following is taken from \cite[Section 2.6]{BlomerHarcos}.

\begin{prop}\label{conductorofeisensteinseries} For any non-archimedean place $\frak{p}$, set $d=v_{\frak{p}}(\frak{d})$ and $m=v_{\frak{p}}(\frak{c}_{\chi})$, and fix some $\varpi$ such that $v_{\frak{p}}(\varpi)=1$. Then for any integer $n\geq 0$, the complex vector space $H_{\frak{p}}(\chi,\chi^{-1},\frak{p}^n)$ has dimension $\max(0,n-2m+1)$. For $n\geq 2m$, an orthogonal basis is $\{\varphi_{\frak{p},j}: 0\leq j\leq n-2m\}$ with functions $\varphi_{\frak{p},j}$ defined as follows.

If $m=0$ and $k=\begin{pmatrix} * & * \cr b\varpi^d & * \end{pmatrix}\in K_{\frak{p}}(\frak{o}_\frak{p})$, let
\begin{equation*}
\varphi_{\frak{p},0}(k)=1;\qquad \varphi_{\frak{p},1}(k)=\left\{\begin{array}{ll} \norm{\frak{p}}^{-1/2}, & \mathrm{if}\ v_{\frak{p}}(b)=0,\cr -\norm{\frak{p}}^{1/2}, & \mathrm{if}\ v_{\frak{p}}(b)\geq 1;\end{array}\right.
\end{equation*}
while for $j\geq 2$,
\begin{equation*}
\varphi_{\frak{p},j}(k)=\left\{\begin{array}{ll} 0, & v_{\frak{p}}(b)\leq j-2,\cr -\norm{\frak{p}}^{j/2-1}, & \mathrm{if}\ v_{\frak{p}}(b)= j-1,\cr \norm{\frak{p}}^{j/2}\left(1-\frac{1}{\norm{\frak{p}}}\right), & \mathrm{if} v_{\frak{p}}(b)\geq j. \end{array}\right.
\end{equation*}

If $m>0$ and $k=\bigl(\begin{smallmatrix} a & * \cr b\varpi^d & * \end{smallmatrix}\bigr)\in K_{\frak{p}}(\frak{o}_\frak{p})$, let
\begin{equation*}
\varphi_{\frak{p},j}(k)=\left\{\begin{array}{ll} \norm{\frak{p}}^{(m+j)/2}\chi_{\frak{p}}(ab^{-1}), & \mathrm{if}\ v_{\frak{p}}(b)=m+j,\cr 0, & \mathrm{if}\ v_{\frak{p}}(b)\neq m+j.\end{array}\right.
\end{equation*}
Moreover,
\begin{equation*}
1-\frac{1}{\norm{\frak{p}}}\leq ||\varphi_{\frak{p},j}||\leq 1.
\end{equation*}
\end{prop}
\begin{proof} See \cite[Lemma 1 and Remark 7]{BlomerHarcos}.
\end{proof}

Therefore, $\frak{c}_{\chi,\chi^{-1}}=(\frak{c}_{\chi})^2$ is the maximal ideal $\frak{c}$ such that $V_{\chi,\chi^{-1}}(\frak{c})$ and $H(\chi,\chi^{-1},\frak{c})$ are nontrivial.

Now turn our attention to the archimedean quasifactors $H_j(\chi,\chi^{-1})$. They are always principal series representations and their parameter $\rr$ is the following. At real places, $\nu_j\in i\RR$ of (\ref{definitionofcasimirelementactions}) is the one satisfying $\chi_j(a)=a^{\nu_j}$ for $a\in\RR_+$ (see \cite[p.83]{BruggemanMotohashinew}). At complex places, $\nu_j\in i\RR$ and $p_j\in\ZZ/2$ of (\ref{definitionofcasimirelementactions}) are those satisfying $\chi_j(ae^{i\theta})=a^{\nu_j}e^{-ip_j\theta}$ for $a\in\RR_+,\theta\in\RR$ (see \cite[Section 3]{BruggemanMotohashi} or \cite[Section 2.3]{Lokvenec}). Now these give rise to the set $W(\chi,\chi^{-1})$ of weights (those occuring in $H_j(\chi,\chi^{-1})$): the only condition is $|q_j|\leq l_j\geq |p_j|$ at complex places.

The analytic conductor is again defined as 
\begin{equation}\label{definitionofanalyticconductor2}
C(\chi,\chi^{-1})=\norm{\frak{c}_{\chi,\chi^{-1}}}\norm{\rr}.
\end{equation}

We can now give an orthogonal basis of $H(\chi,\chi^{-1},\frak{c})$ for any $\frak{c}\subseteq\frak{c}_{\chi}^2$. Given $\frak{t}|\frak{cc}_{\chi}^{-2}$ and any weight $\ww\in W(\chi,\chi^{-1})$, let $\varphi^{\frak{t},\ww}$ be the tensor product of the following local functions. At the archimedean places, let $\varphi_j^{\frak{t},\ww}(k)=\Phi_{q_j}(k)$, $\varphi_j^{\frak{t},\ww}(k)= \Phi_{p_j,q_j}^{l_j}(k)/||\Phi_{p_j,q_j}^{l_j}||_{\mathrm{SU}_2(\CC)}$ for $k\in K_j$ with $j\leq r$, $j>r$, respectively. At non-archimedean places, let $\varphi_{\frak{p}}^{\frak{t},\ww}=\varphi_{\frak{p},v_{\frak{p}}(\frak{t})}$. The global functions form an orthogonal basis of $H(\chi,\chi^{-1},\frak{c})$ and this gives rise to an orthogonal basis in $V_{\chi,\chi^{-1}}$ via the corresponding Eisenstein series $\phi^{\frak{t},\ww}=E(\varphi^{\frak{t},\ww})$. Finally, defining $R^{\frak{t}}:V_{\chi,\chi^{-1}}(\frak{c}_{\chi}^2)\hookrightarrow V_{\chi,\chi^{-1}}(\frak{c})$ as $\phi^{\frak{o},\ww}/||\phi^{\frak{o},\ww}||\mapsto \phi^{\frak{t},\ww}/||\phi^{\frak{t},\ww}||$ for all $\ww$, we obtain the orthogonal decomposition
\begin{equation}\label{newformoldformdecomposition2}
V_{\chi,\chi^{-1}}(\frak{c})=\bigoplus_{\frak{t}|\frak{cc}_{\chi}^{-2}} R^{\frak{t}}V_{\chi,\chi^{-1}}(\frak{c}_{\chi}^2).
\end{equation}

\paragraph{The Fourier-Whittaker expansion and the archimedean Kirillov model}

Similarly to cusp forms, Eisenstein series can also be expanded into Fourier-Whittaker series. Assume $\varphi$ is one of the pure tensors defined above and $\phi=E(\varphi)$, where we dropped $\frak{t}$ and $\ww$ from the notation. Denoting by $\varrho_{E(\varphi,0)}(y)$ the constant term \cite[p.220]{GelbartJacquet}, we obtain the Fourier-Whittaker expansion (see \cite[Sections 2.6 and 2.7]{BlomerHarcos} and \cite[Section 2.4]{MPphd})
\begin{equation}\label{fourierwhittakerexpansionofeisensteinseries}
\phi\left(\begin{pmatrix} y & x \cr 0 & 1\end{pmatrix}\right)= \varrho_{E(\varphi),0}(y)+ \sum_{t\in F^{\times}} \frac{\lambda_{\chi,\chi^{-1}}^{\frak{t}}(ty_{\fin})}{\sqrt{\norm{ty_{\fin}}}} W_{E(\varphi)}(ty_{\infty})\psi(tx)
\end{equation}
where the coefficients satisfy
\begin{equation}\label{boundonheckeeigenvaluesineisensteinspectrum}
\lambda_{\chi,\chi^{-1}}^{\frak{t}}(\frak{m})\ll_{F,\varepsilon} \norm{\gcd(\frak{t},\frak{m})} \norm{\frak{m}}^{\varepsilon},
\end{equation}
for all $\frak{m}\subseteq\frak{o}$. Also,
\begin{equation}\label{estimateofkirillovvectorofeisenstein}
||W_{E(\varphi)}||\ll_{F,\varepsilon} \norm{\frak{t}}^{\varepsilon}C(\chi,\chi^{-1})^{\varepsilon} ||\varphi||,
\end{equation}
where the norms are understood in the spaces $L^2(F_{\infty}^{\times},d_{\infty}^{\times}y)$ and $L^2(K)$ (recall also (\ref{definitionofanalyticconductor2})). Compare these with \cite[(48-50)]{BlomerHarcos}.

The mapping $E(\varphi)\mapsto W_{E(\varphi)}$ has similar properties as in the cuspidal spectrum. In the special case $\frak{c}=\frak{c}_{\chi}^2$, $\frak{t}=\frak{t}_{\chi}=\frak{o}$, $E(\varphi)$ spans the space $V_{\chi,\chi^{-1},\ww}(\frak{c}_{\chi}^2)$ of newforms of weight $\ww$. In this case, we have the alternative definition
\begin{equation}\label{alternativedefinitionforkirillovvector2}
W_{E(\varphi)}(y)=\int_{\lfact{F}{\AAA}} E(\varphi)\left( \begin{pmatrix} y & x \cr 0 & 1 \end{pmatrix} \right)\psi(-x)dx,
\end{equation}
where $y_{\fin}=(1,1,\ldots)$. For $\phi_1,\phi_2\in V_{\chi,\chi^{-1}}(\frak{c}_{\chi}^2)$, we have
\begin{equation}\label{archimedeankirillovmodelformula2}
\langle \phi_1,\phi_2 \rangle = C_{\chi,\chi^{-1}}\langle W_{\phi_1},W_{\phi_2}\rangle
\end{equation}
with some positive constant $C_{\chi,\chi^{-1}}\gg_{F,\varepsilon}C(\chi,\chi^{-1})^{-\varepsilon}$ depending only on $\chi$. Also, $\lambda_{\chi,\chi^{-1}}$ specialize to Hecke eigenvalues.

\subsection{A semi-adelic Kuznetsov formula over number fields} First of all, we introduce some more notations. Given an ideal $\frak{c}$, let
\begin{equation}\label{C(c)E(c)}
\mathcal{C}_{\omega}(\frak{c})=\{\pi\in\mathcal{C}_{\omega}\mid \frak{c}\subseteq\frak{c}_{\pi}\},\qquad \mathcal{E}_{\omega}(\frak{c})=\{\chi\in\mathcal{E}_{\omega}\mid \frak{c}\subseteq\frak{c}_{\chi,\chi^{-1}}\}.
\end{equation}

Now we briefly discuss a variant of the Kuznetsov formula (for details, see \cite[Theorem 1]{MPkuznetsov} or \cite[Theorem 3]{MPphd}) that we will use later, the central character $\omega$ is still assumed to be trivial. In our notation, for a weight function $h$ of the form described below,
\begin{equation}\label{kuznetsovformula}
\begin{split}
\left[K(\frak{o}):K(\frak{c})\right]^{-1} & \sum_{\pi\in\mathcal{C}_{1}(\frak{c})}C_{\pi}^{-1} \sum_{\frak{t}|\frak{cc}_{\pi}^{-1}} h(\rr_{\pi}) \lambda_{\pi}^{\frak{t}}(\alpha\frak{a}^{-1}) \overline{\lambda_{\pi}^{\frak{t}}(\alpha'\frak{a}'^{-1})}+CSC=\\
&\const_F \Delta(\alpha\frak{a}^{-1},\alpha'\frak{a}'^{-1}) \int h(\rr) d\mu+\\ &\const_F \sum_{\frak{m}\in C}\sum_{c\in \frak{amc}}\sum_{\epsilon\in\rfact{\frak{o}_+^{\times}}{\frak{o}^{2\times}}} \frac{KS(\epsilon\alpha,\frak{a}^{-1}\frak{d}^{-1};\alpha'\gamma_{\frak{m}}, \frak{a}'^{-1}\frak{d}^{-1};c,\frak{a}^{-1}\frak{m}^{-1}\frak{d}^{-1}) }{\norm{c\frak{a}^{-1}\frak{m}^{-1}}} \\ & \qquad \qquad \qquad \qquad \qquad \qquad \cdot
\int\mathcal{B}h_{(\rr)}\left(4\pi\frac{(\alpha\alpha'\gamma_\frak{m} \epsilon)^{\frac{1}{2}}}{c}\right) d\mu,
\end{split}
\end{equation}
where $KS$ is a Kloosterman sum, $\mathcal{B}$ is a certain Bessel function, and $d\mu$ is a certain measure of the space of the spectral parameters $\rr$. We explain the notation and the conditions: $\frak{a}^{-1}$ and $\frak{a}'^{-1}$ are nonzero fractional ideals; $\alpha\in\frak{a},\alpha\in\frak{a}'$ such that $\alpha\alpha'$ is totally positive; $C$ is a fixed set of narrow ideal class representatives $\frak{m}$, for which $\frak{m}^2\frak{aa}'^{-1}$ is a principal ideal generated by a totally positive element $\gamma_{\frak{m}}$; $\Delta(\alpha\frak{a}^{-1},\alpha'\frak{a}'^{-1})$ is $1$ if $\alpha\frak{a}^{-1}=\alpha'\frak{a}'^{-1}$, otherwise it is $0$; $CSC$ is an analogous integral over the Eisenstein spectrum.

The weight function $h$ we will use is of the form $h=\prod_jh_j$, where $h_j$'s are defined as follows. Let $a_j,b_j>1,a'_j\in\RR$ be given. Then at real places
\begin{equation}\label{kuznetsovtestfunctionreal}
h_j(\nu_j)=\left\{\begin{array}{ll}e^{(\nu_j^2-\frac{1}{4})/a_j}, & \mathrm{if}\ |\Re\nu_j|<\frac{2}{3}, \cr 1, & \mathrm{if}\ \nu_j\in\frac{1}{2}+\ZZ, \frac{3}{2}\leq|\nu_j|\leq b_j, \cr 0 & \mathrm{otherwise},\end{array}\right.
\end{equation}
while at complex places
\begin{equation}\label{kuznetsovtestfunctioncomplex}
h_j(\nu_j,p_j)=\left\{\begin{array}{ll}e^{(\nu_j^2+a_j'p_j^2-1)/a_j}, & \mathrm{if}\ |\Re\nu_j|<\frac{2}{3}, p_j\in\ZZ, |p_j|\leq b_j, \cr 0 & \mathrm{otherwise}.\end{array}\right.
\end{equation}

For the purpose of this paper, we will choose our parameters as follows. At each place, $a_j>1$ is arbitrary, then set $b_j=\sqrt{a_j}$. Furthermore, at complex places, we use $a_j'=-1$. In this setup, we have the bounds
\begin{equation}\label{estimateoftransformsongeometricside}
\begin{split}
&\int h_j(\nu_j)d\mu_j\ll a_j, \qquad \int (\mathcal{B}_jh_j)_{\nu_j}(t)d\mu_j\ll a_j\min(1,|t|^{1/2});\\ &
\int h_j(\nu_j,p_j)d\mu_j\ll a_j^2, \qquad \int (\mathcal{B}_jh_j)_{(\nu_j,p_j)}(t)d\mu_j\ll a_j\min(1,|t|),
\end{split}
\end{equation}
at real and complex places, respectively (see \cite[pp.124-126]{BruggemanMiatelloPacharoni}, \cite[Section 10]{BruggemanMotohashi} and \cite[Lemma 5.3]{MPphd}).

\subsection{The density of the spectrum}

In this section, we estimate the density of the Eisenstein and the cuspidal spectrum in terms of the spectral parameters. These are the extensions of \cite[Lemma 2 and Lemma 6]{BlomerHarcos}. After the suitable modifications, the proofs given there apply in the more general situation. In this section, we still assume that $\omega=1$, since this is the only case we shall use later.

\paragraph{Density of the Eisenstein spectrum}

\begin{lemm} Let $\frak{c}_1^2\frak{c}_2=\frak{c}\subseteq\frak{o}$, where $\frak{c}_2$ is squarefree. Then for $1\leq X\in\RR$, $1\leq P\in\ZZ$,
\begin{equation*}
\int_{\substack{\varpi\in\mathcal{E}_{1}(\frak{c})\\ |\nu_{\varpi,j}|\leq X\\ |p_{\varpi,j}|\leq P}} 1 d\varpi\ll_F X^{r+s}P^s\norm{\frak{c}_1}.
\end{equation*}
\end{lemm}
\begin{proof} Any Hecke character $\chi$ can be factorized as $\chi=\chi_{\infty}\chi_{\fin}$. Here, $\chi_{\fin}|_{\prod_{\frak{p}}\frak{o}_{\frak{p}}^{\times}}$ is a character of $\prod_{\frak{p}}\frak{o}_{\frak{p}}^{\times}$. By Proposition \ref{conductorofeisensteinseries}, $\frak{c}_{\chi}|\frak{c}_1$, so there are at most $\varphi(\frak{c}_1)$ possibilities for this restriction. Assume given a character $\xi$ of $\prod_{\frak{p}}\frak{o}_{\frak{p}}^{\times}$, we estimate the measure of the set $S$ of those Hecke characters $\chi$ for which $\chi_{\fin}|_{\prod_{\frak{p}}\frak{o}_{\frak{p}}^{\times}}=\xi$. If $S=\emptyset$, the measure is $0$. If $S\neq\emptyset$, fix some $\chi_0\in S$. Then to any $\chi$ in $S$, associate $\chi'=\chi\chi_0^{-1}$. From the non-archimedan part, we see $\chi'$ is trivial on $\prod_{\frak{p}}\frak{o}_{\frak{p}}^{\times}$. From the archimedan part, we see that for $a\in F_{\infty,+}^{\times}$, $\chi'(a_j)=|a_j|^{t_j}$, if $j\leq r$, and $\chi'(a_j)=|a_j|^{t_j}(a_j/|a_j|)^{p_j}$, if $j>r$, where $t_j\in i[-2X,2X]$, and $p_j\in[-2P,2P]\cap \ZZ$. Fix the vector $(p_j)_{j>r}\in [-2P,2P]^{s}\cap \ZZ^s$.

Now $\chi'_{\infty}$ is trivial on the group $U^+$ of totally positive units embedded in $F_{\infty,+}^{\times}$. Fix a generating set $\{u_1,\ldots,u_{r+s-1}\}$ for the torsion-free part of $U^+$. Then by the notation of \cite{BlomerHarcos}, take
\begin{equation*}
M=\begin{pmatrix}\deg[F_1:\RR] & \ldots & \deg[F_{r+s}:\RR] \cr \deg[F_1:\RR]\log |u_{1,1}| & \ldots & \deg[F_{r+s}:\RR]\log |u_{1,r+s}| \cr \cdot & \ & \cdot \cr \cdot & \ & \cdot \cr \cdot & \ & \cdot \cr \deg[F_1:\RR]\log |u_{r+s-1,1}| & \ldots & \deg[F_{r+s}:\RR]\log |u_{r+s-1,r+s}|\end{pmatrix}\in\RR^{(r+s)\times(r+s)}.
\end{equation*}
Then the column vector $t=(t_j)_j\in i[-2X,2X]^{r+s}$ with $iT=\sum_j\deg[F_j:\RR]t_j$ satisfies $Mt\in i\{T\}\times(2\pi i\ZZ)^{r+s-1}$. Using that $M$ is invertible and depends only on $F$, we see
\begin{equation*}
\int_{-2(r+2s)X}^{2(r+2s)X}\#((\{T\}\times(2\pi i\ZZ)^{r+s-1})\cap Mi[-2X,2X]^{r+s}) dT\ll_F X^{r+s},
\end{equation*}
since the integrand is $O_F(X^{r+s-1})$. Taking into account the finiteness of the torsion subgroup of $U^+$ and of $\lfact{F^{\times}F^{\times}_{\infty,+}\prod_{\frak{p}}\frak{o}_{\frak{p}}^{\times}}{\AAA^{\times}}$, finally summing over $(p_j)_{j>r}\in[-2P,2P]^s$, we obtain the statement.
\end{proof}

\begin{coro}\label{densityofeisensteinspectrum2} Let $\frak{c}_1^2\frak{c}_2=\frak{c}\subseteq\frak{o}$, where $\frak{c}_2$ is squarefree. Then for $1\leq X\in\RR$,
\begin{equation*}
\int_{\substack{\varpi\in\mathcal{E}_{1}(\frak{c})\\ j\leq r:|\nu_{\varpi,j}|\leq X\\ j>r: |\nu_{\varpi,j}^2-p_{\varpi,j}^2|\leq X^2}} 1 d\varpi\ll_F X^{r+2s}\norm{\frak{c}_1}.
\end{equation*}
\end{coro}

\paragraph{Density of the cuspidal spectrum}

Using the Kuznetsov formula, we may estimate the density of the cuspidal spectrum as follows.

\begin{lemm}\label{densityofcuspidalspectrumgeneral} Let $\frak{c}\subseteq\frak{o}$ be an ideal. Then for $1\leq X_j\in\RR^{r+s}$,
\begin{equation*}
\begin{split}
& \sum_{\substack{\varpi\in\mathcal{C}_{1}(\frak{c})\\ j\leq r:|\nu_{\varpi,j}|\leq X_j\\ j>r: |\nu_{\varpi,j}^2-p_{\varpi,j}^2|\leq X_j^2}} \sum_{\frak{t}|\frak{cc}_{\varpi}^{-1}}  |\lambda_{\varpi}^{\frak{t}}(\frak{m})|^2 \ll_{F,\varepsilon}\\ & \left(\prod_{j\leq r} X_j^{2+\varepsilon}\right)\left(\prod_{j>r} X_j^{4+\varepsilon}\right)\norm{\frak{c}}^{1+\varepsilon}+ \left(\prod_j X_j^{2+\varepsilon}\right)(\norm{\gcd(\frak{m},\frak{c})})^{1/2}\norm{\frak{m}}^{1/2+\varepsilon}.
\end{split}
\end{equation*}
\end{lemm}
\begin{proof} This is the generalization of \cite[Lemma 6]{BlomerHarcos}, we can repeat its proof. Choose a narrow class representative $\frak{n}$ of $\frak{m}^{-1}$ from a fixed set of narrow class representatives. Then for some $\alpha\in F^{\times}$, $\frak{m}=\alpha\frak{n}^{-1}$, and $1\ll_F \norm{(\alpha)}/\norm{\frak{m}}\ll_F 1$. We apply the Kuznetsov formula (\ref{kuznetsovformula}) with $\alpha=\alpha'$, $\frak{a}=\frak{a}'=\frak{n}$, and the weight function is the one described above, setting $a_j=X_j^2$, $b_j=X_j$ at each archimedean place. On the spectral side of the Kuznetsov formula, we obtain an upper bound on the left-hand side of the statement, since the contribution of the Eisenstein spectrum is nonnegative. For $\varpi\in\mathcal{C}_{1}(\frak{c})$, by (\ref{indexofcongruencesubgroup}), (\ref{proportionalityconstantisessentiallylfunctionresidue}) and Proposition \ref{residueofrankinselbergsquarelfuncionisaboutone}, $[K(\frak{o}):K(\frak{c})]C_{\varpi}\ll_{F,\varepsilon} (\prod_jX_j)^{\varepsilon}\norm{\frak{c}}^{1+\varepsilon}$. Then by (\ref{estimateoftransformsongeometricside}), the delta term gives $\ll_{F,\varepsilon}\bigl(\prod_{j\leq r} X_j^{2+\varepsilon}\bigr)\bigl(\prod_{j>r} X_j^{4+\varepsilon}\bigr)\norm{\frak{c}}^{1+\varepsilon}$. As for the Kloosterman term, we use Weil's bound \cite[(13)]{Venkatesh} together with (\ref{estimateoftransformsongeometricside}) to see it is
\begin{equation}\label{densityofcuspidalspectrumgeneralkloostermanterm}
\begin{split}
\ll_{F,\varepsilon} \left(\prod_j X_j^{2+\varepsilon}\right)\norm{\frak{c}}^{1+\varepsilon}\max_{\frak{a}\in C} \sum_{0\neq c\in \frak{nac}} & \frac{\norm{(\gcd(\frak{m},c\frak{n}^{-1}\frak{a}^{-1}))}^{1/2}}{ \norm{c\frak{n}^{-1}\frak{a}^{-1}}^{1/2-\varepsilon}} \\ & \cdot \prod_{j\leq r} \min(1,|\alpha_j/c_j|^{1/2}) \prod_{j>r} \min(1,|\alpha_j/c_j|),
\end{split}
\end{equation}
where $C$ is a fixed set of narrow class representatives (depending only on $F$) such that $\frak{a}^2$ is a totally positive ideal for each $\frak{a}\in C$. Then sum over the elements $c$ can be rewritten to a sum over the principal ideals $(c)$, the sum over the units is estimated in \cite[Lemma 8.1]{BruggemanMiatello}. Then the above display is
\begin{equation*}
\begin{split}
\ll_{F,\varepsilon} \left(\prod_j X_j^{2+\varepsilon}\right)\norm{\frak{c}}^{1+\varepsilon} \max_{\frak{a}\in C }&\sum_{0\neq (c)\subseteq \frak{nac}}  \frac{\norm{\gcd(\frak{m},c\frak{n}^{-1}\frak{a}^{-1})}^{1/2}}{ \norm{c\frak{n}^{-1}\frak{a}^{-1}}^{1/2-\varepsilon}} \\ & \cdot (1+|\log(\norm{(\alpha/c)})|^{r+s-1})\min(1,\norm{(\alpha/c)}).
\end{split}
\end{equation*}
This is obviously
\begin{equation*}
\ll_{F,\varepsilon} \left(\prod_j X_j^{2+\varepsilon}\right)\norm{\frak{c}}^{1+\varepsilon} \norm{\frak{m}}^{1/2+2\varepsilon}\max_{\frak{a}\in C}\sum_{0\neq (c)\subseteq \frak{nac}}\frac{\norm{(\gcd(\frak{m},c\frak{n}^{-1}\frak{a}^{-1}))}^{1/2} }{\norm{(c)}^{1/2+\varepsilon}}.
\end{equation*}
We estimate now the sum. First extend it to all nonzero ideals contained in $\frak{nac}$ (parametrized as $\frak{bnac}$, where $0\neq\frak{b}\subseteq\frak{o}$), then factorize out $\norm{\gcd(\frak{m},\frak{c})}^{1/2}$. We obtain
\begin{equation*}
\frac{1}{\norm{\frak{nac}}^{1+\varepsilon}} \sum_{\frak{b}\subseteq\frak{o}} \frac{\norm{\gcd(\frak{m},\frak{cb})}^{1/2}}{\norm{\frak{b}}^{1+\varepsilon}} \ll_{F,\varepsilon} \frac{\norm{\gcd(\frak{m},\frak{c})}^{1/2}}{\norm{\frak{nac}}^{1+\varepsilon}} \norm{\frak{m}}^{\varepsilon}.
\end{equation*}
Altogether, the contribution of the Kloosterman term is
\begin{equation}\label{densityofcuspidalspectrumgeneralkloostermancontribution}
\ll_{F,\varepsilon} \left(\prod_j X_j^{2+\varepsilon}\right) \norm{\gcd(\frak{m},\frak{c})}^{1/2}\norm{\frak{m}}^{1/2+\varepsilon}.
\end{equation}
Recall that the contribution of the Eisenstein spectrum is nonnegative. \end{proof}

\begin{coro}\label{densityofcuspidalspectrum2}
Let $\frak{c}\subseteq\frak{o}$ be an ideal. Then for $1\leq X_j\in\RR^{r+s}$,
\begin{equation*}
\sum_{\substack{\varpi\in\mathcal{C}_{1}(\frak{c})\\ j\leq r:|\nu_{\varpi,j}|\leq X_j\\ j>r: |\nu_{\varpi,j}^2-p_{\varpi,j}^2|\leq X_j^2}} \sum_{\frak{t}|\frak{cc}_{\pi}^{-1}} 1 \ll_{F,\varepsilon} \left(\prod_{j\leq r} X_j^{2+\varepsilon}\right)\left(\prod_{j>r} X_j^{4+\varepsilon}\right)\norm{\frak{c}}^{1+\varepsilon}.
\end{equation*}
\end{coro}

\section{$L$-functions}\label{chap:lfunctions}

\subsection{The constant term of an Eisenstein series}

In this section, we follow \cite[Section 2.8]{BlomerHarcos}. Again, we pay special attention to the complex place, which is not covered there.

For some $s\in\CC$, consider the Hecke quasicharacter $\chi(y)=|y|^s$ for $y\in\AAA^{\times}$. Taking also some nonzero ideal $\frak{c}\subseteq\frak{o}$, define the function $\varphi(s)\in H(\chi,\chi^{-1})$ as
\begin{equation*}
\varphi\left(s,\begin{pmatrix} a & x \cr 0 & b \end{pmatrix}k\right)=\left\{ \begin{array}{ll} |a/b|^{1/2+s}, & k\in K_{\infty}\times K(\frak{c}), \cr 0, & k\in K\setminus(K_{\infty}\times K(\frak{c})). \end{array} \right.
\end{equation*}
The constant term \cite[p.220]{GelbartJacquet} of the corresponding Eisenstein series $E(\varphi(s),g)$ is
\begin{equation}\label{constanttermofspecialeisensteinseries}
E_0(\varphi(s),g)=\varphi(s,g)+ \int_{\AAA}\varphi\left(s,\begin{pmatrix} 0 & -1 \cr 1 & \xi\end{pmatrix}g \right)d\xi.
\end{equation}

\begin{prop}
\begin{equation*}
\int_{\AAA}\varphi\left(s,\begin{pmatrix} 0 & -1 \cr 1 & \xi\end{pmatrix}g \right)d\xi=\frac{\Lambda_F(2s)}{\Lambda_F(2s+1)}H(s,g),
\end{equation*}
where
\begin{equation*}
\Lambda_F(s)=|D_F|^{s/2} \prod_{v\cong\RR}\left(\pi^{-s/2}\Gamma(s/2)\right) \prod_{v\cong\CC}\left(2(2\pi)^{-s}\Gamma(s)\right) \prod_{\frak{p}}(1-\norm{p}^{-s})^{-1}
\end{equation*}
for $\Re s>1$, and $H(s,g)$ is a meromorphic function of $s$, its zeros lie on $\Re s=0$, its poles on $\Re s=-1/2$ and it is constant at $s=1/2$:
\begin{equation*}
H(1/2,g)=|\delta|\norm{\frak{c}}^{-1} \prod_{\frak{p}|\frak{c}}(1+\norm{\frak{p}}^{-1})^{-1}= |D_F|^{-1}[K(\frak{o}):K(\frak{c})]^{-1},
\end{equation*}
where $\delta$ is a finite representing idele for $\frak{d}$.
\end{prop}
\begin{proof}
We may write
\begin{equation*}
g=\begin{pmatrix} a & x \cr 0 & b\end{pmatrix}h,\qquad x\in\AAA,a,b\in\AAA^{\times}, h\in\GLtwo{2}{\AAA},
\end{equation*}
where $h_{\infty}\in K_{\infty}$, $h_{\frak{p}}\in K(\frak{o}_{\frak{p}})$ for $\frak{p}\nmid\frak{c}$ and for $\frak{p}|\frak{c}$, $h_{\frak{p}}\in\GLtwo{2}{F_{\frak{p}}}$ is of the form $\bigl(\begin{smallmatrix}1 & 0 \cr 0 & 1\end{smallmatrix}\bigr)$ or $\bigl(\begin{smallmatrix} 0 & -\delta_{\frak{p}}^{-1} \cr \delta_{\frak{p}} & \eta_{\frak{p}}\end{smallmatrix}\bigr)$.

Then our integral becomes
\begin{equation*}
\left|\frac{a}{b}\right|^{1/2-s}|\delta|^{2s} \int_{\AAA}\varphi\left(s,\begin{pmatrix} 0 & -\delta^{-1} \cr \delta & \xi\end{pmatrix}h\right)d\xi,
\end{equation*}
which can be computed as the product of the corresponding local integrals. These are given at \cite[pp.22-24]{BlomerHarcos} for real and non-archimedean places.

First assume $\frak{p}$ is a non-archimedean place. If $\frak{p}\nmid \frak{c}$, then the local integral is
\begin{equation*}
\int_{F_{\frak{p}}}\varphi_{\frak{p}}\left(s,\begin{pmatrix} 0 & -\delta_{\frak{p}}^{-1} \cr \delta_{\frak{p}} & \xi\end{pmatrix}h_{\frak{p}}\right)d\xi= \frac{1-\norm{\frak{p}}^{-1-2s}}{1-\norm{\frak{p}}^{-2s}}.
\end{equation*}
If $\frak{p}|\frak{c}$, then we have two cases:
\begin{equation*}
\int_{F_{\frak{p}}}\varphi_{\frak{p}} \left(s,\begin{pmatrix} 0 & -\delta_{\frak{p}}^{-1} \cr \delta_{\frak{p}} & \xi \end{pmatrix}\right)d\xi= \norm{\frak{p}}^{-2sv_{\frak{p}}(\frak{c})} \frac{1-\norm{\frak{p}}^{-1}}{1-\norm{\frak{p}}^{-2s}},
\end{equation*}
and
\begin{equation*}
\int_{F_{\frak{p}}}\varphi_{\frak{p}} \left(s,\begin{pmatrix} 0 & -\delta_{\frak{p}}^{-1} \cr \delta_{\frak{p}} & \xi \end{pmatrix} \begin{pmatrix} 0 & -\delta_{\frak{p}}^{-1} \cr \delta_{\frak{p}} & \eta_{\frak{p}} \end{pmatrix} \right)d\xi= \norm{\frak{p}}^{-v_{\frak{p}}(\frak{c})}.
\end{equation*}

Now assume $v$ is a real place, then the local integral is
\begin{equation*}
\int_{\RR}\varphi_j\left(s,\begin{pmatrix} 0 & -1 \cr 1 & \xi\end{pmatrix}h_j\right)d\xi=\frac{\Gamma(1/2)\Gamma(s)}{\Gamma(1/2+s)}.
\end{equation*}

Finally, assume $v\cong \CC$. Using the Iwasawa decomposition
\begin{equation*}
\begin{pmatrix} 0 & -1 \cr 1 & \xi \end{pmatrix}= \begin{pmatrix} \frac{1}{\sqrt{|\xi|^2+1}} & \frac{-\overline{\xi}}{\sqrt{|\xi|^2+1}} \cr 0 & \sqrt{|\xi|^2+1} \end{pmatrix} \begin{pmatrix} \frac{\overline{\xi}}{\sqrt{|\xi|^2+1}} & \frac{-1}{\sqrt{|\xi|^2+1}} \cr \frac{1}{\sqrt{|\xi|^2+1}} & \frac{\xi}{\sqrt{|\xi|^2+1}} \end{pmatrix},
\end{equation*}
we can compute that the local integral is
\begin{equation*}
\int_{\CC} \frac{1}{(1+|\xi|^2)^{1+2s}}d\xi =2\pi\frac{\Gamma(1)\Gamma(2s)}{\Gamma(1+2s)}.
\end{equation*}
As for
\begin{equation}\label{indexofcongruencesubgroup}
\norm{\frak{c}}^{-1} \prod_{\frak{p}|\frak{c}}(1+\norm{\frak{p}}^{-1})^{-1}= [K(\frak{o}):K(\frak{c})]^{-1},
\end{equation}
consult \cite[Proposition 2.5]{IwaniecTopics}. Collecting these, the proof is complete.
\end{proof}

\subsection{A Rankin-Selberg convolution}\label{sec:lfunctions_rankinselbergconvolution}

Earlier, we referred to this section twice: in the construction of the isometries $R^{\frak{t}}$ (\ref{orthogonaloldforms}) and in the proof of Proposition \ref{archimedeankirillovmodelproposition}. Now we borrow the Rankin-Selberg method from \cite[pp.25-26]{BlomerHarcos} in order to prove the essential equivalence of the Kirillov model promised earlier (i.e. to complete the proof of Proposition \ref{archimedeankirillovmodelproposition}), and also to relate the proportionality constant to the residue of a certain $\GLone{2}\times\GLone{2}$ $L$-function. We will also obtain that for $\phi_1,\phi_2\in V_{\pi,\ww}(\frak{c}_{\pi})$, $\langle R_{\frak{t}_1}\phi_1,R_{\frak{t}_2}\phi_2\rangle= \langle \phi_1,\phi_2 \rangle C(\frak{t}_1,\frak{t}_2,\pi)$ with a constant $C(\frak{t}_1,\frak{t}_2,\pi)$ independent of $\ww$, this was the fact used in the construction of $R^{\frak{t}}$.

Let $\phi_1,\phi_2\in V_{\pi,\ww}$ be newforms of some weight $\ww\in W(\pi)$ and let $\frak{t}_1,\frak{t}_2\subseteq\frak{o}$ be nonzero ideals. If $\frak{c}$ is a nonzero ideal divisible by $\frak{t}_1\frak{c}_{\pi}$, $\frak{t}_2\frak{c}_{\pi}$, then $\psi_1=R_{\frak{t}_1}\phi_1$, $\psi_2=R_{\frak{t}_2}\phi_2$ are elements in $V_{\pi,\ww}(\frak{c})$.

Define
\begin{equation*}
F(s)=\int_{\lfact{\GLtwo{2}{F}Z(\AAA)}{\GLtwo{2}{\AAA}}} \psi_1(g)\overline{\psi_2(g)}E(\varphi(s),g)dg,
\end{equation*}
where $\varphi(s,g)$ is defined in the previous section. It follows from the theory of Eisenstein series that this integral is absolutely convergent for all $s$ which is not a pole of $E(\varphi(s),g)$ (see \cite[Section 5]{GelbartJacquet}), and also that the possible residue comes from the residue of the constant term (\ref{constanttermofspecialeisensteinseries}). Now we compute $\res_{s=1/2}F(s)$ in two ways.

On the one hand, using the results of the previous section,
\begin{equation}\label{residueinrankinselberg1}
\res_{s=1/2}F(s)=C_F\frac{\langle R_{\frak{t}_1}\phi_1,R_{\frak{t}_2}\phi_2 \rangle}{[K(\frak{o}):K(\frak{c})]},\qquad C_F=\frac{\res_{s=1/2}\Lambda_F(2s)}{|D_F|\Lambda_F(2)}.
\end{equation}

On the other hand, assume first $\Re s>1/2$ for the absolute convergence of (\ref{definitionofeisensteinseries}) (see \cite[p.372]{Bump}) and unfold the integral
\begin{equation*}
\begin{split}
F(s)&= \int_{\lfact{B(F)Z(\AAA)}{\GLtwo{2}{\AAA}}} \psi_1(g)\overline{\psi_2(g)} \varphi(s,g) dg\\ &= \int_{\lfact{F^{\times}}{\AAA^{\times}}} \int_{\lfact{F}{\AAA}} \int_K \psi_1\left(\begin{pmatrix}y & x \cr 0 & 1\end{pmatrix}k\right) \overline{\psi_2\left(\begin{pmatrix}y & x \cr 0 & 1\end{pmatrix}k\right)} \varphi\left(s,\begin{pmatrix}y & x \cr 0 & 1\end{pmatrix}k\right)dkdx\frac{d^{\times}y}{|y|}\\ &= \int_{\lfact{F^{\times}}{\AAA^{\times}}} \int_{\lfact{F}{\AAA}} \int_{K_{\infty}\times K(\frak{c})} \psi_1\left(\begin{pmatrix}y & x \cr 0 & 1\end{pmatrix}k\right) \overline{\psi_2\left(\begin{pmatrix}y & x \cr 0 & 1\end{pmatrix}k\right)} |y|^{s-1/2} dkdxd^{\times}y.
\end{split}
\end{equation*}
Here, the integral over $K_{\infty}\times K(\frak{c})$ is $[K(\frak{o}):K(\frak{c})]^{-1}$. To see this, observe that $\psi_1\overline{\psi_2}$ is invariant at real and non-archimedean places, while at complex places, we apply the more general \cite[Corollary 1.10(b)]{Knapp}. Therefore,
\begin{equation*}
F(s)=\frac{1}{[K(\frak{o}):K(\frak{c})]} \int_{\lfact{F^{\times}}{\AAA^{\times}}} \int_{\lfact{F}{\AAA}} \psi_1\overline{\psi_2}\left(\begin{pmatrix} y & x \cr 0 & 1 \end{pmatrix}\right)|y|^{s-1/2}dxd^{\times}y.
\end{equation*}

Take now finite representing ideles $t_1,t_2$ of the ideals $\frak{t}_1,\frak{t}_2$, respectively. The Fourier-Whittaker expansion (\ref{fourierwhittakerexpansion2newforms}), the definition of $R_{\frak{t}}$ (\ref{raisinglevel}) and $\vol(\lfact{F}{\AAA})=1$ give rise to
\begin{equation*}
\begin{split}
F(s)= \frac{\norm{\frak{t}_1\frak{t}_2}^{1/2}}{[K(\frak{o}):K(\frak{c})]} &\int_{\AAA^{\times}} \frac{\lambda_{\pi}(y_{\fin}t_1^{-1}) \overline{\lambda_{\pi}(y_{\fin}t_2^{-1})} }{\norm{y_{\fin}}} W_{\phi_1}(y_{\infty}) \overline{W_{\phi_2}(y_{\infty})}|y|^{s-1/2}d^{\times}y\\
= \frac{\norm{\frak{t}_1\frak{t}_2}^{1/2}}{[K(\frak{o}):K(\frak{c})]} &\int_{F_{\infty}^{\times}} W_{\phi_1}(y_{\infty})\overline{W_{\phi_2}(y_{\infty})} |y_{\infty}|^{s-1/2}d^{\times}y_{\infty} \\ \cdot &\int_{\AAA_{\fin}^{\times}} \frac{\lambda_{\pi}(y_{\fin}t_1^{-1}) \overline{\lambda_{\pi}(y_{\fin}t_2^{-1})} }{\norm{y_{\fin}}^{1/2+s}}d^{\times}y_{\fin}. 
\end{split}
\end{equation*}
Let now $s\rightarrow 1/2$ from above, then the first integral is $\langle W_{\phi_1},W_{\phi_2} \rangle$, where the inner product is understood in $L^2(F_{\infty}^{\times},d_{\infty}^{\times}y)$. In the second integral, define $\frak{t}_1'=\frak{t}_1\gcd(\frak{t}_1,\frak{t}_2)^{-1}$, $\frak{t}_2'=\frak{t}_2\gcd(\frak{t}_1,\frak{t}_2)^{-1}$, we obtain
\begin{equation}\label{residueinrankinselberg2}
\res_{s=1/2}F(s)= \frac{\langle W_{\phi_1},W_{\phi_2} \rangle}{\norm{\frak{t}_1'\frak{t}_2'}^{1/2}[K(\frak{o}):K(\frak{c})]} \res_{s=1} \sum_{0\neq\frak{m}\subseteq{\frak{o}}} \frac{\lambda_{\pi}(\frak{mt}_2') \overline{\lambda_{\pi}(\frak{mt}_1')}}{\norm{\frak{m}}^s}
\end{equation}
by a linear change of variable $\frak{m}=y_{\fin}t_1t_2\gcd(\frak{t}_1,\frak{t}_2)^{-1}$.

For arbitrary ideals $\frak{t}_1,\frak{t}_2$, this gives
\begin{equation*}
\langle R_{\frak{t}_1}\phi_1,R_{\frak{t}_2}\phi_2\rangle= \langle \phi_1,\phi_2 \rangle C(\frak{t}_1,\frak{t}_2,\pi),
\end{equation*}
where $C(\frak{t}_1,\frak{t}_2,\pi)$ is a constant not depending on the weight $\ww$. This independence of the weight is essential in the construction of $R^{\frak{t}}$ (\ref{orthogonaloldforms}) as we indicated it earlier.

Using the equations (\ref{residueinrankinselberg1}) and (\ref{residueinrankinselberg2}) about $\res_{s=1/2}F(s)$, and taking $\frak{t}_1=\frak{t}_2=\frak{o}$, we obtain (\ref{essentialisometryofkirillovmodel}) with
\begin{equation*}
C_{\pi}=\frac{|D_F|\Lambda_F(2)}{\res_{s=1/2}\Lambda_F(2s)}\res_{s=1} \sum_{0\neq\frak{m}\subseteq{\frak{o}}} \frac{\lambda_{\pi}(\frak{m}) \overline{\lambda_{\pi}(\frak{m})}}{\norm{\frak{m}}^s}.
\end{equation*}
Here, the first factor $|D_F|\Lambda_F(2)/\res_{s=1/2}\Lambda_F(2s)$ is a positive constant depending only on $F$, while
\begin{equation*}
L^{\frak{c}_{\pi}}(s,\pi\times\pi) \zeta_F(2s) \sum_{0\neq\frak{m}\subseteq{\frak{o}}} \frac{\lambda_{\pi}(\frak{m}) \overline{\lambda_{\pi}(\frak{m})}}{\norm{\frak{m}}^s}= L(s,\pi\times\pi)
\end{equation*}
with $L(s,\pi\times\pi)$ defined in \cite[Sections 1-2]{GelbartJacquetgl3} and $L^{\frak{c}_{\pi}}(s,\pi\times\pi)$ is a finite Euler product over places dividing $\frak{c}_{\pi}$, the number of such places is $O_{F,\varepsilon}(\norm{\frak{c}_{\pi}}^{\varepsilon})$. Checking the cases from \cite[Section 1]{GelbartJacquetgl3} and \cite[Chapter I, \S\S{}2-3]{JacquetLanglands}), we obtain
\begin{equation}\label{proportionalityconstantisessentiallylfunctionresidue}
\norm{\frak{c}_{\pi}}^{-\varepsilon}\res_{s=1}L(s,\pi\times\pi) \ll_{F,\varepsilon} C_{\pi} \ll_{F,\varepsilon} \norm{\frak{c}_{\pi}}^{\varepsilon} \res_{s=1}L(s,\pi\times\pi).
\end{equation}
By \cite[Lemma b]{HoffsteinRamakrishnan}, we have a constant $B$ depending only on $F$ such that
\begin{equation*}
C(\pi)^{-B}\leq C(\pi\times\pi)\leq C(\pi)^B
\end{equation*}
holds for the analytic conductors. For later references, we also record
\begin{equation}\label{boundonweightedsumofsquaresofheckeeigenvalues}
\sum_{\norm{\frak{m}}\leq x} \frac{|\lambda_{\pi}(\frak{m})|^2}{\norm{\frak{m}}}\ll_{F,\varepsilon} C(\pi)^{B'}x^{\varepsilon}
\end{equation}
with some $B'$ depending only on $F$, which follows from the upper bound of the above display by a contour integration similar to the one in \cite[Proof of Lemma 2.1]{HoffsteinLockhart}.

\begin{prop}\label{residueofrankinselbergsquarelfuncionisaboutone} We have
\begin{equation*}
C(\pi)^{-\varepsilon}\ll_{F,\varepsilon} \res_{s=1} L(s,\pi\times\pi) \ll_{F,\varepsilon} C(\pi)^{\varepsilon},
\end{equation*}
recall (\ref{definitionofanalyticconductor}).
\end{prop}
\begin{proof} See \cite[Lemma 3]{BlomerHarcos} or \cite[Proposition 3.2]{MPphd}.
\end{proof}

\section{Sobolev norms}\label{chap:sobolevnorms}

Assume that we are given a smooth automorphic vector $\phi$ appearing in an automorphic representation. The aim of this section is to give a pointwise estimate for the associated Kirillov vector $W_{\phi}$, and, when $\phi$ is a cuspidal newform, the supremum norm of $\phi$, both in terms of some Sobolev norm of $\phi$.

Let $d\geq 0$ be an integer. Assume that $\phi\in L^2(\lfact{\GLtwo{2}{F}}{\GLtwo{2}{\AAA}},\omega)$ is a function such that $X_1\ldots X_d \phi$ exists for every sequence $X_1,\ldots, X_d$, where each $X_k$ is one of those differential operators given in (\ref{differentialoperatorbasisreal}) and (\ref{differentialoperatorbasiscomplex}). Then the Sobolev norm $||\phi||_{S_d}$ of $\phi$ is defined via
\begin{equation*}
||\phi||_{S_d}^2=\sum_{k=0}^{d} \sum_{\{X_1,\ldots,X_k\}\in\{ \mathbf{H}_j,\mathbf{R}_j,\mathbf{L}_j,\mathbf{H}_{1,j}, \mathbf{H}_{2,j},\mathbf{V}_{1,j},\mathbf{V}_{2,j}, \mathbf{W}_{1,j},\mathbf{W}_{2,j}\}^k} ||X_1\ldots X_k\phi||^2.
\end{equation*}

\subsection{Bounds on Bessel functions}

About the classical $J$-Bessel function of parameter $p\in\ZZ/2$, record the bounds
\begin{equation}\label{jbesselestimateformula}
|J_{2p}(x)|\leq 1\ \mathrm{for}\ \mathrm{all}\ x\in(0,\infty), \qquad |J_{2p}(x)|\ll x^{-1/2}\ \mathrm{for}\ \mathrm{all}\ x\in(\max(1/2,(2p)^2),\infty),
\end{equation}
see \cite[2.2(1)]{Watson} and \cite[8.451(1-8)]{GR}.

Now we define and estimate a function $j$ that later will turn out to be the Bessel function of a certain representation (after a simple transformation of the argument).

\begin{lemm}\label{representationbesselestimatelemma}
Assume $\nu\in\CC$ and $p\in\ZZ/2$ are given such that either $\Re\nu=0$ (principal series) or $\Re\nu\neq 0$, $\Im\nu=0$, $|\nu|\leq 2\theta= 7/32$, $p=0$ (complementary series). Define
\begin{equation}\label{definitionofrepresentationbessel}
j(t)= 4\pi|t|^2 \int_0^{\infty}y^{2\nu} \left(\frac{yt+y^{-1}\overline{t}}{|yt+y^{-1}\overline{t}|}\right)^{2p} J_{2p}(2\pi|yt+y^{-1}\overline{t}|)d_{\RR}^{\times}y.
\end{equation}
Then $j(t)$ is an even function of $t\in\CC^{\times}$ satisfying the bound
\begin{equation}\label{representationbesselestimateformula}
j(t)\ll |t|^2(1+|t|^{-1/2})(1+|p|).
\end{equation}
\end{lemm}
\begin{proof} It is clear that $j(t)=j(-t)$, so we are left to prove (\ref{representationbesselestimateformula}). Assume first that $p\neq 0$, which implies that we are in the principal series. Then trivially
\begin{equation*}
j(t)\ll |t|^2 \int_0^{\infty} |J_{2p}(2\pi|yt+y^{-1}\overline{t}|)|d_{\RR}^{\times}y.
\end{equation*}
The integral is invariant under $y\leftrightarrow 1/y$, so we have
\begin{equation*}
j(t)\ll |t|^2 \int_1^{\infty} |J_{2p}(2\pi|yt+y^{-1}\overline{t}|)|d_{\RR}^{\times}y.
\end{equation*}
Here
\begin{equation*}
\int_1^2 |J_{2p}(2\pi|yt+y^{-1}\overline{t}|)|d_{\RR}^{\times}y\ll 1
\end{equation*}
and
\begin{equation*}
\int_2^{\max\left(\frac{4p^2}{\pi |t|},2\right)} |J_{2p}(2\pi|yt+y^{-1}\overline{t}|)|d_{\RR}^{\times}y\ll \max\left(\log\left(\frac{4p^2}{\pi |t|}\right),0\right)
\end{equation*}
by $|J_{2p}(x)|\leq 1$ of (\ref{jbesselestimateformula}). On the remaining domain, $y\geq 2$, hence $|yt+y^{-1}\overline{t}|\geq y|t|/2$. Moreover, since $y\geq 4p^2/(\pi |t|)$, we have $2\pi|yt+y^{-1}\overline{t}|\geq (2p)^2>1/2$, so we may apply $|J_{2p}(x)|\ll x^{-1/2}$ of (\ref{jbesselestimateformula}), obtaining
\begin{equation*}
\int_{\max\left(\frac{4p^2}{\pi |t|},2\right)}^{\infty} |J_{2p}(2\pi|yt+y^{-1}\overline{t}|)|d_{\RR}^{\times}y \ll 1+|t|^{-1/2}.
\end{equation*}
Altogether,
\begin{equation*}
j(t)\ll |t|^2\left(1+|t|^{-1/2}+\max\left(\log\left(\frac{4p^2}{\pi |t|}\right),0\right)\right),
\end{equation*}
which obviously implies
\begin{equation}\label{representationbesselestimateprincipalseries}
j(t)\ll |t|^2(1+|t|^{-1/2})(1+|p|).
\end{equation}

If $p=0$, in particular, in the complementary series, a similar calculation yields (using also that $2|\Re\nu|\leq 7/16$) 
\begin{equation}\label{representationbesselestimatecomplementaryseries}
j(t)\ll |t|^2(1+|t|^{-1/2}).
\end{equation}
(This time the integral might not be invariant under $y\leftrightarrow 1/y$, however, replacing $y^{2\Re\nu}$ by $y^{2|\Re\nu|}$, we may write $\int_1^{\infty}$ in place of $\int_0^{\infty}$; and the domain of integration $[1,\infty]$ is splitted up as $[1,2]\cup[2,\max(1/|t|,2)]\cup[\max(1/|t|,2),\infty]$.)

Collecting the bounds (\ref{representationbesselestimateprincipalseries}), (\ref{representationbesselestimatecomplementaryseries}),
we arrive at (\ref{representationbesselestimateformula}).
\end{proof}

\subsection{Bounds on Whittaker functions}

We would like to give estimates on the Whittaker functions defined in (\ref{definitionofrealwhittakerfunction}) and (\ref{definitionofcomplexwhittakerfunction}). At real places, we refer to \cite{BlomerHarcosDuke}.
\begin{lemm}\label{realwhittakerestimatelemma}
For all $\nu$,
\begin{equation}\label{realwhittakerestimateformula1}
\whit_{q,\nu}(y)\ll |y|^{1/2} \left(\frac{|y|}{|q|+|\nu|+1}\right)^{-1-|\Re\nu|} \exp\left(-\frac{|y|}{|q|+|\nu|+1}\right).
\end{equation}
For $\nu\in (\ZZ/2)\cup i\RR$ and for any $0<\varepsilon<1/4$,
\begin{equation}\label{realwhittakerestimateformula2}
\whit_{q,\nu}(y)\ll_{\varepsilon} |y|^{1/2-\varepsilon}(|q|+|\nu|+1).
\end{equation}
For $\nu\in (-1/2,1/2)$ and for any $0<\varepsilon<1$,
\begin{equation}\label{realwhittakerestimateformula3}
\whit_{q,\nu}(y)\ll_{\varepsilon} |y|^{1/2-|\nu|-\varepsilon}(|q|+|\nu|+1)^{1+|\nu|}.
\end{equation}
\end{lemm}
\begin{proof} See \cite[(24-26)]{BlomerHarcosDuke} (and also \cite[(26-28)]{BlomerHarcos}).
\end{proof}

At complex places, introduce
\begin{equation}\label{jacquettransform0}
\J_{(l,q),(\nu,p)}(y)=\whit_{(l,q),(\nu,p)}(y) \left(\frac{\sqrt{8(2l+1)}}{(2\pi)^{\Re\nu}}{{2l}\choose{l-q}}^{\frac{1}{2}} {{2l}\choose{l-p}}^{-\frac{1}{2}} \sqrt{\left|\frac{\Gamma(l+1+\nu)}{\Gamma(l+1-\nu)}\right|}\right)^{-1},
\end{equation}
the unnormalized Whittaker function appearing in \cite[Section 5]{BruggemanMotohashi} and \cite[Section 4.1]{Lokvenec}; our function $\J_{(l,q),(\nu,p)}(y)$ is the same as $\J_1\varphi_{l,q}(\nu,p)(a(y))$ in \cite{Lokvenec}. The advantage of this unnormalized function is its regularity in $\nu$. Note that $\J_{(l,q),(\nu,p)}$ is nothing else but (\ref{definitionofcomplexwhittakerfunction}) without its first line.
\begin{lemm}\label{complexwhittakerpreliminaryestimatelemma}
For $0<|y|\leq 1$ and $\varepsilon>0$,
\begin{equation}\label{complexwhittakerpreliminaryestimateformula1}
\whit_{(l,q),(\nu,p)}(y)\ll_{\varepsilon} |y|^{1-|\Re\nu|-\varepsilon}(1+|p|+l)^{1+|p|/2}.
\end{equation}
For $|y|\geq (l^4+1)(|\nu|^2+1)$,
\begin{equation}\label{complexwhittakerpreliminaryestimateformula2}
\whit_{(l,q),(\nu,p)}(y)\ll \exp\left(-\frac{|y|}{|\nu|+l+1}\right).
\end{equation}
\end{lemm}
\begin{proof}
It is clear from the definition and the fact $|\Re\nu|\leq 7/32$ that
\begin{equation*}
\whit_{(l,q),(\nu,p)}(y)\ll \J_{(l,q),(\nu,p)}(y)(1+|p|+l)^{1+|p|/2}.
\end{equation*}
Together with \cite[(4.28)]{Lokvenec}, this shows the bound (\ref{complexwhittakerpreliminaryestimateformula1}). As for (\ref{complexwhittakerpreliminaryestimateformula2}), take $|y|\geq (l^4+1)(|\nu|^2+1)\geq 1$. We first estimate $\J_{(l,q),(\nu,p)}$ from its expression in terms of $K$-Bessel functions (recall (\ref{definitionofcomplexwhittakerfunctioncomplement1}) and (\ref{definitionofcomplexwhittakerfunctioncomplement2})). We trivially have
\begin{equation*} \xi_p^l(q,k),(2\pi|y|)^{l+1-k},
(1+l)(1+|p|+l)^{1+|p|/2} \ll e^{|y|/(3(|\nu|+l+1))}
\end{equation*}
for the binomial factor, for the power of $|y|$, and for the summation over $k$ together with the transition factor from $\J_{(l,q),(\nu,p)}$ to $\whit_{(l,q),(\nu,p)}$.
Now we would like to estimate
\begin{equation*}
\frac{K_{\nu+l-|q+p|-k}(4\pi|y|)}{\Gamma(l+1+\nu-k)},
\end{equation*}
where $0\leq k\leq l-\max(|p|,|q|)$. Instead of this, we may write
\begin{equation*}
\frac{K_{\nu+a}(4\pi|y|)}{\Gamma(b+1+\nu)},
\end{equation*}
where $0\leq a\leq b\leq l$: in the principal series $\Re\nu=0$, this is justified by $K_s(x)=K_{-s}(x)$ (see \cite[3.7(6)]{Watson}) and $|\Gamma(x)|=|\Gamma(\overline{x})|$, hence take $b=l-k$, then $a=|l-k-|q+p||$ (and we conjugate $\nu$, if $l-k<|q+p|$), $0\leq a\leq b\leq l$ follows from the constraint on $k$; while in the complementary series, $p=0$ implies $l-|q+p|-k\geq 0$, from which $0\leq a\leq b\leq l$ is satisfied by setting $b=l-k$, $a=l-k-|q+p|$. By Basset's integral \cite[\textsection 6.16]{Watson},
\begin{equation*}
\frac{K_{\nu+a}(4\pi|y|)}{\Gamma(b+1+\nu)}=\frac{\Gamma(\nu+a+1/2)}{\Gamma(\nu+b+1)} \frac{1}{2\sqrt{\pi}(2\pi|y|)^{\nu+a}} \int_{-\infty}^{\infty}\frac{e^{-i4\pi|y|t}}{(1+t^2)^{\nu+a+1/2}}dt.
\end{equation*}
From Stirling's formula, we see that the quotient of the $\Gamma$-factors is $O(1)$. As for the rest, integrating by parts, then shifting the contour to $\Im t=-(|\nu|+a+2)^{-1}$ (similarly as in \cite[(4.2-5)]{BruggemanMotohashinew}),
\begin{equation*}
\frac{1}{2\sqrt{\pi}(2\pi|y|)^{\nu+a}} \int_{-\infty}^{\infty}\frac{e^{-i4\pi|y|t}}{(1+t^2)^{\nu+a+1/2}}dt \ll \frac{|\nu|+a+1}{|y|^{\nu+a-1}} \exp\left(\frac{-(3+1/3)\pi|y|}{|\nu|+a+1}\right).
\end{equation*}
Here, $|\nu|+a+1\ll |y|^{1/2}$, so as above,
\begin{equation*}
|\nu|+a+1\ll e^{|y|/(|\nu|+a+1)},\qquad |y|^{-\nu-a+1}\ll e^{|y|/(3(|\nu|+a+1))},
\end{equation*}
giving
\begin{equation*}
\frac{K_{\nu+l-|q+p|-j}(4\pi|y|)}{\Gamma(l+1+\nu-j)}\ll \exp\left(-\frac{2|y|}{|\nu|+l+1}\right).
\end{equation*}
Altogether
\begin{equation*}
\whit_{(l,q),(\nu,p)}(y)\ll \exp\left(-\frac{|y|}{|\nu|+l+1}\right)
\end{equation*}
as claimed.
\end{proof}

Now borrowing an idea from \cite[p.330]{BlomerHarcosDuke}, we give a further bound on $\whit_{(l,q),(\nu,p)}$.
\begin{lemm}\label{complexwhittakerestimatelemma}
For all $y\in\CC^{\times}$,
\begin{equation}\label{complexwhittakerestimateformula}
\whit_{(l,q),(\nu,p)}(y)\ll (|y|^{3/4}+|y|)(l^4+1)(|\nu|^2+1)(|p|+1).
\end{equation}
\end{lemm}
\begin{proof}
Our starting point is a special Jacquet-Langlands functional equation
\begin{equation}\label{specialjacquetlanglandsfunctionalequation}
\whit_{(l,q),(\nu,p)}(y)=\kappa(p,l,q)\pi  \int_{\CC^{\times}}j(\sqrt{t})\whit_{(l,-q),(\nu,p)}(t/y)d^{\times}_{\CC}t,
\end{equation}
where $j$ is defined in (\ref{definitionofrepresentationbessel}) and $|\kappa(p,l,q)|=1$. This is proved in \cite[Theorem 2 and (3)]{BruggemanMotohashi13} in a different formulation, one is straight-forward from the other using \cite[(2.30), (2.43) and (4.2)]{Lokvenec}. Note that in \cite{BruggemanMotohashi13}, it is stated only for the principal series (i.e. $\Re\nu=0$) and even (i.e. $p\in\ZZ$) representations, but the result extends to the complementary series by analytic continuation, the odd case can be handled similarly (see \cite[p.90]{BruggemanMotohashi13}). Also note that $j(\sqrt{t})$ does not lead to confusion, since $j(t)$ is an even function of $t$ (by Lemma \ref{representationbesselestimatelemma}).

In (\ref{specialjacquetlanglandsfunctionalequation}), split up the integral as
\begin{equation*}
\begin{split}
\whit_{(l,q),(\nu,p)}(y)\ll& \overbrace{\int_{0<|t|<|y|(l^4+1)(|\nu|^2+1)}j(\sqrt{t})\whit_{(l,-q),(\nu,p)}(t/y) d^{\times}_{\CC}t}^{\mathrm{I}}\\ &+\overbrace{\int_{|t|\geq |y|(l^4+1)(|\nu|^2+1)}j(\sqrt{t})\whit_{(l,-q),(\nu,p)}(t/y)d^{\times}_{\CC}t}^{\mathrm{II}}.
\end{split}
\end{equation*}
First estimate $\mathrm{I}$. Using Cauchy-Schwarz,
\begin{equation*}
\begin{split}
\mathrm{I}\ll &\left(\int_{0<|t|<|y|(l^4+1)(|\nu|^2+1)}|j(\sqrt{t})|^2d^{\times}_{\CC}t\right)^{1/2} \\ &\cdot \left(\int_{0<|t|<|y|(l^4+1)(|\nu|^2+1)}|\whit_{(l,-q),(\nu,p)}(t/y)|^2d^{\times}_{\CC}t\right)^{1/2}.
\end{split}
\end{equation*}
The second factor is at most $1$, since the Whittaker functions have $L^2$-norm $1$ (recall (\ref{whittakerfunctionsareorthonormal}) and the remark after that). In the first factor, we may apply (\ref{representationbesselestimateformula}), giving
\begin{equation*}
\mathrm{I} \ll
\max(|y|,|y|^{3/4})(l^4+1)(|\nu|^2+1)(|p|+1). 
\end{equation*}
In the second term $\mathrm{II}$, we apply Lemma \ref{complexwhittakerpreliminaryestimatelemma} together with (\ref{representationbesselestimateformula}). This gives
\begin{equation*}
\mathrm{II} \ll
\max(|y|,|y|^{3/4})(|\nu|+l+1)(|p|+1).
\end{equation*}
Summing up, we arrive at (\ref{complexwhittakerestimateformula}).
\end{proof}

From (\ref{whittakerfunctionsareorthonormal}), we know that the square-integral of a Whittaker function is $1$. The next lemma encapsulates the fact that a Whittaker function cannot concentrate to a neighborhood of $0$ or $\infty$. To formulate it properly, we introduce the notation, for any $a\in \RR^{r+s}$,
\begin{equation*}
S(a)=\left\{y=(y_1,\ldots,y_{r+s})\in\RR^{r+s}: \left\{\begin{array}{ll} |y_j|> |a_j|, & \mathrm{for}\ \mathrm{all}\ j\leq r,\cr y_j> |a_j|, & \mathrm{for}\ \mathrm{all}\ j>r\end{array}\right.\right\}.
\end{equation*}

\begin{lemm}\label{whittakerfunctionsdonotconcentrate} There exist some positive constants $C_0,C_1$ depending only on $F$ and $\rr$ with the following property. For any $t\in F_{\infty}^{\times}$ and $\ww\in W(\pi)$ (where $\pi$ is an automorphic representation with spectral parameter $\rr$), we have
\begin{equation}\label{integralofwhittakerestimateformula}
\int_{S(\varepsilon/t)} |\whit_{\ww,\rr}(ty)|^2\frac{dy}{\prod_{j\leq r}|y_j|^2 \prod_{j>r}|y_j|^3}>C_1|t|_{\infty}\left(\prod_{j\leq r}(1+q_j^2)\prod_{j> r}(1+l_j^4)\right)^{-1},
\end{equation}
if $\varepsilon$ is chosen such that $\varepsilon_j\leq C_0(1+q_j^8)^{-1}$ at real, and $\varepsilon_j\leq C_0(1+l_j^{16})^{-1}$ at complex places.
\end{lemm}
\begin{proof} Observe that the integral on the left-hand side of (\ref{integralofwhittakerestimateformula}) can be written as
\begin{equation*}
|t|_{\infty} \int_{S(\varepsilon)} |\whit_{\ww,\rr}(y)|^2\frac{dy}{\prod_{j\leq r}|y_j|^2 \prod_{j>r}|y_j|^3},
\end{equation*}
so we are left to estimate this. By (\ref{whittakerfunctionsareorthonormal}), we have a positive constant $A$ depending only on $F$ such that
\begin{equation*}
\int_{S(0)} |\whit_{\ww,\rr}(y)|^2\frac{dy}{\prod_{j\leq r}|y_j| \prod_{j>r}|y_j|^2}=A.
\end{equation*}
Now observe that by (\ref{realwhittakerestimateformula2}), (\ref{realwhittakerestimateformula3}) and (\ref{complexwhittakerestimateformula}), for all $0<\varepsilon<1$,
\begin{equation*}
\begin{split}
&\left(\int_{-\varepsilon}^0+\int_0^{\varepsilon}\right) |\whit_{q,\nu}(y)|^2\frac{dy}{|y|}\ll_{F,\nu} \varepsilon^{1/2}(1+q^4),\\ & \int_0^{\varepsilon} |\whit_{(l,q),(\nu,p)}(y)|^2\frac{dy}{|y|^2}\ll_{F,\nu,p} \varepsilon^{1/2}(1+l^8)
\end{split}
\end{equation*}
at real and complex places, respectively (in the real case, use also that $|\Re\nu|\leq 7/64$).
Also by (\ref{realwhittakerestimateformula1}) and (\ref{complexwhittakerpreliminaryestimateformula2}),
\begin{equation*}
\begin{split}
&\left(\int_{-\infty}^{-B(1+q^2)}+\int_{B(1+q^2)}^{\infty}\right) |\whit_{q,\nu}(y)|^2\frac{dy}{|y|}< \frac{1}{2(r+s)},\\ & \int_{B(1+l^4)}^{\infty} |\whit_{(l,q),(\nu,p)}(y)|^2\frac{dy}{|y|^2}< \frac{1}{2(r+s)}
\end{split}
\end{equation*}
for some positive constant $B$ depending on $F$ and $\rr$. Altogether,
\begin{equation*}
\int_{\substack{y\in S(\varepsilon)\\|y_j|<A(1+q_j^2)\ (j\leq r)\\ y_j<A(1+l_j^4)\ (j>r)}} |\whit_{\ww,\rr}(y)|^2\frac{dy}{\prod_{j\leq r}|y_j| \prod_{j>r}|y_j|^2}>C_1A^{r+s}.
\end{equation*}
with some positive number $C_1$ (depending only on $F$ and $\rr$), if $\varepsilon$ is small enough (as in the statement, with an appropriate $C_0$). From this, the statement is obvious.
\end{proof}

\subsection{A bound on the supremum norm of a cusp form}

The aim of this section is to give a bound of the form $||\phi||_{\sup}\ll_{F,\pi} ||\phi||_{S_d}$, where $\phi$ is a sufficiently smooth newform in the cuspidal representation $\pi$, and the order $d$ depends only on $F$. We need some preparatory lemmas.

\begin{lemm}\label{killingcomplexrotations} Assume $(\pi,V_{\pi})$ is an irreducible cuspidal representation, and $\phi\in V_{\pi}$ is of pure weight $\ww$. Then for any $k_{\infty}\in K_{\infty}$ and $g\in\GLtwo{2}{\AAA}$
\begin{equation*}
|\phi(gk_{\infty})|\ll_F |\phi(g)|\prod_{j=r+1}^{r+s}(l_j+1)^7.
\end{equation*}
\end{lemm}
\begin{proof} We may assume $||\phi||=1$. First observe that $\phi'(g)=\phi(gk_{\infty})$ is in the same irreducible representation of $K_{\infty}$ as $\phi$, therefore, we may write
\begin{equation*}
\begin{split}
\phi'(g)=\phi(gk_{\infty})=\phi(g)\sum_{|q_{r+1}|\leq l_{r+q},\ldots,|q_{r+s}|\leq l_{r+s}} & \alpha(g;q_1,\ldots,q_{r+s}) \\ & \cdot \prod_{j=1}^r \Phi_{q_j}(k_j) \prod_{j=r+1}^{r+s} \frac{\Phi_{p_j,q_j}^{l_j}(k_j)}{||\Phi_{p_j,q_j}^{l_j}||_{\mathrm{SU}_2(\CC)}},
\end{split}
\end{equation*}
where for each $g$,
\begin{equation*}
\sum_{|q_{r+1}|\leq l_{r+q},\ldots,|q_{r+s}|\leq l_{r+s}} |\alpha(g;q_1,\ldots,q_{r+s})|^2=1,
\end{equation*}
in particular, each $|\alpha(g;q_1,\ldots,q_{r+s})|\leq 1$. Since the sum has $\ll_F \prod_{j=r+1}^{r+s}(l_j+1)$ terms, it suffices to prove
\begin{equation*}
\frac{|\Phi_{p,q}^{l}(k)|}{ ||\Phi_{p,q}^{l}||_{\mathrm{SU}_2(\CC)}}\ll (l+1)^6.
\end{equation*}
This follows from \cite[Lemma on p.348 and Corollary on p.349]{BernsteinReznikov} with $n=4$ by the standard quaternion representation of $\mathrm{SU}_2(\CC)$. Each derivation gives a factor $\ll (l+1)^{3/2}$, see \cite[(2.19), (2.31)]{Lokvenec}.
\end{proof}

\begin{lemm}\label{freitag} Let $N=2^rh$, where $h$ is the class number of $F$. There are finitely many elements $a_1,\ldots,a_N\in\GLtwo{2}{F}$ regarded as elements of $\GLtwo{2}{F_{\infty}}$ and some $\delta>0$ such that for any $g\in\GLtwo{2}{F_{\infty}}$, there exist elements $z\in Z(F_{\infty})$, $\gamma\in\SLtwo{2}{\frak{o}}$ (regarded as an element of $\GLtwo{2}{F_{\infty}}$) and $k\in K_{\infty}$ such that
\begin{equation*}
g=z\gamma a_j \begin{pmatrix}y & x \cr 0 & 1\end{pmatrix}k,
\end{equation*}
for some $1\leq j\leq N$, where $\bigl(\begin{smallmatrix}y & x \cr 0 & 1\end{smallmatrix}\bigr)\in B(F_{\infty})$ satisfies $y_1,\ldots,y_{r+s}>\delta$ (in particular, all of them are real).
\end{lemm}
\begin{proof} The statement is proved in \cite[Theorem 3.6]{Freitag} for $\SLone{2}$ over totally real fields (in that case, the implied $N$ is the class number of $F$). The case of a general number field $F$ can be handled using the same technique. The transition from $\SLone{2}$ to $\GLone{2}$ is straight-forward, see \cite[Lemma 4.9]{MPphd}.
\end{proof}

\begin{prop}\label{boundonsupremumnorm} Let $(\pi,V_{\pi})$ be an irreducible cuspidal representation. Assume that $\phi\in V_{\pi}(\frak{c}_{\pi})$ such that $||\phi||_{S_{2(7r+18s)}}$ exists. Then
\begin{equation*}
||\phi||_{\infty}=\sup_{g\in\GLtwo{2}{\AAA}}|\phi(g)|\ll_{F,\pi} ||\phi||_{S_{2(7r+18s)}}.
\end{equation*}
\end{prop}
\begin{proof} We follow the proof of \cite[Lemma 5]{BlomerHarcos}. Note that there is a correction made later in its erratum, which we also build in.
First assume $\phi\in V_{\pi}(\frak{c}_{\pi})$ is of pure weight $\ww$. Let $\eta_1,\ldots,\eta_h\in\AAA_{\fin}^{\times}$ be finite ideles representing the ideal classes. By strong approximation \cite[Theorem 3.3.1]{Bump}, there exist $\gamma\in\GLtwo{2}{F}$, $g'\in\GLtwo{2}{F_{\infty}}$, $k\in K(\frak{o})$ such that for some $1\leq j\leq h$,
\begin{equation*}
g=\gamma\left(\overbrace{g'}^{\in\GLtwo{2}{F_{\infty}}} \times \overbrace{\begin{pmatrix} \eta_j^{-1} & 0 \cr 0 & 1 \end{pmatrix}k}^{\in\GLtwo{2}{\AAA_{\fin}}}\right).
\end{equation*}
Now decompose $g'$ in the sense of Lemma \ref{freitag} as
\begin{equation*}
g'=z\gamma' a_{j'} \begin{pmatrix}y' & x' \cr 0 & 1\end{pmatrix}k',
\end{equation*}
where $a_{j'}\in\GLtwo{2}{F}$ (regarded as an element of $\GLtwo{2}{F_{\infty}}$) is from the fixed set $\{a_1,\ldots,a_{2^rh}\}$, $y'>\delta$ at all archimedean places, where $\delta>0$ is fixed (depending only on $F$), $z\in Z(F_{\infty})$, $\gamma'\in\SLtwo{2}{\frak{o}}$, $k'\in K_{\infty}$.
From now on, we regard $z$ as an element in $Z(\AAA)$, therefore we have
\begin{equation*}
g=z\gamma\gamma'a_{j'}\left( \overbrace{\begin{pmatrix}y' & x' \cr 0 & 1 \end{pmatrix}k'}^{\in\GLtwo{2}{F_{\infty}}} \times \overbrace{a_{j'}^{-1}\gamma'^{-1}\begin{pmatrix}\eta_j^{-1} & 0 \cr 0 & 1 \end{pmatrix}k}^{\in\GLtwo{2}{\AAA_{\fin}}}\right).
\end{equation*}
Here, $a_{j'}^{-1}\gamma'^{-1}\bigl(\begin{smallmatrix}\eta_j^{-1} & 0 \cr 0 & 1 \end{smallmatrix}\bigr)k$ lies in a fixed compact subset of $\GLtwo{2}{\AAA_{\fin}}$, which can be covered with finitely many left cosets of the open subgroup $K(\frak{c}_{\pi})$. Therefore
\begin{equation*}
g=z\gamma^*\left(\overbrace{\begin{pmatrix}y' & x' \cr 0 & 1\end{pmatrix}}^{\in\GLtwo{2}{F_{\infty}}}\times \overbrace{m}^{\in\GLtwo{2}{\AAA_{\fin}}}\right) \left(\overbrace{k^*_{\infty}}^{\in\GLtwo{2}{F_{\infty}}}\times \overbrace{k^*_{\fin}}^{\GLtwo{2}{\AAA_{\fin}}}\right),
\end{equation*}
where $\gamma^*\in\GLtwo{2}{F}$, $k^*=k^*_{\infty}\times k^*_{\fin}\in K_{\infty}\times K(\frak{c}_{\pi})$, and $m\in\GLtwo{2}{\AAA_{\fin}}$ runs through a finite set depending only on $F$ and $\frak{c}_{\pi}$, $y'>\delta$ at all archimedean places.

Now let $\phi_m(g)=\phi(gm)$. Obviously, $\phi$ and $\phi_m$ have the same supremum and Sobolev norms, and when $g$ decomposes as above,
\begin{equation}\label{insupremumestimatecomplexrotationskilled}
\begin{split}
|\phi(g)|&=\left|\phi_m \left( \overbrace{\begin{pmatrix} y' & x' \cr 0 & 1\end{pmatrix}k^*_{\infty}}^{\in\GLtwo{2}{F_{\infty}}} \times \overbrace{\begin{pmatrix} 1 & 0 \cr 0 & 1\end{pmatrix}}^{\in\GLtwo{2}{\AAA_{\fin}}} \right) \right|\\ & \ll_F \left|\phi_m \left( \overbrace{\begin{pmatrix} y' & x' \cr 0 & 1\end{pmatrix}}^{\in\GLtwo{2}{F_{\infty}}} \times \overbrace{\begin{pmatrix} 1 & 0 \cr 0 & 1\end{pmatrix}}^{\in\GLtwo{2}{\AAA_{\fin}}} \right) \right| \prod_{j=r+1}^{r+s} (l_j+1)^7,
\end{split}
\end{equation}
where we applied Lemma \ref{killingcomplexrotations} in the last estimate.

The function $\phi_m$ can be regarded as a classical automorphic function on $\GLtwo{2}{F_{\infty}}$ (see \cite[Section 4.3]{MPphd}). Therefore, analogously to (\ref{fourierwhittakerexpansion}), we see that $\phi_m$ (as a function on $\GLtwo{2}{F_{\infty}}$) can be expanded into Fourier series
\begin{equation}\label{insupremumestimatefourierwhittakerexpansion}
\phi_m\left(\begin{pmatrix} y' & x' \cr 0 & 1\end{pmatrix}\right)= \sum_{0\neq t\in \frak{f}}a(t)\whit_{\ww,\rr}(ty')\psi_{\infty}(tx'),
\end{equation}
where $\frak{f}$ is a fractional ideal (regarded as a lattice in $F_{\infty}$) depending only on $F$ and $\pi$.

We need some bound on the Fourier-Whittaker coefficients, which we work out in the following lemma. 
\begin{lemm}\label{estimateoffouriercoeffecients}
\begin{equation*}
a(t)\ll_{F,\pi} ||\phi||\prod_{j\leq r}(1+|q_j|^{5}) \prod_{j> r}(1+l_j^{10}).
\end{equation*}
\end{lemm}
\begin{proof} By Plancherel's formula,
\begin{equation*}
\sum_{0\neq t\in \frak{f}} |a(t)\whit_{\ww,\rr}(ty')|^2= \const(F,\pi) \int_{\rfact{F_{\infty}}{\frak{f}'}} \left|\phi_m\left(\begin{pmatrix} y' & x' \cr 0 & 1\end{pmatrix}\right)\right|^2 dx',
\end{equation*}
where $\frak{f}'$ is the dual of $\frak{f}$. Take only a single term on the left-hand side. Choose $C_0$ as in Lemma \ref{whittakerfunctionsdonotconcentrate} and then take $\varepsilon$ to be the largest which is allowed there. Integrate both sides on the domain $S(\varepsilon/t)$ with respect to the measure $dy'/(\prod_{j\leq r}|y'_j|^2 \prod_{j>r}|y'_j|^3)$ (the invariant measure on the symmetric space $\rfact{\GLtwo{2}{F_{\infty}}}{K_{\infty}}$ is $dx'dy'/(\prod_{j\leq r}|y'_j|^2 \prod_{j>r}|y'_j|^3)$). By Lemma \ref{whittakerfunctionsdonotconcentrate}, we obtain
\begin{equation*}
\begin{split}
&|a(t)|^2|t|_{\infty} \left(\prod_{j\leq r}(1+q_j^2)\prod_{j> r}(1+l_j^4)\right)^{-1}\\ & \ll_{F,\pi} \int_{\rfact{F_{\infty}}{\frak{f}'} \times S(\varepsilon/t)} \left|\phi_m\left(\begin{pmatrix} y' & x' \cr 0 & 1\end{pmatrix}\right)\right|^2 \frac{dx'dy'}{\prod_{j\leq r}|y'_j|^2 \prod_{j>r}|y'_j|^3}.
\end{split}
\end{equation*}
The domain $\rfact{F_{\infty}}{\frak{f}'} \times S(\varepsilon/t)$ covers each point of $\lrfact{Z(\AAA)\GLtwo{2}{F}}{\GLtwo{2}{\AAA}}{(mK(\frak{c}_{\pi})m^{-1})}$ at most $O_{F,\pi}(|t/\varepsilon|_{\infty})$ times (see \cite[Lemma 2.10]{IwaniecSpectral}), which, together with the choice of $\varepsilon$, gives
\begin{equation*}
|a(t)|^2\ll_{F,\pi}||\phi||^2\prod_{j\leq r}(1+q_j^{10})\prod_{j>r}(1+l_j^{20}),
\end{equation*}
and the claim follows.
\end{proof}
Now (\ref{insupremumestimatecomplexrotationskilled}) and (\ref{insupremumestimatefourierwhittakerexpansion}) give
\begin{equation}\label{insupremumestimatepreliminaryestimate0}
|\phi(g)|\ll_{F,\pi} ||\phi|| \prod_{j\leq r}(1+|q_j|^5) \prod_{j> r}(1+l_j^{17}) \sum_{0\neq t\in \frak{f}} |\whit_{\ww,\rr}(ty')|.
\end{equation}
We turn our attention to $\sum_{0\neq t\in \frak{f}} |\whit_{\ww,\rr}(ty')|$.

From (\ref{realwhittakerestimateformula1}), (\ref{realwhittakerestimateformula2}), (\ref{realwhittakerestimateformula3}), (\ref{complexwhittakerpreliminaryestimateformula2}) and (\ref{complexwhittakerestimateformula}), we see that
\begin{equation}\label{insupremumestimateuniformboundsonwhittakerfunctions}
\begin{split}
&\whit_{q,\nu}(y)\ll_{F,\pi} (|q|^3+1)\exp\left(-\frac{|y|}{2(q^2+1)(|\nu|^2+1)}\right),\\ &\whit_{(l,q),(\nu,p)}(y)\ll_{F,\pi} (l^8+1)\exp\left(-\frac{|y|}{2(l^4+1)(|\nu|^2+1)}\right)
\end{split}
\end{equation}
holds for all $y\neq 0$, at real and complex places, respectively.

Setting $A_j=|q_j|^3+1$, $B_j=2(q_j^2+1)(|\nu_j|^2+1)$ at real places, and $A_j=l_j^8+1$, $B_j=2(l_j^4+1)(|\nu_j|^2+1)$ at complex places, (\ref{insupremumestimateuniformboundsonwhittakerfunctions}) and a simple calculation yields
\begin{equation*}
\sum_{0\neq t\in\frak{f}}|\whit_{\ww,\rr}(ty')| \ll_{F,\frak{f}} \prod_{j=1}^{r+s}A_jB_j^{\deg[F_j:\RR]},
\end{equation*}
where we used that $|y'_j|>\delta$ at all places, and also the fact that a lattice $L$ in $F_{\infty}$ contains $O_L(N^{r+2s})$ points of supremum norm $\leq N$.

Therefore,
\begin{equation*}
\sum_{0\neq t\in\frak{f}}|\whit_{\ww,\rr}(ty')|\ll_{F,\pi} \prod_{j=1}^r (|q_j|^5+1) \prod_{j=r+1}^{r+s} (l_j^{16}+1),
\end{equation*}
which, together with (\ref{insupremumestimatepreliminaryestimate0}), give rise to
\begin{equation}\label{insupremumestimatepreliminaryestimate}
|\phi(g)|\ll_{F,\pi} ||\phi|| \prod_{j\leq r} (1+q_j^{10}) \prod_{j>r} (1+l_j^{33}).
\end{equation}

Assume now a sufficiently smooth $\phi\in V_{\pi}$ is not necessarily of pure weight. We may decompose it as
\begin{equation}\label{insupremumestimateweightdecomposition}
\phi=\sum_{\ww\in W(\pi)}b_{\ww}\phi_{\ww},
\end{equation}
where $\phi_{\ww}$ is a weight $\ww$ function of norm $1$ in $V_{\pi}$.
Let us follow the common practice and using the smoothness of $\phi$, estimate $b_{\ww}$ in terms of $\sup\ww=\max(|q_1|,\ldots,|q_r|,l_{r+1},\ldots,l_{r+s})$. Using Parseval, then (\ref{definitionofcompactcasimirelementactionsreal}) and (\ref{definitionofcompactcasimirelementactionscomplex}), we find, for any nonnegative integer $k$,
\begin{equation}\label{insupremumestimateweightcoefficientsestimate}
b_{\ww}=\langle \phi, \phi_{\ww} \rangle\ll_k \frac{1}{(1+(\sup \ww))^{2k}}\langle \Omega_{\frak{k},j}^k\phi, \phi_{\ww} \rangle\ll_k \frac{1}{(1+(\sup \ww))^{2k}} ||\phi||_{S_{2k}},
\end{equation}
where $j$ is the index of an archimedean place, where the maximum (in the definition of $\sup\ww$) is attained. Together with (\ref{insupremumestimatepreliminaryestimate}) and (\ref{insupremumestimateweightdecomposition}), this implies
\begin{equation*}
|\phi(g)|\ll_{F,\pi,k} \sum_{\ww\in W(\pi)} (1+\sup \ww)^{10r+33s-2k} ||\phi||_{S_{2k}}.
\end{equation*}
Here, choosing $k=7r+18s$, we obtain the statement by noting that $\sup\ww$ attains the positive integer $N$ on a set of cardinality $O_F(N^{r+2s-1})$.
\end{proof}

\subsection{A bound on Kirillov vectors}

\begin{prop}\label{boundonkirillovvectors} Let $(\pi,V_{\pi})$ be an irreducible automorphic representation occuring in $L^2(\lfact{\GLtwo{2}{F}}{\GLtwo{2}{\AAA}},\omega)$. Let $\frak{t}\subseteq\frak{o}$ be an ideal, $a,b,c$ be nonnegative integers, $0<\varepsilon<1/4$. Let $P\in\CC[x_1,\ldots,x_{r+2s}]$ be a polynomial of degree at most $a$ in each variable. Set then
\begin{equation*}
\mathcal{D}=P\left(\left(y_j\frac{\partial}{\partial y_j}\right)_{j\leq r}, \left(y_j\frac{\partial}{\partial y_j}\right)_{j>r}, \left(\overline{y_j}\frac{\partial}{\partial \overline{y_j}}\right)_{j>r}\right).
\end{equation*}
Assume $\phi\in R^{\frak{t}}V_{\pi}(\frak{c}_{\pi})$ such that $||\phi||_{S_{2(3r+4s+2)+(r+s)(a+b+2c)}}$ exists. Then $\mathcal{D}W_{\phi}$ exists and
\begin{equation*}
\begin{split}
&\mathcal{D}W_{\phi}(y)\ll_{a,b,c,P,F,\varepsilon} ||\phi||_{S_{2(3r+4s+2)+(r+s)(a+b+2c)}}\norm{\frak{t}}^{\varepsilon} \norm{\frak{c}_{\pi}}^{\varepsilon} \norm{\rr}^{-c} \\ & \cdot\prod_{j=1}^r (|y_j|^{1/2-\varepsilon}+|y_j|^{1/2-\theta-\varepsilon})(\min(1,|y_j|^{-b})) \prod_{j=r+1}^{r+s}(|y_j|^{3/4}+|y_j|)(\min(1,|y_j|^{-b})).
\end{split}
\end{equation*}
\end{prop}
\begin{proof} We follow the proof of \cite[Lemma 4]{BlomerHarcos}. First assume $\phi\in R^{\frak{t}}V_{\pi}(\frak{c}_{\pi})$ is of pure weight $\ww$. Then we may write
\begin{equation*}
|W_{\phi}(y)|=||W_{\phi}||\cdot|\whit_{\ww,\rr}(y)|.
\end{equation*}
Using Proposition \ref{archimedeankirillovmodelproposition}, (\ref{archimedeankirillovmodelformula}), (\ref{estimateofkirillovvectorofeisenstein}), (\ref{archimedeankirillovmodelformula2}), the remark after that, (\ref{proportionalityconstantisessentiallylfunctionresidue}), Proposition \ref{residueofrankinselbergsquarelfuncionisaboutone}, and the estimates (\ref{realwhittakerestimateformula2}), (\ref{realwhittakerestimateformula3}), (\ref{complexwhittakerestimateformula}), we have, for $0<\varepsilon<1/4$,
\begin{equation*}
\begin{split}
W_{\phi}(y)\ll_{F,\varepsilon} ||\phi|| \norm{\frak{t}}^{\varepsilon} \norm{\frak{c}_{\pi}}^{\varepsilon} \norm{\rr}^{\varepsilon}& \prod_{j=1}^r (1+|\nu_j|+|q_j|)^{1+\theta}(|y_j|^{1/2-\varepsilon} +|y_j|^{1/2-\theta-\varepsilon})\\ &\cdot\prod_{j=r+1}^{r+s} (1+|p_j|)(1+|\nu_j|^2)(1+l_j^4) (|y_j|^{3/4}+|y_j|).
\end{split}
\end{equation*}
This gives
\begin{equation*}
\begin{split}
W_{\phi}(y)\ll_{F,\varepsilon} ||\phi|| \norm{\frak{t}}^{\varepsilon} \norm{\frak{c}_{\pi}}^{\varepsilon} \norm{\rr}^2& \prod_{j=1}^r (1+|q_j|)^{1+\theta}(|y_j|^{1/2-\varepsilon} +|y_j|^{1/2-\theta-\varepsilon})\\ &\cdot\prod_{j=r+1}^{r+s} (1+l_j^4) (|y_j|^{3/4}+|y_j|).
\end{split}
\end{equation*}

Now take an arbitrary $\phi\in R^{\frak{t}}V_{\pi}(\frak{c}_{\pi})$, which is sufficiently smooth. Then recalling (\ref{insupremumestimateweightdecomposition}) and (\ref{insupremumestimateweightcoefficientsestimate}), in
\begin{equation*}
\phi=\sum_{\ww\in W(\pi)}b_{\ww}\phi_{\ww},\qquad ||\phi_{\ww}||=1,
\end{equation*}
we have
\begin{equation*}
b_{\ww}\ll_k\frac{1}{(1+(\sup \ww))^{2k}} ||\phi||_{S_{2k}}.
\end{equation*}

Now choosing $k=3r+4s$, we obtain
\begin{equation}\label{inkirillovvectorestimatepreliminary}
\begin{split}
W_{\phi}(y)\ll_{F,\varepsilon} ||\phi||_{S_{2(3r+4s)}} \norm{\frak{t}}^{\varepsilon} \norm{\frak{c}_{\pi}}^{\varepsilon} \norm{\rr}^2&\prod_{j=1}^r (|y_j|^{1/2-\varepsilon} +|y_j|^{1/2-\theta-\varepsilon})\\ & \cdot \prod_{j=r+1}^{r+s}(|y_j|^{3/4}+|y_j|).
\end{split}
\end{equation}
The differential operators given in (\ref{differentialoperatorbasisreal}) and (\ref{differentialoperatorbasiscomplex}) act on the sufficiently smooth Kirillov vectors. We record the action of some of them (neglecting some absolute scalars for simplicity). Of course, $\Omega_{j(,\pm)}$ act by $\lambda_{(\pm)}$. From (\ref{alternativedefinitionforkirillovvector1}) and (\ref{alternativedefinitionforkirillovvector2}), it is easy to derive that $\mathbf{R}_j$, $\mathbf{V}_{1,j}+\mathbf{W}_{1,j}$, $\mathbf{V}_{2,j}+\mathbf{W}_{2,j}$ act via a multiplication by $y_j$, $\Re y_j$, $\Im y_j$, respectively; finally $\mathbf{H}_j$ by $y_j(\partial/\partial y_j)$, and $\mathbf{H}_{1,j}$, $\mathbf{H}_{2,j}$ by $y_j(\partial/\partial y_j)+\overline{y_j}(\partial/\partial \overline{y_j})$, $iy_j(\partial/\partial y_j)-i\overline{y_j}(\partial/\partial \overline{y_j})$, respectively.

Now assume given $a,b,c$ and the polynomial $P$ as in the statement. Then 
\begin{equation*}
\mathcal{D}= \const_{F,\mathcal{D}} P\Bigl((\mathbf{H}_j)_{j\leq r}, ((\mathbf{H}_{1,j}-i\mathbf{H}_{2,j})/2)_{j>r}, ((\mathbf{H}_{1,j}+i\mathbf{H}_{2,j})/2)_{j>r}\Bigr),
\end{equation*}
and define the differential operator
\begin{equation*}
\begin{split}
\mathcal{D}'=&\left(\prod_{j\leq r} \Omega_j^{c+2} \prod_{j>r} \Omega_{j,+}^{c+2}\right) \left(\prod_{\substack{1\leq j\leq r\\ |y_j|\geq 1}} \mathbf{R}_j^b\right) \\ & \left(\prod_{\substack{r+1\leq j\leq r+s\\ |y_j|\geq 1\\ |\Re y_j|\geq |\Im y_j|}} (\mathbf{V}_{1,j}+\mathbf{W}_{1,j})^b\right) \left(\prod_{\substack{r+1\leq j\leq r+s\\ |y_j|\geq 1\\ |\Re y_j|< |\Im y_j|}} (\mathbf{V}_{2,j}+\mathbf{W}_{2,j})^b\right).
\end{split}
\end{equation*}
Applying (\ref{inkirillovvectorestimatepreliminary}) to $\mathcal{D}'\mathcal{D}\phi$, we obtain the statement.
\end{proof}

\section{The spectral decomposition of shifted convolution sums}\label{chap:shiftedconvolutionsums}

The aim of this section is to prove a variant of \cite[Theorem 2]{BlomerHarcos} for arbitrary number fields.

We focus on the subspace $L^2(\lrfact{\GLtwo{2}{F}}{\GLtwo{2}{\AAA}}{K(\frak{c})},\omega)$, the subspace consisting of functions that are right $K(\frak{c})$-invariant. Its spectral decomposition is similar to (\ref{spectraldecomposition}), the only modification is the restriction of $\mathcal{C}_{\omega},\mathcal{E}_{\omega}$ to $\mathcal{C}_{\omega}(\frak{c}),\mathcal{E}_{\omega}(\frak{c})$, respectively (recall (\ref{C(c)E(c)})). We write
\begin{equation*}
\int_{(\frak{c})} f_{\varpi}d\varpi= \sum_{\pi\in\mathcal{C}_{\omega}(\frak{c})}f_{\pi} + \int_{\mathcal{E}_{\omega}(\frak{c})}f_{\varpi}d\varpi,
\end{equation*}
if $f$ is a function of those infinite-dimensional representations, which are not orthogonal to $L^2(\lrfact{\GLtwo{2}{F}}{\GLtwo{2}{\AAA}}{K(\frak{c})},\omega)$, 

\begin{theo}\label{spectraldecompositionofshiftedconvolutionsum} We have a spectral decomposition of shifted convolution sums in the sense of Part \ref{spectraldecompositionofshiftedconvolutionsumpart1} with functions satisfying the bound in Part \ref{spectraldecompositionofshiftedconvolutionsumpart2}.
\end{theo}

\begin{partoftheo}\label{spectraldecompositionofshiftedconvolutionsumpart1}
Assume $\pi_1,\pi_2$ are irreducible cuspidal representations of the same central character. Let $l_1,l_2\in\frak{o}\setminus\{0\}$, and set $\frak{c}=\mathrm{lcm}(l_1\frak{c}_{\pi_1},l_2\frak{c}_{\pi_2})$. Let moreover $W_1,W_2:F_{\infty}^{\times}\rightarrow\CC$ be arbitrary Schwarz functions, that is, they are smooth and tend to $0$ faster then any power of $y^{-1}$ or $y$, as $y$ tends to $\infty$ or $0$, respectively. Then for any $\varpi\in\mathcal{C}_{1}(\frak{c})\cup\mathcal{E}_{1}(\frak{c})$ and $\frak{t}|\frak{cc}_{\pi}^{-1}$, there exists a function $W_{\varpi,\frak{t}}:F_{\infty}^{\times}\rightarrow\CC^{\times}$ depending only on $F,\pi_1,\pi_2,W_1,W_2,\varpi,\frak{t}$ such that the following holds. For any $Y\in (0,\infty)^{r+s}$, any ideal $\frak{n}\subseteq\frak{o}$ and any $0\neq q\in\frak{n}$, there is a spectral decomposition of the shifted convolution sum
\begin{equation*}
\begin{split}
\sum_{l_1t_1-l_2t_2=q,0\neq t_1,t_2\in\frak{n}} \frac{\lambda_{\pi_1}(t_1\frak{n}^{-1})\overline{ \lambda_{\pi_2}(t_2\frak{n}^{-1})}}{ \sqrt{\norm{t_1t_2\frak{n}^{-2}}}} &W_1\left(\left(\frac{(l_1t_1)_j}{Y_j}\right)_j\right) \overline{W_2\left(\left(\frac{(l_2t_2)_j}{Y_j}\right)_j\right)}\\ & =\int_{(\frak{c})}\sum_{\frak{t}|\frak{cc}_{\varpi}^{-1}} \frac{\lambda_{\varpi}^{\frak{t}}(q\frak{n}^{-1})}{\sqrt{\norm{q\frak{n}^{-1}}}} W_{\varpi,\frak{t}}\left(\left(\frac{q_j}{Y_j}\right)_j\right)d\varpi,
\end{split}
\end{equation*}
where $\lambda_{\varpi}^{\frak{t}}(\frak{m})$ is given in (\ref{orthogonalizedfouriercofficients}).
\end{partoftheo}
\begin{proof} First apply Proposition \ref{archimedeankirillovmodelproposition} to see that there exist functions $\phi_1\in V_{\pi_1}(\frak{c}_{\pi_1})$, $\phi_2\in V_{\pi_2}(\frak{c}_{\pi_2})$ such that $W_{\phi_1}=W_1$, $W_{\phi_2}=W_2$. Set then
\begin{equation*}
\Phi=R_{(l_1)}\phi_1R_{(l_2)}\overline{\phi_2}.
\end{equation*}
Then since $\frak{c}=\mathrm{lcm}(l_1\frak{c}_{\pi_1},l_2\frak{c}_{\pi_2})$, we see that $\Phi$ is right $K(\frak{c})$-invariant. Also, since $W_1,W_2$ are from the Schwarz space, $\phi_1$, $\phi_2$ are smooth and have finite Sobolev norms of arbitrarily large order, so does $\Phi\in L^2(\lrfact{\GLtwo{2}{F}}{\GLtwo{2}{\AAA}}{K(\frak{c})},1)$ (use Proposition \ref{archimedeankirillovmodelproposition} and Proposition \ref{boundonsupremumnorm} together with \cite[Lemma 8.4]{Venkateshsparse}). Then by (\ref{spectraldecomposition}), (\ref{newformoldformdecomposition1}), (\ref{newformoldformdecomposition2}) and the remark made in the beginning of this section, we can decompose $\Phi$ as
\begin{equation}\label{inspectraldecompositionofshiftedconvolutionsumslemmaspectraldecomposition}
\Phi=\Phi_{\sp}+\int_{(\frak{c})} \sum_{\frak{t}|\frak{cc}_{\pi}^{-1}}\Phi_{\varpi,\frak{t}} d\varpi,
\end{equation}
where $\Phi_{\varpi,\frak{t}}\in R^{\frak{t}}(V_{\varpi}(\frak{c}_{\varpi}))$ and $\Phi_{\sp}$ is the orthogonal projection of $\Phi$ to $L_{\sp}$. Now set $W_{\varpi,\frak{t}}=W_{\Phi_{\varpi,\frak{t}}}$. We claim this fulfills the property stated in Part \ref{spectraldecompositionofshiftedconvolutionsumpart1}. Given $Y\in(0,\infty)^{r+s}$, $\frak{n}\subseteq\frak{o}$, $0\neq q\in\frak{n}$, let $(y_{\fin})=\frak{n}$, and $y_{\infty}=Y$. We compute
\begin{equation}\label{inspectraldecompositionofshiftedconvolutionsumsharmonic}
\int_{\lfact{F}{\AAA}} \Phi\left(\begin{pmatrix}y^{-1} & x \cr 0 & 1 \end{pmatrix}\right)\psi(-qx)dx
\end{equation}
in two ways. On the one hand, we use (\ref{inspectraldecompositionofshiftedconvolutionsumslemmaspectraldecomposition}). Here, $q\neq 0$ implies that $\Phi_{\sp}$ has zero contribution to (\ref{inspectraldecompositionofshiftedconvolutionsumsharmonic}), and we obtain
\begin{equation*}
\int_{\lfact{F}{\AAA}} \Phi\left(\begin{pmatrix}y^{-1} & x \cr 0 & 1 \end{pmatrix}\right)\psi(-qx)dx=\int_{(\frak{c})}\sum_{\frak{t}|\frak{cc}_{\varpi}^{-1}} \frac{\lambda_{\varpi}^{\frak{t}}(q\frak{n}^{-1})}{\sqrt{\norm{q\frak{n}^{-1}}}} W_{\varpi,\frak{t}}\left(\left(\frac{q_j}{Y_j}\right)_j\right)d\varpi
\end{equation*}
from (\ref{fourierwhittakerexpansion2oldforms}) and (\ref{fourierwhittakerexpansionofeisensteinseries}). On the other hand, using (\ref{raisinglevel}) and (\ref{fourierwhittakerexpansion2newforms}) together with the choice of $\phi_1$, $\phi_2$, we obtain
\begin{equation*}
\begin{split}
\int_{\lfact{F}{\AAA}} \Phi\left(\begin{pmatrix}y^{-1} & x \cr 0 & 1 \end{pmatrix}\right)\psi(-qx)dx=&\sum_{l_1t_1-l_2t_2=q,0\neq t_1,t_2\in\frak{n}} \frac{\lambda_{\pi_1}(t_1\frak{n}^{-1})\overline{ \lambda_{\pi_2}(t_2\frak{n}^{-1})}}{ \sqrt{\norm{t_1t_2\frak{n}^{-2}}}} \\ &\cdot W_1\left(\left(\frac{(l_1t_1)_j}{Y_j}\right)_j\right) \overline{W_2\left(\left(\frac{(l_2t_2)_j}{Y_j}\right)_j\right)}.
\end{split}
\end{equation*}
The equality of the last two displays is exactly the statement.
\end{proof}

\begin{partoftheo}\label{spectraldecompositionofshiftedconvolutionsumpart2} Conditions as in Part \ref{spectraldecompositionofshiftedconvolutionsumpart1}. Assume $\mathcal{D}$ is a differential operator as in Proposition \ref{boundonkirillovvectors}. Then for any $0<\varepsilon<1/4$ and nonnegative integers $b,c$, we have, for all $y\in F_{\infty}^{\times}$,
\begin{equation*}
\begin{split}
&\int_{(\frak{c})}\sum_{\frak{t}|\frak{cc}_{\varpi}^{-1}} (\norm{\rr_{\varpi}})^{2c} |\mathcal{D}W_{\varpi,\frak{t}}(y)|^2 d\varpi \ll_{F,\varepsilon,\pi_1,\pi_2,a,b,c,P} \norm{(l_1l_2)}^{\varepsilon} ||W_1||^2_{S_{\alpha}}||W_2||^2_{S_{\alpha}} \\ & \cdot \prod_{j=1}^r (|y_j|^{1-\varepsilon}+|y_j|^{1-2\theta-\varepsilon})(\min(1,|y_j|^{-2b})) \prod_{j=r+1}^{r+s}(|y_j|^{3/2}+|y_j|^2)(\min(1,|y_j|^{-2b}))
\end{split}
\end{equation*}
with $\alpha=2(3r+4s+2)+(r+s)(a+b+2c)+2(7r+18s)$.
\end{partoftheo}
\begin{proof} Let $\Phi$ be the function appearing in the proof of Part \ref{spectraldecompositionofshiftedconvolutionsumpart1}. Then by Proposition \ref{boundonkirillovvectors} and a consequence of (\ref{spectraldecompositionofderivatives}) (see \cite[(85)]{BlomerHarcos}), we have
\begin{equation*}
\begin{split}
&\int_{(\frak{c})}\sum_{\frak{t}|\frak{cc}_{\varpi}^{-1}} (\norm{\rr_{\varpi}})^{2c} |\mathcal{D}W_{\varpi,\frak{t}}(y)|^2 d\varpi \ll_{F,\varepsilon,\pi_1,\pi_2,a,b,c,P} \norm{l_1l_2}^{\varepsilon} ||\Phi||^2_{S_{\beta}} \\ & \cdot \prod_{j=1}^r (|y_j|^{1-\varepsilon}+|y_j|^{1-2\theta-\varepsilon})(\min(1,|y_j|^{-2b})) \prod_{j=r+1}^{r+s}(|y_j|^{3/2}+|y_j|^2)(\min(1,|y_j|^{-2b}))
\end{split}
\end{equation*}
with $\beta=2(3r+4s+2)+(r+s)(a+b+2c)$. For any differential operator $\mathcal{D}'\in U(\frak{g})$ of order $k$, we have
\begin{equation*}
||\mathcal{D}'\phi_1||_{\infty}\ll_{F,\pi_1} ||\phi_1||_{S_{k+2(7r+18s)}},\qquad ||\mathcal{D}'\phi_2||_{\infty}\ll_{F,\pi_2} ||\phi_2||_{S_{k+2(7r+18s)}}
\end{equation*}
by Proposition \ref{boundonsupremumnorm}. Since $\lfact{Z(\AAA)\GLtwo{2}{F}}{\GLtwo{2}{\AAA}}$ has finite volume, and the operators $R_{(l_{1,2})}$ do not affect Sobolev norms, $||\Phi||_{S_{\beta}}\ll_{F,\pi_1,\pi_2} ||\phi_1||_{S_{\beta+2(7r+18s)}}||\phi_2||_{S_{\beta+2(7r+18s)}}$. Now Proposition \ref{residueofrankinselbergsquarelfuncionisaboutone} and (\ref{proportionalityconstantisessentiallylfunctionresidue}) completes the proof.
\end{proof}

From the $L^2$-bound presented in Theorem \ref{spectraldecompositionofshiftedconvolutionsum}, Part \ref{spectraldecompositionofshiftedconvolutionsumpart2}, we can easily deduce $L^1$-bounds (using Cauchy-Schwarz, Corollary \ref{densityofeisensteinspectrum2} and Corollary \ref{densityofcuspidalspectrum2}), these are essentially generalizations of \cite[Remark 12]{BlomerHarcos}.

\begin{coro}\label{spectraldecompositionofshiftedconvolutionsumcorollaryeisenstein}
Conditions as in Theorem \ref{spectraldecompositionofshiftedconvolutionsum}. Assume $\mathcal{D}$ is a differential operator as in Proposition \ref{boundonkirillovvectors}. Then for any $0<\varepsilon<1/4$ and nonnegative integers $b,c'$, we have, for all $y\in F_{\infty}^{\times}$,
\begin{equation*}
\begin{split}
&\int_{\mathcal{E}_{1}(\frak{c})}\sum_{\frak{t}|\frak{cc}_{\varpi}^{-1}} (\norm{\rr_{\varpi}})^{c'} |\mathcal{D}W_{\varpi,\frak{t}}(y)| d\varpi \ll_{F,\varepsilon,\pi_1,\pi_2,a,b,c',P} \norm{\frak{l}}^{1/4}\norm{(l_1l_2)}^{\varepsilon} ||W_1||_{S_{\alpha'}}||W_2||_{S_{\alpha'}} \\ & \cdot \prod_{j=1}^r (|y_j|^{1/2-\varepsilon}+|y_j|^{1/2-\theta-\varepsilon})(\min(1,|y_j|^{-b})) \prod_{j=r+1}^{r+s}(|y_j|^{3/4}+|y_j|)(\min(1,|y_j|^{-b}))
\end{split}
\end{equation*}
with $\alpha'=2(3r+4s+2)+(r+s)(a+b+2c'+4(r+2s))+2(7r+18s)$, where $\frak{l}$ stands for the largest square divisor of $\mathrm{lcm}((l_1),(l_2))$.
\end{coro}

\begin{coro}\label{spectraldecompositionofshiftedconvolutionsumcorollarycuspidal}
Conditions as in Theorem \ref{spectraldecompositionofshiftedconvolutionsum}. Assume $\mathcal{D}$ is a differential operator as in Proposition \ref{boundonkirillovvectors}. Then for any $0<\varepsilon<1/4$ and nonnegative integers $b,c'$, we have, for all $y\in F_{\infty}^{\times}$,
\begin{equation*}
\begin{split}
&\sum_{\varpi\in\mathcal{C}_{1}(\frak{c})}\sum_{\frak{t}|\frak{cc}_{\varpi}^{-1}} (\norm{\rr_{\varpi}})^{c'} |\mathcal{D}W_{\varpi,\frak{t}}(y)| \ll_{F,\varepsilon,\pi_1,\pi_2,a,b,c',P} \norm{(l_1l_2)}^{1/2+\varepsilon} ||W_1||_{S_{\alpha'}}||W_2||_{S_{\alpha'}} \\ & \cdot \prod_{j=1}^r (|y_j|^{1/2-\varepsilon}+|y_j|^{1/2-\theta-\varepsilon})(\min(1,|y_j|^{-b})) \prod_{j=r+1}^{r+s}(|y_j|^{3/4}+|y_j|)(\min(1,|y_j|^{-b}))
\end{split}
\end{equation*}
with $\alpha'=2(3r+4s+2)+(r+s)(a+b+2c'+4(r+2s))+2(7r+18s)$.
\end{coro}

\section{A Burgess type subconvex bound for twisted $\GLone{2}$ $L$-functions}\label{chap:burgesssubconvexity}

In this section, as an application of Theorem \ref{spectraldecompositionofshiftedconvolutionsum}, we prove Theorem \ref{subconvextheorem}. For totally real fields, this was proved by Blomer and Harcos in \cite{BlomerHarcos}. We also note that for arbitrary number fields, Wu \cite{Wu} recently proved this result, using a different method. Our approach is the extension of the one in \cite[Section 3.3]{BlomerHarcos}.

Assume that $\pi$ is an irreducible automorphic cuspidal representation. Let $\frak{q}\subseteq\frak{o}$ be an ideal, $\chi$ a Hecke character of conductor $\frak{q}$. We may also think of $\chi$ as a character on the group of fractional ideals coprime to $\frak{q}$, extended to be $0$ on other ideals. There exists characters $\chi_{\fin}$ of $(\rfact{\frak{o}}{\frak{q}})^{\times}$ and $\chi_{\infty}$ of $F_{\infty}^{\times}$ satisfying $\chi((t))=\chi_{\fin}(t)\chi_{\infty}(t)$ for all $t\in\frak{o}$ coprime to $\frak{q}$. The transition from one meaning to another of Hecke characters can be found at several places (see \cite[Sections 1.7 and 3.1]{Bump}, for example). Our goal is to estimate $L(1/2,\pi\otimes\chi)$ in terms of $\norm{\frak{q}}$. Fix any $\varepsilon>0$. From now on, the implicit constants in $\ll$ are always meant to depend on $F,\varepsilon,\pi,\chi_{\infty}$, even if it is not emphasized in the subscript like $\ll_{F,\varepsilon,\pi,\chi_{\infty}}$. Fix an ideal $\frak{n}$ coprime to $\frak{q}$ satisfying 
\begin{equation}\label{magnitudeoffrakn}
\norm{\frak{n}}\ll \norm{\frak{q}}^{\varepsilon}
\end{equation}
and note that in every narrow ideal class, there is a representative $\frak{n}$ with these properties.

First we introduce the following notation. For given positive real numbers $a<b$,
\begin{equation}\label{definitionofinterval}
[[a,b]]=\{x\in F_{\infty,+}^{\times}:a\leq |x_j|\leq b\}.
\end{equation}

Let $G_0$ be a smooth, compactly supported function on $F_{\infty,+}^1=\{x\in F_{\infty,+}^{\times}:|x|_{\infty}=1\}$ satisfying that $\sum_{u\in\frak{o}^{\times}_+}G_0(ux)=1$ for all $x\in F_{\infty,+}^1$. We extend this function to $F_{\infty,+}^{\times}$ as $G(x)=G_0(x/|x|_{\infty})$, then $\sum_{u\in\frak{o}^{\times}_+}G(ux)=1$ for all $x\in F_{\infty,+}^{\times}$. Assume that $G_0$ is supported on $[[c_1,c_2]]$, then $G$ is supported on $F_{\infty,+}^{\diag}[[c_1,c_2]]$, where $c_1,c_2$ are constants depending only on $F$ (recall (\ref{definitionofinterval})). Fix moreover a compact fundamental domain $\mathcal{G}_0$ for the action of $\frak{o}^{\times}_+$ on $F_{\infty,+}^1$ and let $\mathcal{G}=F_{\infty,+}^{\diag}\mathcal{G}_0$ be its extension to $F_{\infty,+}^{\times}$.

\subsection{The amplification method}

Let $\xi$ be a character of $(\rfact{\frak{o}}{\frak{q}})^{\times}$. Parametrized by $v=(v_1,\ldots,v_{r+s})\in (i\RR)^{r+s},p=(p_{r+1},\ldots,p_{r+s})\in\ZZ^{s}$, assume that $W_{v,p}$ are functions on $F_{\infty,+}^{\times}$ satisfying the following properties:
\begin{enumerate}[(i)]
\begin{item}\label{cond1}
$W_{v,p}$ is smooth and supported on $[[c_3,c_4]]$ for some $c_3<c_1$ and $c_4>c_2$ depending only on $F$;
\end{item}
\begin{item}\label{cond2} for any differential operator $\mathcal{D}$ of the form
\begin{equation*}
\mathcal{D}=\left(\left(\frac{\partial}{\partial y_j}\right)^{\mu_j}_{j\leq r} \left(\frac{\partial}{\partial y_j}\right)^{\mu_{j,1}}_{j>r} \left(\frac{\partial}{\partial \overline{y_j}}\right)^{\mu_{j,2}}_{j>r}\right),
\end{equation*}
with nonnegative integers $\mu_{j,*}$, we have
\begin{equation*}
\mathcal{D}W_{v,p}(y)\ll_{\mathcal{D}}\prod_{j=1}^r(1+|v_j|)^{\mu_j} \prod_{j=r+1}^{r+s}(1+|v_j|+|p_j|)^{\mu_{j,1}+\mu_{j,2}}.
\end{equation*}
Compare this with (\ref{definitionofspectralparameter}) and (\ref{definitionofnormofspectralparameter}), and for convenience, introduce
\begin{equation}\label{definitionofparameternorm}
\norm{v,p}=\prod_{j=1}^r(1+|v_j|) \prod_{j=r+1}^{r+s}(1+|v_j|+|p_j|)^2.
\end{equation}
\end{item}
\end{enumerate}

Then set
\begin{equation}\label{definitionofbigl}
\LL_{\xi}(v,p)=\sum_{0<<t\in\frak{n}} \frac{\lambda_{\pi}(t\frak{n}^{-1})\xi(t)}{\sqrt{\norm{t\frak{n}^{-1}}}} W_{v,p}\left(\frac{t}{Y^{1/(r+2s)}}\right),
\end{equation}
where $0<<t$ means that we sum over the totally positive elements. The only assumption on the positive real number $Y$ is that
\begin{equation}\label{magnitudeofy}
Y\ll \norm{\frak{q}}^{1+\varepsilon}.
\end{equation}
Introduce $\mathcal{K}=\frak{n}\cap F_{\infty,+}^{\diag}[[c_3,c_4]]$. We see that the numbers $t$ that give a positive contribution are all in the set $\frak{n}\cap\mathcal{K}$ and also satisfy $t\in[[c_3,c_4]]Y^{1/(r+2s)}$, this latter implies $|t|_{\infty}\asymp_F Y$.

Assume $L$ (the amplification length) is a further parameter satisfying
\begin{equation}\label{magnitudeofl}
\log L\asymp \log\norm{\frak{q}}.
\end{equation}

\begin{lemm} Denote by $\Pi_{\frak{q},+}(L,2L)$ the set of totally positive, principal prime ideals $\frak{l}\subseteq\frak{o}$ satisfying $\norm{\frak{l}}\in [L,2L]$ and $\frak{l}\nmid\frak{q}$. Set $\pi_{\frak{q},+}(L,2L)=\#\Pi_{\frak{q},+}(L,2L)$. Then
\begin{equation*}\pi_{\frak{q},+}(L,2L) \gg L\norm{\frak{q}}^{-\varepsilon}.
\end{equation*}
\end{lemm}
\begin{proof} This follows immediately from the results \cite[Corollary 6 of Proposition 7.8 and Proposition 7.9(ii)]{Narkiewicz} about the natural density of prime ideals in narrow ideal classes. (See also \cite[Chapter VII, \textsection 13]{Neukirch} for analogous statements about the Dirichlet density.)
\end{proof}

Therefore,
\begin{equation*}
\begin{split}
|\LL_{\chi_{\fin}}(v,p)|^2&= \frac{1}{\pi_{\frak{q},+}(L,2L)^2} \left|\LL_{\chi_{\fin}}(v,p)\sum_{\substack{l\in\frak{o}\cap\mathcal{G}\\ (l)\in\Pi_{\frak{q},+}(L,2L)}} 1\right|^2\\ &\ll \frac{\norm{\frak{q}}^{\varepsilon}}{L^2} \sum_{\xi\in\widehat{(\rfact{\frak{o}}{\frak{q}})^{\times}}} \left|\LL_{\xi}(v,p)\sum_{\substack{l\in\frak{o}\cap\mathcal{G}\\ (l)\in\Pi_{\frak{q},+}(L,2L)}} \xi(l)\overline{\chi_{\fin}(l)}\right|^2.
\end{split}
\end{equation*}
Observe that the $\xi$-sum is the square integral of the Fourier transform of the function
\begin{equation*}
(\rfact{\frak{o}}{\frak{q}})^{\times}\ni x\mapsto \sum_{t\in\frak{n}\cap\mathcal{K}} \sum_{\substack{l\in\frak{o}\cap\mathcal{G}\\ (l)\in\Pi_{\frak{q},+}(L,2L)\\ lt\equiv x \pmod{\frak{q}}}} \overline{\chi_{\fin}(l)} \frac{\lambda_{\pi}(t\frak{n}^{-1})}{\sqrt{\norm{t\frak{n}^{-1}}}} W_{v,p}\left(\frac{t}{Y^{1/(r+2s)}}\right),
\end{equation*}
so Plancherel gives
\begin{equation*}
\begin{split}
|\LL_{\chi_{\fin}}(v,p)|^2&\ll \frac{\varphi(\frak{q})\norm{\frak{q}}^{\varepsilon}}{L^2} \\ & \cdot \sum_{x\in(\rfact{\frak{o}}{\frak{q}})^{\times}} \left|\sum_{\substack{l\in\frak{o}\cap\mathcal{G}\\ (l)\in\Pi_{\frak{q},+}(L,2L)}} \overline{\chi_{\fin}(l)} \sum_{\substack{t\in\frak{n}\cap\mathcal{K}\\ lt\equiv x \pmod{\frak{q}}}} \frac{\lambda_{\pi}(t\frak{n}^{-1})}{\sqrt{\norm{t\frak{n}^{-1}}}} W_{v,p}\left(\frac{t}{Y^{1/(r+2s)}}\right) \right|^2.
\end{split}
\end{equation*}
This can be further majorized using $\varphi(\frak{q})\leq\norm{\frak{q}}$ and $(\rfact{\frak{o}}{\frak{q}})^{\times}\subset\rfact{\frak{o}}{\frak{q}}$, giving
\begin{equation}\label{amplifiedsecondmoment}
\begin{split}
|\LL_{\chi_{\fin}}(v,p)|^2&\ll \frac{\norm{\frak{q}}^{1+\varepsilon}}{L^2}  \sum_{\substack{l_1,l_2\in\frak{o}\cap\mathcal{G}\\ (l_1),(l_2)\in\Pi_{\frak{q},+}(L,2L)}} \overline{\chi_{\fin}(l_1)}\chi_{\fin}(l_2)\\ &\sum_{\substack{t_1,t_2\in\frak{n}\cap\mathcal{K}\\ l_1t_1-l_2t_2\in\frak{q}}} \frac{\lambda_{\pi}(t_1\frak{n}^{-1}) \overline{\lambda_{\pi}(t_2\frak{n}^{-1})}}{\sqrt{\norm{t_1t_2\frak{n}^{-2}}}} W_{v,p}\left(\frac{t_1}{Y^{1/(r+2s)}}\right) \overline{W_{v,p}\left(\frac{t_2}{Y^{1/(r+2s)}}\right)}.
\end{split}
\end{equation}
In (\ref{amplifiedsecondmoment}), the contribution of $l_1t_1-l_2t_2=0$ will be referred as the diagonal contribution $DC$, and that of $l_1t_1-l_2t_2\neq 0$ as the off-diagonal contribution $ODC$.
We will estimate them separately, optimize in the choice of the parameter $L$ (taking care about (\ref{magnitudeofl})), which will give rise to an estimate of $\LL_{\chi_{\fin}}(v,p)$. Using Mellin inversion, this bound on $\LL_{\chi_{\fin}}(v,p)$ (with implicit parameters satisfying (\ref{magnitudeoffrakn}) and (\ref{magnitudeofy})) will give rise to a Burgess type subconvex bound on $L(1/2,\pi\otimes\chi)$.

\subsection{Estimate of the diagonal contribution}

First focus on $DC$. Then by Cauchy-Schwarz,
\begin{equation*}
DC \ll \frac{\norm{\frak{q}}^{1+\varepsilon}}{L^2}\sum_{\substack{l\in\frak{o}\cap\mathcal{G}\\ (l)\in\Pi_{\frak{q},+}(L,2L)}} \sum_{\substack{t\in\frak{n}\cap\mathcal{K}\\ |t|_{\infty}\asymp_F Y}} \frac{|\lambda_{\pi}(t\frak{n}^{-1})|^2}{\norm{t\frak{n}^{-1}}} |\{(l',t')\in(\frak{o}\cap\mathcal{G})\times(\frak{n}\cap\mathcal{K}):l't'=lt\}|.
\end{equation*}
Here, $|\{(l',t')\in(\frak{o}\cap\mathcal{G})\times(\frak{n}\cap\mathcal{K}):l't'=lt\}|$ is at most the number of divisors of $(lt)$, which is $\ll \norm{(lt)}^{\varepsilon}\ll (LY)^{\varepsilon}$. By (\ref{boundonweightedsumofsquaresofheckeeigenvalues}), (\ref{magnitudeoffrakn}) and (\ref{magnitudeofy}), we see
\begin{equation*}
\sum_{\substack{t\in\frak{n}\cap\mathcal{K}\\ |t|_{\infty}\asymp_FY}} \frac{|\lambda_{\pi}(t\frak{n}^{-1})|^2}{\norm{t\frak{n}^{-1}}}\ll \norm{\frak{q}}^{\varepsilon},
\end{equation*}
and estimate the number of prime ideals $(l)$ trivially by $\ll L$. Altogether,
\begin{equation}\label{diagonalcontributionestimate}
DC \ll \frac{\norm{\frak{q}}^{1+\varepsilon}}{L}.
\end{equation}

\subsection{Off-diagonal contribution: spectral decomposition and Eisenstein part}

\paragraph{Spectral decomposition}

The estimate of the off-diagonal contribution requires much more work. Assume $\mathcal{G}_0$ is supported on $[[c_5,c_6]]$ for some constants $c_5,c_6$ depending only on $F$. Then only $l_1,l_2\in[[c_5L^{1/(r+2s)},c_6L^{1/(r+2s)}]]$ and $t_1,t_2\in[[c_3Y^{1/(r+2s)},c_4Y^{1/(r+2s)}]]$ have nonzero contribution. If $l_1,l_2,t_1,t_2$ satisfy these constraints, then
\begin{equation*}
l_1t_1-l_2t_2\in\mathcal{B}=\{x\in F_{\infty}: |x_j|\leq c_7(LY)^{1/(r+2s)}\}
\end{equation*}
with $c_7=2c_4c_6$. Now a term in $ODC$ corresponding to some fixed $l_1,l_2$ can be written as
\begin{equation}\label{offdiagonalcontribution2}
\sum_{\substack{q\in\frak{qn}\cap\mathcal{B}\\ q\neq 0}} \sum_{\substack{l_1t_1-l_2t_2=q\\ 0\neq t_1,t_2\in\frak{n}}} \frac{\lambda_{\pi}(t_1\frak{n}^{-1}) \overline{\lambda_{\pi}(t_2\frak{n}^{-1})}}{\sqrt{\norm{t_1t_2\frak{n}^{-2}}}} W_1\left(\frac{l_1t_1}{(LY)^{1/(r+2s)}};v,p\right) \overline{W_2\left(\frac{l_2t_2}{(LY)^{1/(r+2s)}};v,p\right)},
\end{equation}
where $W_1,W_2$ are smooth functions on $F_{\infty,+}^{\times}$ defined as
\begin{equation*}
W_1(y;v,p)=W_{v,p}(yL^{1/(r+2s)}/l_1),\qquad W_2(y;v,p)=W_{v,p}(yL^{1/(r+2s)}/l_2).
\end{equation*}
Now by the assumptions made on $W_{v,p}$ and $l_1,l_2$, we have that $W_1,W_2$ are smooth of compact support $[[c_8,c_9]]$ (where $c_8,c_9$ depend on $F$) and for any differential operator $\mathcal{D}$ of the form
\begin{equation*}
\mathcal{D}=\left(\left(\frac{\partial}{\partial y_j}\right)^{\mu_j}_{j\leq r} \left(\frac{\partial}{\partial y_j}\right)^{\mu_{j,1}}_{j>r} \left(\frac{\partial}{\partial \overline{y_j}}\right)^{\mu_{j,2}}_{j>r}\right),
\end{equation*}
with nonnegative integers $\mu_{j,*}$, we have
\begin{equation}\label{assumptionalboundonderivatives}
\mathcal{D}W_{1,2}(y;v,p)\ll_{\mathcal{D}}\norm{v,p}^{\mu},
\end{equation}
where $\mu=\max_j(\mu_{j,*})$ (recall (\ref{definitionofparameternorm})).

Now by Theorem \ref{spectraldecompositionofshiftedconvolutionsum}, (\ref{offdiagonalcontribution2}) can be rewritten as
\begin{equation}\label{offdiagonalcontributiondecomposed}
\sum_{0\neq q\in\frak{qn}\cap\mathcal{B}} \int_{(\frak{c})} \sum_{\frak{t}|\frak{cc}_{\varpi}^{-1}} \frac{\lambda_{\varpi}^{\frak{t}}(q\frak{n}^{-1})}{ \sqrt{\norm{q\frak{n}^{-1}}}} W_{\varpi,\frak{t}}\left(\frac{q}{(LY)^{1/(r+2s)}};v,p\right) d\varpi,
\end{equation}
where $\frak{c}=\frak{c}_{\pi}\mathrm{lcm}((l_1),(l_2))$.

\paragraph{Eisenstein spectrum} \label{subsec:chap:burgesssubconvexity_offdiagonalandisenstein_eisenstein}

First we estimate the contribution of the Eisenstein spectrum to (\ref{offdiagonalcontributiondecomposed}). We use Corollary \ref{spectraldecompositionofshiftedconvolutionsumcorollaryeisenstein} with $\mathcal{D}=1,a=c'=0,b=2$. The largest square divisor of $\mathrm{lcm}((l_1),(l_2))$ is $\frak{o}$, hence
\begin{equation*}
\int_{\mathcal{E}_{1}(\frak{c})} \sum_{\frak{t}|\frak{cc}_{\varpi}^{-1}} |W_{\varpi,\frak{t}}(y;v,p)|\ll \norm{(l_1l_2)}^{\varepsilon} ||W_1||_{S_{\alpha_1}} ||W_2||_{S_{\alpha_1}}
\end{equation*}
with some positive integer $\alpha_1$ depending only on $F$, uniformly in $y,v,p$. Moreover, by \cite[Lemma 8.4]{Venkateshsparse} and (\ref{assumptionalboundonderivatives}), for any positive $\alpha$,
\begin{equation}\label{changefromsobolevtoparameternorm}
||W_{1,2}||_{S^{\alpha}} \ll_{\alpha} \norm{v,p}^{2\alpha}
\end{equation}
giving
\begin{equation*}
\int_{\mathcal{E}_{1}(\frak{c})} \sum_{\frak{t}|\frak{cc}_{\varpi}^{-1}} |W_{\varpi,\frak{t}}(y;v,p)|\ll \norm{(l_1l_2)}^{\varepsilon} \norm{v,p}^{4\alpha_1}.
\end{equation*}

Taking into account (\ref{boundonheckeeigenvaluesineisensteinspectrum}), (\ref{magnitudeoffrakn}), (\ref{magnitudeofy}) and (\ref{magnitudeofl}), we see that the contribution of the Eisenstein spectrum to (\ref{offdiagonalcontributiondecomposed}) is
\begin{equation*}
\ll \norm{v,p}^{4\alpha_1} \norm{\frak{q}}^{\varepsilon}\sum_{0\neq q\in\frak{qn}\cap\mathcal{B}} \frac{\norm{\gcd(\frak{c},(q))}}{\sqrt{\norm{(q)}}}.
\end{equation*}
In the sum, each ideal $(q)$ appears with multiplicity $\ll \norm{\frak{q}}^{\varepsilon}$. Indeed, each ideal $(q)\subseteq\frak{o}$ has a generator $q$ satisfying $|q_j|\geq c_5$ at each archimedean place. Hence the possible units $\epsilon$ for which $q\epsilon\in\mathcal{B}$ all satisfy $|\epsilon_j|\leq c_{10}(LY)^{1/(r+2s)}$ at each place, for some constant $c_{10}$ depending only on $F$. The number of such units is $\ll \log(\norm{\frak{q}})^{r+s-1}$ by (\ref{magnitudeofy}) and (\ref{magnitudeofl}). Then the above display is
\begin{equation*}
\ll \norm{v,p}^{4\alpha_1} \norm{\frak{q}}^{2\varepsilon} \sum_{\substack{0\neq (q)\subseteq \frak{qn}\\ \norm{(q)}\ll LY}} \frac{\norm{\gcd(\frak{c},(q))}}{\sqrt{\norm{(q)}}}.
\end{equation*}
Here, the sum is $\ll \norm{\frak{q}}^{-1+\varepsilon}(LY)^{1/2}$, since $\gcd(\frak{c},(q))=\gcd(\frak{c}_{\pi},(q))$, which has norm $O_{F,\pi}(1)$.

Altogether, using again (\ref{magnitudeofy}), in (\ref{offdiagonalcontributiondecomposed}), the Eisenstein spectrum has contribution
\begin{equation}\label{eisensteinspectrumcontribution}
\ll \norm{v,p}^{4\alpha_1} \norm{\frak{q}}^{-1/2+\varepsilon}L^{1/2},
\end{equation}
which is analogous to \cite[(116)]{BlomerHarcos}.

\subsection{Off-diagonal contribution: cuspidal spectrum}

Set
\begin{equation*}
\mathcal{C}_{1}(\frak{c},\varepsilon)= \{\varpi\in\mathcal{C}_{1}(\frak{c}): \norm{\rr_{\varpi}} \leq\norm{\frak{q}}^\varepsilon\}.
\end{equation*}
Later we will prove that the contribution of representations outside $\mathcal{C}_{1}(\frak{c},\varepsilon)$ is small. So restrict to $\mathcal{C}_{1}(\frak{c},\varepsilon)$, and fix also the sign of $q$ as follows. For any sign $\xi\in\{\pm 1\}^r$,
set
\begin{equation*}
\mathcal{B}(\xi)=\{y\in\mathcal{B}: \sign(y)=\xi\}.
\end{equation*}
Then focus on the quantity
\begin{equation}\label{cuspidalessentialcontribution}
\sum_{q\in\frak{qn}\cap \mathcal{B}(\xi)} \sum_{\varpi\in\mathcal{C}_{1}(\frak{c},\varepsilon)} \sum_{\frak{t}|\frak{cc}_{\varpi}^{-1}} \frac{\lambda_{\varpi}^{\frak{t}}(q\frak{n}^{-1})}{ \sqrt{\norm{q\frak{n}^{-1}}}} W_{\varpi,\frak{t}} \left(\frac{q}{(LY)^{1/(r+2s)}};v,p\right).
\end{equation}

We follow again \cite{BlomerHarcos}. Consider the Mellin transform
\begin{equation}\label{mellintransform}
\hat{W}_{\varpi,\frak{t}}^{\xi}(v',p';v,p)= \int_{F_{\infty,+}^{\times}} W_{\varpi,\frak{t}}(\xi y;v,p) \prod_{j=1}^{r+s}|y_j|^{v'_j} \prod_{j=r+1}^{r+s} \left(\frac{y_j}{|y_j|}\right)^{p'_j} d_{\infty}^{\times}y.
\end{equation}
We would like to invert this. As for $p'$, observe that $W_{\varpi,\frak{t}}(y;v,p)$ is continuous on the set where each $|y_j|$ is fixed (which is the product of $s$ circles), so the standard Fourier analysis of the circle group is applicable. Also from Corollary \ref{spectraldecompositionofshiftedconvolutionsumcorollarycuspidal}, we see that the set $(i\RR)^{r+s}$ (which is the product of $r+s$ lines) can be used for Mellin inversion (see \cite[17.41]{GR}). Therefore, (\ref{cuspidalessentialcontribution}) is
\begin{equation*}
\begin{split}
&\ll \sum_{p'\in\ZZ^s} \int_{(i\RR)^{r+s}}  (LY)^{(v'_1+\ldots+v'_{r+s})/(r+2s)} \\ & \cdot \sum_{\varpi\in\mathcal{C}_{1}(\frak{c},\varepsilon)} \sum_{\frak{t}|\frak{cc}_{\varpi}^{-1}} \left(\hat{W}_{\varpi,\frak{t}}^{\xi}(v',p';v,p) \sum_{q\in\frak{qn}\cap\mathcal{B}(\xi)} \frac{\lambda_{\varpi}^{\frak{t}}(q\frak{n}^{-1})}{ \sqrt{\norm{q\frak{n}^{-1}}}} \prod_{j=1}^{r+s}|q_j|^{-v'_j} \prod_{j=r+1}^{r+s}\left(\frac{q_j}{|q_j|}\right)^{-p'_j}\right) dv'_j.
\end{split}
\end{equation*}
By Cauchy-Schwarz, this is
\begin{equation}\label{cuspidalspectrumbeforesplittingmellinandarithmetic}
\begin{split}
&\ll \sum_{p'\in\ZZ^s} \int_{(i\RR)^{r+s}} \left(\sum_{\varpi\in\mathcal{C}_{1}(\frak{c},\varepsilon)} \sum_{\frak{t}|\frak{cc}_{\varpi}^{-1}} \left|\hat{W}_{\varpi,\frak{t}}^{\xi}(v',p';v,p)\right|^2\right)^{1/2} \\ & \cdot \left(\sum_{\varpi\in\mathcal{C}_{1}(\frak{c},\varepsilon)} \sum_{\frak{t}|\frak{cc}_{\varpi}^{-1}} \left|\sum_{q\in\frak{qn}\cap\mathcal{B}(\xi)} \frac{\lambda_{\varpi}^{\frak{t}}(q\frak{n}^{-1})}{ \sqrt{\norm{q\frak{n}^{-1}}}} \prod_{j=1}^{r+s}|q_j|^{-v'_j} \prod_{j=r+1}^{r+s}\left(\frac{q_j}{|q_j|}\right)^{-p'_j}\right|^2 \right)^{1/2} |dv'_j|.
\end{split}
\end{equation}
In what follows, we estimate the Mellin part
\begin{equation}\label{mellinpart}
\left(\sum_{\varpi\in\mathcal{C}_{1}(\frak{c},\varepsilon)} \sum_{\frak{t}|\frak{cc}_{\varpi}^{-1}} \left|\hat{W}_{\varpi,\frak{t}}^{\xi}(v',p';v,p)\right|^2\right)^{1/2}
\end{equation}
and the arithmetic part
\begin{equation}\label{arithmeticpart}
\left(\sum_{\varpi\in\mathcal{C}_{1}(\frak{c},\varepsilon)} \sum_{\frak{t}|\frak{cc}_{\varpi}^{-1}} \left|\sum_{q\in\frak{qn}\cap\mathcal{B}(\xi)} \frac{\lambda_{\varpi}^{\frak{t}}(q\frak{n}^{-1})}{ \sqrt{\norm{q\frak{n}^{-1}}}} \prod_{j=1}^{r+s}|q_j|^{-v'_j} \prod_{j=r+1}^{r+s}\left(\frac{q_j}{|q_j|}\right)^{-p'_j}\right|^2 \right)^{1/2}
\end{equation}
separately.

\paragraph{Estimate of the Mellin part}

Recall the definition (\ref{mellintransform}) of the Mellin transform. Our plan is to insert differentiations (using that $W$'s are highly differentiable) to show that the Mellin part decays fast in terms of $\norm{v',p'}$.

At real places ($j\leq r$), for $v'_j\neq 0$,
\begin{equation*}
\int_{\RR_+^{\times}} W(y_j)y_j^{v'_j}d_{\RR}^{\times}y_j= -\frac{1}{v'_j} \int_{\RR_+^{\times}} y_j\frac{\partial}{\partial y_j}W(y_j)y_j^{v'_j}d_{\RR}^{\times}y_j,
\end{equation*}
so at those real places, where $|v'_j|\geq 1$, we can gain a factor $|v'_j|^{-1}$ using the differential operator $y_j(\partial/\partial y_j)$. The complex places ($j>r$) can be handled similarly. For $v'_j\neq 0$,
\begin{equation*}
\int_{\CC^{\times}} W(y_j)|y_j|^{v'_j}\left(\frac{y_j}{|y_j|}\right)^{p'_j}d_{\CC}^{\times}y_j= -\frac{1}{v'_j}\int_{\CC^{\times}} |y_j|\frac{\partial}{\partial|y_j|} W(y_j)|y_j|^{v'_j}\left(\frac{y_j}{|y_j|}\right)^{p'_j}d_{\CC}^{\times}y_j,
\end{equation*}
while for $p'_j\neq 0$,
\begin{equation*}
\int_{\CC^{\times}} W(y_j)|y_j|^{v'_j}\left(\frac{y_j}{|y_j|}\right)^{p'_j}d_{\CC}^{\times}y_j= -\frac{1}{ip'_j}\int_{\CC^{\times}} \frac{\partial}{\partial (y_j/|y_j|)} W(y_j)|y_j|^{v'_j}\left(\frac{y_j}{|y_j|}\right)^{p'_j}d_{\CC}^{\times}y_j.
\end{equation*}
This means that at those complex places, where $|v'_j|\geq 1$ (or $|p'_j|\geq 1$, respectively), we can gain a factor $|v'_j|^{-1}$ (or $|p'_j|^{-1}$, respectively), by inserting the differential operator $y(\partial /\partial y)$ (or $\partial/\partial (y/|y|)$, respectively).

A simple calculation shows that for any real-differentiable complex function $f(z)$ with $z=re^{i\theta}$ ($r>0$, $\theta\in[0,2\pi]$), both $r\partial f/\partial r$ and $\partial f/\partial \theta$ are $\ll |z\partial f/\partial z| + |\overline{z}\partial f/\partial \overline{z}|$.

Therefore, set the differential operators
\begin{equation*}
\mathcal{D}_{(e,f,g)}=\left(\left(\left(y_j\frac{\partial}{\partial y_j}\right)^{e_j}\right)_{j\leq r}, \left(\left(y_j\frac{\partial}{\partial y_j}\right)^{f_j}\right)_{j>r}, \left(\left(\overline{y_j}\frac{\partial}{\partial \overline{y_j}}\right)^{g_j}\right)_{j>r} \right),
\end{equation*}
where $0\leq e_j\leq 3$ ($j\leq r$), $0\leq f_j\leq 6$, $0\leq g_j\leq 6$ ($j>r$).
Then the above argument, together with (\ref{mellintransform}) and Cauchy-Schwarz, implies that (\ref{mellinpart}) is
\begin{equation*}
\begin{split}
\ll (\norm{v',p'})^{-3/2} \sum_{(e,f,g)} & \left( \int_{F_{\infty,+}^{\times}} \int_{F_{\infty,+}^{\times}} \left( \sum_{\varpi\in\mathcal{C}_{1}(\frak{c},\varepsilon)} \sum_{\frak{t}|\frak{cc}_{\varpi}^{-1}} |\mathcal{D}_{(e,f,g)}W_{\varpi,\frak{t}}(y;v,p)|^2 \right)^{1/2} \right. \\ & \qquad \qquad \left. \left( \sum_{\varpi\in\mathcal{C}_{1}(\frak{c},\varepsilon)} \sum_{\frak{t}|\frak{cc}_{\varpi}^{-1}} |W_{\varpi,\frak{t}}(y';v,p)|^2 \right)^{1/2} d_{\infty}^{\times}yd_{\infty}^{\times}y' \right)^{1/2}.
\end{split}
\end{equation*}
Now we apply Theorem \ref{spectraldecompositionofshiftedconvolutionsum}, Part \ref{spectraldecompositionofshiftedconvolutionsumpart2} with $a=6,b=2,c=0$ in the first sum, and with $a=0,b=2,c=0$ in the second sum. Together with (\ref{changefromsobolevtoparameternorm}), this implies that the integrand is
\begin{equation*}
\begin{split}
\ll \norm{\frak{q}}^{\varepsilon}\norm{v,p}^{4\alpha_2} &\prod_{j=1}^{r}\min(|y_j|^{1/4},|y_j|^{-3/2})\min(|y'_j|^{1/4},|y'_j|^{-3/2}) \\ &\cdot \prod_{j=r+1}^{r+s}\min(|y_j|^{3/4},|y_j|^{-1})\min(|y'_j|^{3/4},|y'_j|^{-1})
\end{split}
\end{equation*}
with some positive integer $\alpha_2$ depending only on $F$.
Altogether, the Mellin part (\ref{mellinpart}) is
\begin{equation}\label{estimateofmellinpart}
\ll \norm{\frak{q}}^{\varepsilon} \norm{v,p}^{4\alpha_2} \norm{v',p'}^{-3/2}.
\end{equation}

\paragraph{Estimate of the arithmetic part}

Our next goal is to give a bound on (\ref{arithmeticpart}), which is uniform in $v',p'$. Fix $v',p'$ and consider
\begin{equation}\label{squareofarithmeticpart}
\sum_{\varpi\in\mathcal{C}_{1}(\frak{c},\varepsilon)} \sum_{\frak{t}|\frak{cc}_{\varpi}^{-1}} \left|\sum_{q\in\frak{qn}\cap\mathcal{B}(\xi)} \frac{\lambda_{\varpi}^{\frak{t}}(q\frak{n}^{-1})}{ \sqrt{\norm{q\frak{n}^{-1}}}} \prod_{j=1}^{r+s}|q_j|^{-v'_j} \prod_{j=r+1}^{r+s}\left(\frac{q_j}{|q_j|}\right)^{-p'_j}\right|^2.
\end{equation}

Following \cite[p.45]{BlomerHarcos}, introduce, for any ideal $\frak{a}\subseteq\frak{o}$,
\begin{equation*}
f(\frak{a};v',p')= \sum_{\substack{q\in\mathcal{B}(\xi)\\ (q)=\frak{an}}} \prod_{j=1}^{r+s}|q_j|^{-v'_j} \prod_{j=r+1}^{r+s}\left(\frac{q_j}{|q_j|}\right)^{-p'_j}.
\end{equation*}
The number of possible units $\epsilon$ for which $q\epsilon\in\mathcal{B}$ is $\ll_{F,\varepsilon}\norm{\frak{q}}^{\varepsilon}$ (recall the argument in Section \ref{subsec:chap:burgesssubconvexity_offdiagonalandisenstein_eisenstein}), hence
\begin{equation}\label{qsumidealcollection}
|f(\frak{a};v',p')|\ll_{F,\varepsilon}\norm{\frak{q}}^{\varepsilon}.
\end{equation}
With this notation, we can rewrite the innermost sum in (\ref{squareofarithmeticpart}) as
\begin{equation*}
\sum_{q\in\frak{qn}\cap\mathcal{B}(\xi)} \frac{\lambda_{\varpi}^{\frak{t}}(q\frak{n}^{-1})}{ \sqrt{\norm{q\frak{n}^{-1}}}} \prod_{j=1}^{r+s}|q_j|^{-v'_j} \prod_{j=r+1}^{r+s}\left(\frac{q_j}{|q_j|}\right)^{-p'_j}= \sum_{\norm{\frak{m}}\ll LY/\norm{\frak{qn}}} \frac{\lambda_{\varpi}^{\frak{t}}(\frak{mq})}{ \sqrt{\norm{\frak{mq}}}} f(\frak{mq};v',p'),
\end{equation*}
where $\ll$ in the sum means that we may choose a constant depending only on $F$ such that this holds. Now on the right-hand side, for each occuring $\frak{m}$, transfer each prime factor dividing both $\frak{m}$ and $\frak{q}$ from $\frak{m}$ to $\frak{q}$. This does not affect the summand (since it depends only on the product $\frak{mq}$) and lets us write
\begin{equation}\label{qsumcoprimeform}
\sum_{\norm{\frak{m}}\ll LY/\norm{\frak{qn}}} \frac{\lambda_{\varpi}^{\frak{t}}(\frak{mq})}{\sqrt{\norm{\frak{mq}}}} f(\frak{mq};v',p') = \sum_{\frak{q}|\frak{q}'|\frak{q}^{\infty}}\sum_{\substack {\norm{\frak{m}}\ll LY/\norm{\frak{q}'\frak{n}}\\ \gcd(\frak{m},\frak{q})=\frak{o}}} \frac{\lambda_{\varpi}^{\frak{t}}(\frak{mq}')}{\sqrt{\norm{\frak{mq}'}}} f(\frak{mq}';v',p').
\end{equation}

For coprime ideals $\frak{m}$ and $\frak{q}'$, $\lambda_{\varpi}^{\frak{t}}(\frak{mq}')$ can be expressed as \begin{equation*}
\lambda_{\varpi}^{\frak{t}}(\frak{mq}')= \sum_{\frak{b}|\gcd(\frak{q}'\gcd(\frak{q}',\frak{t})^{-1}, \gcd(\frak{q}',\frak{t}))} \mu(\frak{b}) \lambda_{\varpi}(\frak{q}'\gcd(\frak{q}',\frak{t})^{-1}\frak{b}^{-1}) \lambda_{\varpi}^{\frak{t}}(\frak{m}\gcd(\frak{q}',\frak{t})\frak{b}^{-1}),
\end{equation*}
see \cite[Lemma 6.2]{MPphd} (which is based on \cite[pp.73-74]{BlomerHarcosMichel}).

Using this in (\ref{qsumcoprimeform}), by (\ref{towardsramanujanpetersson}), we see
\begin{equation*}
\lambda_{\varpi}(\frak{q}'\gcd(\frak{q}',\frak{t})^{-1}\frak{b}^{-1}) \ll \norm{\frak{q}'}^{\theta+\varepsilon}.
\end{equation*}
We claim $\gcd(\frak{q}',\frak{t})|\frak{c}_{\pi}$. Indeed, $\frak{t}|\frak{cc}_{\varpi}^{-1}$ with $\frak{c}=\frak{c}_{\pi}\mathrm{lcm}((l_1),(l_2))$, where $l_1,l_2$ are primes not dividing $\frak{q}$. Altogether, the $q$-sum in (\ref{squareofarithmeticpart}) can be estimated as
\begin{equation}\label{qsumestimate}
\ll \sum_{\frak{q}|\frak{q}'|\frak{q}^{\infty}} \norm{\frak{q}'}^{-1/2+\theta+\varepsilon} \sum_{\frak{b}|\frak{c}_{\pi}} \left| \sum_{\substack {\norm{\frak{m}}\ll LY/\norm{\frak{q}'\frak{n}}\\ \gcd(\frak{m},\frak{q})=\frak{o}}} \frac{\lambda_{\varpi}^{\frak{t}}(\frak{mb})}{\sqrt{\norm{\frak{m}}}} f(\frak{mq}';v',p') \right|.
\end{equation}

Similarly as in Lemma \ref{densityofcuspidalspectrumgeneral}, take the function $h$ defined in (\ref{kuznetsovtestfunctionreal}), (\ref{kuznetsovtestfunctioncomplex}) with $a_j=\norm{\frak{q}}^{2\varepsilon}$ at real, $a_j=\norm{\frak{q}}^{\varepsilon}$ at complex places, $b_j=\sqrt{a_j}$ at all archimedean places, finally $a_j'=-1$ at complex places. This has the property that it gives weight $\gg 1$ to representations in $\mathcal{C}_{1}(\frak{c},\varepsilon)$.
\begin{equation*}
\begin{split}
\sum_{\varpi\in\mathcal{C}_{1}(\frak{c},\varepsilon)} \sum_{\frak{t}|\frak{cc}_{\varpi}^{-1}} & \left|\sum_{q\in\frak{qn}\cap\mathcal{B}(\xi)} \frac{\lambda_{\varpi}^{\frak{t}}(q\frak{n}^{-1})}{ \sqrt{\norm{q\frak{n}^{-1}}}} \prod_{j=1}^{r+s}|q_j|^{-v'_j} \prod_{j=r+1}^{r+s}\left(\frac{q_j}{|q_j|}\right)^{-p'_j}\right|^2 \\ & \ll \sum_{\varpi\in\mathcal{C}_{1}(\frak{c},\varepsilon)} \sum_{\frak{t}|\frak{cc}_{\varpi}^{-1}} h(\rr_{\varpi})\left|\sum_{q\in\frak{qn}\cap\mathcal{B}(\xi)} \frac{\lambda_{\varpi}^{\frak{t}}(q\frak{n}^{-1})}{ \sqrt{\norm{q\frak{n}^{-1}}}} \prod_{j=1}^{r+s}|q_j|^{-v'_j} \prod_{j=r+1}^{r+s}\left(\frac{q_j}{|q_j|}\right)^{-p'_j}\right|^2.
\end{split}
\end{equation*}
In the summation over $\varpi$, multiply by a factor $C_{\varpi}^{-1}$, which is $\gg \norm{\frak{q}}^{-\varepsilon}$ by
(\ref{proportionalityconstantisessentiallylfunctionresidue}) and Proposition \ref{residueofrankinselbergsquarelfuncionisaboutone}. We also add the analogous nonnegative contribution of the Eisenstein spectrum.

Therefore, using (\ref{qsumidealcollection}), (\ref{qsumestimate}) estimates the $\varpi$-sum of (\ref{squareofarithmeticpart}) as
\begin{equation}\label{squareofarithmeticpartbeforekuznetsov}
\begin{split}
\sum_{\varpi\in\mathcal{C}_{1}(\frak{c},\varepsilon)} \sum_{\frak{t}|\frak{cc}_{\varpi}^{-1}} & \left|\sum_{q\in\frak{qn}\cap\mathcal{B}(\xi)} \frac{\lambda_{\varpi}^{\frak{t}}(q\frak{n}^{-1})}{ \sqrt{\norm{q\frak{n}^{-1}}}} \prod_{j=1}^{r+s}|q_j|^{-v'_j} \prod_{j=r+1}^{r+s}\left(\frac{q_j}{|q_j|}\right)^{-p'_j}\right|^2 \\ & \ll \norm{\frak{q}}^{-1+2\theta+\varepsilon} \max_{\frak{b}_1,\frak{b}_2|\frak{c}_{\pi}} \sum_{\norm{\frak{m}_1},\norm{\frak{m}_2}\ll LY/\norm{\frak{q}}} \frac{1}{\sqrt{\norm{\frak{m}_1\frak{m}_2}}} \\ & \qquad \qquad \qquad \qquad \left|\sum_{\varpi\in\mathcal{C}_{1}(\frak{c})} C_{\varpi}^{-1} \sum_{\frak{t}|\frak{cc}_{\varpi}^{-1}} h(\rr_{\varpi}) \lambda_{\varpi}^{\frak{t}}(\frak{m}_1\frak{b}_1) \overline{\lambda_{\varpi}^{\frak{t}}(\frak{m}_2\frak{b}_2)}+CSC \right|.
\end{split}
\end{equation}

We apply the Kuznetsov formula (\ref{kuznetsovformula}) to estimate the last line of (\ref{squareofarithmeticpartbeforekuznetsov}), with $\alpha=\alpha'=1$, $\frak{a}^{-1}=\frak{m}_1\frak{b}_1$, $\frak{a}'^{-1}=\frak{m}_2\frak{b}_2$. The delta term is, up to a constant multiple,
\begin{equation*}
[K(\frak{o}):K(\frak{c})]\Delta(\frak{m}_1\frak{b}_1,\frak{m}_2\frak{b}_2) \int h(\rr_{\varpi}) d\mu.
\end{equation*}
Here, by (\ref{estimateoftransformsongeometricside}), the integral of $h$ gives $\ll \norm{\frak{q}}^{2(r+s)\varepsilon}$, and also $[K(\frak{o}):K(\frak{c})]\ll L^2 \norm{\frak{q}}^{\varepsilon}$ by (\ref{indexofcongruencesubgroup}) and (\ref{magnitudeofl}). When $\Delta(\frak{m}_1\frak{b}_1,\frak{m}_2\frak{b}_2)\neq 0$, $\norm{\frak{m}_1}\asymp_{F,\pi}\norm{\frak{m}_2}$, so the sum over $\frak{m}_1,\frak{m}_2$ can be replaced by a sum over $\frak{m}$.
Using (\ref{magnitudeofy}), we see that $LY/\norm{\frak{q}}\ll \norm{\frak{q}}^{\varepsilon}L$, and taking into account also (\ref{magnitudeofl}), we obtain that
\begin{equation*}
\sum_{\norm{\frak{m}}\ll LY/\norm{\frak{q}}} \frac{1}{\norm{\frak{m}}}\ll \norm{\frak{q}}^{\varepsilon}.
\end{equation*}
Altogether, the delta term of the geometric side of the Kuznetsov formula (\ref{kuznetsovformula}) contributes
\begin{equation}\label{deltacontributiontoarithmeticpart}
\ll \norm{\frak{q}}^{-1+2\theta+\varepsilon}L^2
\end{equation}
to the right-hand side of (\ref{squareofarithmeticpartbeforekuznetsov}).

As for the Kloosterman term, similarly to (\ref{densityofcuspidalspectrumgeneralkloostermanterm}), we have to estimate
\begin{equation*}
\begin{split}
\max_{\frak{a}\in C}\sum_{\epsilon\in\rfact{\frak{o}_+^{\times}}{{\frak{o}^{2\times}}}}\sum_{0\neq c\in \frak{m}_1^{-1}\frak{b}_1^{-1}\frak{ac}} & \frac{\norm{(\gcd(\frak{m}_1\frak{b}_1,\frak{m}_2\frak{b}_2, c\frak{m}_1\frak{b}_1\frak{a}^{-1}))}^{1/2}}{ \norm{c\frak{m}_1\frak{b}_1\frak{a}^{-1}}^{1/2-\varepsilon}} \\ & \cdot \prod_{j\leq r} \min(1,|\epsilon_j\gamma_{\frak{a},j}/c_j|^{1/2}) \prod_{j>r} \min(1,|\epsilon_j\gamma_{\frak{a},j}/c_j|),
\end{split}
\end{equation*}
where $\gamma_{\frak{a}}$ is a totally positive generator of the ideal $\frak{a}^{2}(\frak{m}_1\frak{b}_1)^{-1}\frak{m}_2\frak{b}_2$, $C$ is a fixed set of narrow class representatives (depending only on $F$ and the narrow class of $(\frak{m}_1\frak{b}_1)^{-1}\frak{m}_2\frak{b}_2$) with the property that such a $\gamma_{\frak{a}}$ exists for each $\frak{a}\in C$. The sum over $\epsilon\in\rfact{\frak{o}_+^{\times}}{{\frak{o}^{2\times}}}$ is negligible. Now take a totally positive $\beta\in\frak{o}$ such that $(\beta)\supseteq\frak{m}_1\frak{b}_1$, $\norm{(\beta)}\gg \norm{\frak{m}_1\frak{b}_1}$, and then the above is
\begin{equation*}
\begin{split}
\ll \max_{\frak{a}\in C} \sum_{0\neq c\in \frak{ac}} & \frac{\norm{(\gcd(\frak{m}_1\frak{b}_1,\frak{m}_2\frak{b}_2, c\frak{a}^{-1}))}^{1/2}}{ \norm{c\frak{a}^{-1}}^{1/2-\varepsilon}} \\ & \cdot \prod_{j\leq r} \min(1,|\gamma_{\frak{a},j}\beta_j|^{1/4}/|c_j|^{1/2}) \prod_{j>r} \min(1,|\gamma_{\frak{a},j}\beta_j|^{1/2}/|c_j|),
\end{split}
\end{equation*}
Then the same method as in the proof of Lemma \ref{densityofcuspidalspectrumgeneral} shows that the previous display can be estimated as
\begin{equation*}
\ll \norm{(\gamma_{\frak{a}}\beta)}^{1/4+\varepsilon} \norm{\gcd(\frak{m}_1,\frak{m}_2,\frak{c})}^{1/2+\varepsilon} \norm{\frak{c}}^{-1-\varepsilon}.
\end{equation*}
The last factor $\norm{\frak{c}}^{-1-\varepsilon}$ cancels $[K(\frak{o}):K(\frak{c})]$. Noting that $\norm{(\gamma_{\frak{a}}\beta)}\ll \norm{\frak{m}_1\frak{m}_2}$, we see that the Kloosterman term contributes
\begin{equation*}
\ll \norm{\frak{q}}^{-1+2\theta+\varepsilon} \sum_{\norm{\frak{m}_1},\norm{\frak{m}_2}\ll LY/\norm{\frak{q}}} \norm{\gcd(\frak{m}_1,\frak{m}_2,\frak{c})}^{1/2+\varepsilon} \norm{\frak{m}_1\frak{m}_2}^{-1/4+\varepsilon}
\end{equation*}
to the right-hand side of (\ref{squareofarithmeticpartbeforekuznetsov}). Obviously
\begin{equation*}
\norm{\gcd(\frak{m}_1,\frak{m}_2,\frak{c})}^{1/2} \leq \norm{\gcd(\frak{m}_1,\frak{c})}^{1/4} \norm{\gcd(\frak{m}_2,\frak{c})}^{1/4},
\end{equation*}
so the above display is (using also (\ref{magnitudeofy}) and (\ref{magnitudeofl}) again)
\begin{equation*}
\ll \norm{\frak{q}}^{-1+2\theta+2\varepsilon} \left(\sum_{\norm{\frak{m}}\ll LY/\norm{\frak{q}}} \left(\frac{\norm{\gcd(\frak{m},\frak{c})}}{\norm{\frak{m}}} \right)^{1/4} \right)^2.
\end{equation*}
Here, if $\frak{m}$ is divisible by $l_1$ or $l_2$, then $\norm{\gcd(\frak{m},\frak{c})}\ll L$ (by (\ref{magnitudeofl}), an ideal of norm $\ll\norm{\frak{q}}^{\varepsilon}L$ cannot have two different prime divisors $l_1,l_2$), this happens at most for $\norm{\frak{q}}^{\varepsilon}$ many $\frak{m}$'s. For other $\frak{m}$'s, $\norm{\gcd(\frak{m},\frak{c})}\ll 1$. Therefore, the Kloosterman contribution to (\ref{squareofarithmeticpartbeforekuznetsov}) is
\begin{equation}\label{kloostermancontributiontoairthmeticpart}
\ll \norm{\frak{q}}^{-1+2\theta+\varepsilon}L^{3/2}.
\end{equation}

Taking square-roots, we obtain from (\ref{deltacontributiontoarithmeticpart}) and (\ref{kloostermancontributiontoairthmeticpart}) that the arithmetic part (\ref{arithmeticpart}) is
\begin{equation}\label{estimateofarithmeticpart}
\ll \norm{\frak{q}}^{-1/2+\theta+\varepsilon}L.
\end{equation}

\paragraph{Summing up in the cuspidal spectrum}

Inside $\mathcal{C}_{1}(\frak{c},\varepsilon)$, (\ref{cuspidalspectrumbeforesplittingmellinandarithmetic}), (\ref{estimateofmellinpart}) and (\ref{estimateofarithmeticpart}) show that the contribution (\ref{cuspidalessentialcontribution}) is
\begin{equation*}
\begin{split}
&\ll \sum_{p'\in\ZZ^s}\int_{(i\RR)^{r+s}} \norm{v,p}^{4\alpha_2} \norm{v',p'}^{-3/2} \norm{\frak{q}}^{-1/2+\theta+\varepsilon}L |dv'_j| \\ & \ll \norm{\frak{q}}^{-1/2+\theta+\varepsilon}L \norm{v,p}^{4\alpha_2},
\end{split}
\end{equation*}
and this bound holds (with the implicit constant multiplied by $2^r$) without restricting the summation in (\ref{cuspidalspectrumbeforesplittingmellinandarithmetic}) to a specific sign $\xi$.

Now we concentrate on representations outside $\mathcal{C}_{1}(\frak{c},\varepsilon)$. First of all, from Lemma \ref{densityofcuspidalspectrumgeneral}, we see that
\begin{equation*}
\lambda_{\varpi}^{\frak{t}}(q\frak{n}^{-1})\ll L^{1/2+\varepsilon} \norm{(q)}^{1/4+\varepsilon} \norm{\rr_{\varpi}},
\end{equation*}
therefore, with a large $c'$ (depending on $\varepsilon$), we may write (using (\ref{magnitudeofl})), outside $\mathcal{C}_{1}(\frak{c},\varepsilon)$,
\begin{equation*}
\qquad \frac{\lambda_{\varpi}^{\frak{t}}(q\frak{n}^{-1})}{ \sqrt{\norm{q\frak{n}^{-1}}}}\ll L^{1/2+\varepsilon}\norm{\frak{q}}^{-1/4}\norm{\rr_{\varpi}}\ll \norm{\rr_{\varpi}}^{c'}.
\end{equation*}
Now by Cauchy-Schwarz, outside $\mathcal{C}_{1}(\frak{c},\varepsilon)$, the cuspidal contribution is, with some $c$ much larger than $c'$,
\begin{equation*}
\begin{split}
\ll
& \left(\sum_{0\neq q\in\frak{qn}\cap \mathcal{B}} \sum_{\varpi\notin\mathcal{C}_{1}(\frak{c},\varepsilon) } \sum_{\frak{t}|\frak{cc}_{\varpi}^{-1}} \norm{\rr_{\varpi}}^{2(c'-c)}\right)^{1/2} \\ & \cdot
\left(\sum_{0\neq q\in\frak{qn}\cap \mathcal{B}} \sum_{\varpi\notin\mathcal{C}_{1}(\frak{c},\varepsilon)} \sum_{\frak{t}|\frak{cc}_{\varpi}^{-1}} \left|\norm{\rr_{\varpi}}^c W_{\varpi,\frak{t}} \left(\frac{q}{(LY)^{1/(r+2s)}};v,p\right)\right|^{2}\right)^{1/2}.
\end{split}
\end{equation*}
The first factor is $\ll_k \norm{\frak{q}}^{-k}$ for any $k\in\NN$, if $c-c'$ is large enough, as it follows from Corollary \ref{densityofcuspidalspectrum2}. As for the second factor, apply Theorem \ref{spectraldecompositionofshiftedconvolutionsum}, Part \ref{spectraldecompositionofshiftedconvolutionsumpart2} with $a=0$, $b=0$ and the above $c$. The number of $q$'s in $\frak{q}\frak{n}\cap\mathcal{B}$ is $O_F(LY)$. Then together with (\ref{changefromsobolevtoparameternorm}), we see that the second factor is $\ll \norm{\frak{q}}^{-1+\varepsilon}\norm{v,p}^{4\alpha_3}$ with some positive integer $\alpha_3$ depending only on $F$. To match $Y$ and $L$ with $\norm{\frak{q}}$, we use (\ref{magnitudeofy}) and (\ref{magnitudeofl}) throughout.

Altogether, the cuspidal spectrum has contribution
\begin{equation}\label{cuspidalspectrumcontribution}
\ll \norm{v,p}^{4\max(\alpha_2,\alpha_3)} \norm{\frak{q}}^{-1/2+\theta+\varepsilon}L.
\end{equation}

\subsection{Choice of the amplification length}

Set $\alpha=\max(\alpha_1,\alpha_2,\alpha_3)$. Summing trivially over $l_1,l_2$, and using (\ref{offdiagonalcontribution2}), (\ref{offdiagonalcontributiondecomposed}), (\ref{eisensteinspectrumcontribution}) and (\ref{cuspidalspectrumcontribution}), we see
\begin{equation*}
ODC\ll \norm{v,p}^{4\alpha} \norm{\frak{q}}^{1/2+\theta+\varepsilon}L.
\end{equation*}
This estimate, together with (\ref{diagonalcontributionestimate}) and through (\ref{amplifiedsecondmoment}), gives rise to
\begin{equation*}
\begin{split}
& |\LL_{\chi_{\fin}}(v,p)|^2\ll \norm{v,p}^{4\alpha} (\norm{\frak{q}}^{1+\varepsilon}L^{-1}+ \norm{\frak{q}}^{1/2+\theta+\varepsilon}L), \\ &|\LL_{\chi_{\fin}}(v,p)|\ll \norm{v,p}^{2\alpha} (\norm{\frak{q}}^{1/2+\varepsilon}L^{-1/2}+ \norm{\frak{q}}^{1/4+\theta/2+\varepsilon}L^{1/2}).
\end{split}
\end{equation*}
We see that the optimal choice is $L=\norm{\frak{q}}^{1/4-\theta/2}$, which meets the condition (\ref{magnitudeofl}). With this, we obtain the bound
\begin{equation}\label{estimateofbigl}
|\LL_{\chi_{\fin}}(v,p)|\ll \norm{v,p}^{2\alpha} \norm{\frak{q}}^{3/8+\theta/4+\varepsilon}.
\end{equation}

\subsection{Proof of Theorem \ref{subconvextheorem}}

In the derivation of the subconvex bound on $L(1/2,\pi\otimes\chi)$, the starting point is \cite[(75)]{BlomerHarcos}, a consequence of the approximate functional equation \cite[Theorem 2.1]{Harcos}: there is a constant $c=c(F,\pi,\chi_{\infty},\varepsilon)>0$ and a smooth function $V:(0,\infty)\rightarrow\CC$ supported on $[1/2,2]$, satisfying $V^{(j)}(y)\ll_{F,\pi,\chi_{\infty},j} 1$ for each nonnegative integer $j$, such that
\begin{equation}\label{approximatefunctionalequation}
L(1/2,\pi\otimes\chi)\ll_{F,\pi,\chi_{\infty},\varepsilon} \norm{\frak{q}}^{\varepsilon} \max_{Y\leq c\norm{\frak{q}}^{1+\varepsilon}} \left| \sum_{0\neq\frak{m}\subseteq\frak{o}} \frac{\lambda_{\pi}(\frak{m})\chi(\frak{m})}{\sqrt{\norm{\frak{m}}}} V\left(\frac{\norm{\frak{m}}}{Y}\right) \right|.
\end{equation}

\begin{proof}[Proof of Theorem \ref{subconvextheorem}] We start out from (\ref{approximatefunctionalequation}) as follows. First of all, we split up the sum over ideals according to their narrow class (with representatives $\frak{n}$ satisfying (\ref{magnitudeoffrakn})). Then
\begin{equation*}
L(1/2,\pi\otimes\chi)\ll_{F,\pi,\chi_{\infty},\varepsilon} \norm{\frak{q}}^{\varepsilon} \max_{Y\leq c\norm{\frak{q}}^{1+\varepsilon}} \left| \sum_{0<<t\in\frak{n}\ppmod{\frak{o}_+^{\times}}} \frac{\lambda_{\pi}(t\frak{n}^{-1})\chi(t\frak{n}^{-1})}{ \sqrt{\norm{t\frak{n}^{-1}}}} V\left(\frac{|t|_{\infty}}{Y}\right) \right|
\end{equation*}
for some $c=c(F,\pi,\chi_{\infty},\varepsilon)$, hence (\ref{magnitudeofy}) is satisfied. Here, by the partition of unity introduced in the beginning of this section, the sum on the right-hand side can be rewritten as
\begin{equation*}
\sum_{0<<t\in\frak{n}} \frac{\lambda_{\pi}(t\frak{n}^{-1})\chi(t\frak{n}^{-1})}{ \sqrt{\norm{t\frak{n}^{-1}}}} G(t_{\infty})V\left(\frac{|t|_{\infty}}{Y}\right)W\left( \frac{t_{\infty}}{Y^{1/(r+2s)}}\right),
\end{equation*}
where $W$ is a smooth nonnegative function which is $1$ on $[[c_1,c_2]]$ and supported on $[[c_3,c_4]]$. Now introducing the Mellin transform
\begin{equation*}
\hat{V}(v,p)=\int_{F_{\infty,+}^{\times}} G(y)V(y)\chi_{\infty}(y) \prod_{j=1}^{r+s}|y_j|^{v_j} \prod_{j=r+1}^{r+s} \left(\frac{y_j}{|y_j|}\right)^{p_j} d^{\times}y,
\end{equation*}
we have, by Mellin inversion, that the above display is
\begin{equation*}
\ll_{F} \sum_{p\in\ZZ^s} \int_{v\in(i\RR)^{r+s}} \hat{V}(v,p) \sum_{0<<t\in\frak{n}} \frac{\lambda_{\pi}(t\frak{n}^{-1}) \chi_{\fin}(t)}{\sqrt{\norm{t\frak{n}^{-1}}}} W_{v,p}\left(\frac{t}{Y^{1/(r+2s)}}\right)dv,
\end{equation*}
where
\begin{equation*}
W_{v,p}(y)=W(y)\prod_{j=1}^{r+s}|y_j|^{-v_j} \prod_{j=r+1}^{r+s} \left(\frac{y_j}{|y_j|}\right)^{-p_j} d^{\times}y.
\end{equation*}
Since $F(y)$, $V(y)$, $W(y)$ are all smooth and compactly supported, we see that
\begin{equation*}
\hat{V}(v,p)\ll_{F,\pi,\chi_{\infty},\varepsilon,\beta} \norm{v,p}^{-\beta}
\end{equation*}
for all $\beta\in\NN$ and also that the family of $W_{v,p}$'s satisfies \ref{cond1} and \ref{cond2}. Then
\begin{equation*}
L(1/2,\pi\otimes\chi)\ll_{F,\pi,\chi_{\infty},\varepsilon,\beta} \sum_{p\in\ZZ^s}\int_{v\in(i\RR)^{r+s}}\LL_{\chi_{\fin}}(v,p) \norm{v,p}^{-\beta} dv
\end{equation*}
with $\LL$ of (\ref{definitionofbigl}) satisfying all conditions we needed in its estimate. Now taking a $\beta$ which is much larger than $2\alpha$, (\ref{estimateofbigl}) completes the proof.
\end{proof}
\bibliography{magap_shifted_convolution_sums_and_burgess_type_subconvexity_over_number_fields}
\bibliographystyle{alpha}

\end{document}